\documentclass[onefignum,onetabnum]{siamart171218}
\usepackage{enumerate}
\usepackage{color}
\usepackage[normalem]{ulem}
\providecommand{\highlight}[1]{{\color{blue}#1}}

\usepackage{ifthen}
         \provideboolean{showchanges}
         \setboolean{showchanges}{false}
         \provideboolean{newchanges}
         \setboolean{newchanges}{false}
         \providecommand{\changes}[1]{
           \ifthenelse{\boolean{showchanges}}{{\highlight{#1}}}{#1}
         }
         \providecommand{\changefromto}[3][replace with]{
           \ifthenelse{\boolean{showchanges}}
           {{\sout{#2}\margnote{#1}}{\highlight{#3}}}
           {#3\xspace}
         }
         \providecommand{\ChangePar}[2]{
           \ifthenelse{\boolean{showchanges}}
           {{\par$\mapsfrom$ \textcolor{red!20}{#1}}{\par$\mapsto$ \textcolor{blue}{#2}}}
           {\par #2}
         }
         \providecommand{\InsertPar}[1]{
           \ifthenelse{\boolean{showchanges}}
           {{\par$\mapsto$ \textcolor{blue}{#1}}}
           {\par #1}
         }

\provideboolean{showdelete}
\setboolean{showdelete}{false}%
\newcommand{\delete}[1]{
  \ifthenelse{\boolean{showdelete}} {{\color{red}{#1}}}{}
}

\usepackage{lipsum}
\usepackage{amsfonts}
\usepackage{graphicx}
\usepackage{epstopdf}
\usepackage{algorithmic}
\ifpdf
  \DeclareGraphicsExtensions{.eps,.pdf,.png,.jpg}
\else
  \DeclareGraphicsExtensions{.eps}
\fi

\newcommand{\pdt}{\ensuremath{\partial_t}}

\newcommand{\pdtau}{\ensuremath{\bar{\partial}_\tau}}

\renewcommand{\O}{\ensuremath{{\Omega}}}
\newcommand{\G}{\ensuremath{{\Gamma}}}
\newcommand{\IR}{\mathbb R}
\newcommand{\ve}{\varepsilon}

\newcommand{\T}{\mathcal{T}}

\newcommand{\Lp}[1]{\ensuremath{{L}^{#1}}}

\newcommand{\Hil}[1]{\ensuremath{{H}^{#1}}}

\newcommand{\Ltwop}[3]{\ensuremath{\left\langle#1,#2\right\rangle}_{#3}}

\newcommand{\Th}{\ensuremath{{\mathcal{S}_{h,\O}}}}

\newcommand{\Shi}{\ensuremath{{\mathcal{S}_{h,i}}}}
\newcommand{\She}{\ensuremath{{\mathcal{S}_{h,e}}}}
\newcommand{\Vho}{\ensuremath{{\mathbb{V}_{h,\Omega}}}}
\newcommand{\vhee}{\ensuremath{{\mathbb{V}_{h,e}}}}
\newcommand{\vhi}{\ensuremath{{\mathbb{V}_{h,i}}}}
\newcommand{\vhe}{\ensuremath{{\mathbb{V}_{h,e}^\#}}}
\newcommand{\vhg}{\ensuremath{{\mathbb{V}_{h,\G}}}}
\newcommand{\Vhee}{\ensuremath{{\mathbb{W}_{h,e}}}}
\newcommand{\Vhi}{\ensuremath{{\mathbb{W}_{h,i}}}}
\newcommand{\Vhe}{\ensuremath{{\mathbb{W}_{h,e}^\#}}}
\newcommand{\Vhg}{\ensuremath{{\mathbb{W}_{h,\G}}}}
\newcommand{\myall}{\ensuremath{ \mbox{ for all }}}

\newcommand{\la}{\langle}
\newcommand{\ra}{\rangle}

\usepackage{subcaption}
\usepackage{tikz}

\newsiamremark{remark}{Remark}
\newsiamremark{hypothesis}{Hypothesis}
\newsiamremark{assumption}{Assumption}
\crefname{hypothesis}{Hypothesis}{Hypotheses}
\newsiamthm{claim}{Claim}

\headers{Multiscale analysis and simulation of a signalling process}{M. Ptashnyk, and  C. Venkataraman}

\title{Multiscale analysis and simulation of a signalling process with surface diffusion \thanks{Submitted to the editors -.
\funding{CV wishes to acknowledge the kind hospitality of the Hausdorff Institute for Mathematics in Bonn during the trimester program on multiscale problems in 2017.}}}

\author{Mariya Ptashnyk\thanks{Department of Mathematics, Heriot-Watt University, Edinburgh, UK 
  (\email{m.ptashnyk@hw.ac.uk}).}
\and Chandrasekhar Venkataraman\thanks{School of Mathematical and Physical Sciences, University of Sussex, UK
  (\email{cv42@sussex.ac.uk}).}
}

\usepackage{amsopn}

\ifpdf
\hypersetup{
  pdftitle={Multiscale analysis and simulation of a signalling process with surface diffusion},
  pdfauthor={M. Ptashnyk, and C. Venkataraman}
}
\fi

\begin{document}

\maketitle

\begin{abstract}
We present and analyse a model for cell signalling processes in biological tissues. The model includes  diffusion and nonlinear reactions on the cell surfaces, and both inter- and intracellular  signalling.
Using techniques from the theory of two-scale convergence as well the unfolding method,  we show convergence of the solutions to the model to solutions of a two-scale macroscopic problem.    We also present a two-scale bulk-surface finite element method for the approximation of the macroscopic model. We report on some benchmarking results as well as numerical simulations in a biologically relevant regime that illustrate the influence of cell-scale heterogeneities on macroscopic  concentrations.
\end{abstract}

\begin{keywords}
 Intercellular signalling, receptor-ligand interactions, homogenisation, nonlinear parabolic equations, surface diffusion, bulk-surface problems, surface finite elements
\end{keywords}

\begin{AMS}
  35B27, 35Kxx, 65M60
\end{AMS}

\section{Introduction} 
Interactions between cells and the response of cells to external stimuli   are  largely regulated by intracellular signalling processes which are themselves activated by interactions between cell membrane receptors and signalling molecules (ligands) diffusing  in the extracellular space.  Consequently, receptor-ligand interactions and the activation of intracellular signalling pathways are involved in many   important biological processes such as the immune response, cell movement and division, tissue development and  homeostasis or  repair, e.g., \cite{Knauer_1984, Lauffenburger_1993, Mesecke_2011}.   The complexity of the biochemistry involved in signalling networks,  necessitates an integrated approach combining theoretical and computational studies with experimental and modelling efforts to further our understanding of cell signalling.  Motivated by this need, in this work, we consider  the modelling and analysis  of signalling processes in biological tissues. Specifically, we are interested in modelling both the cell scale phenomena of receptor binding and cell signalling along with the tissue level dynamics of the ligands. 

Mathematical modelling and analysis of signalling processes involving receptor-ligand interactions  and GTPase (protein) molecules for  a single cell was considered in a number of recent works, for example \cite{Alphonse_2016, Elliott_2017, Raetz_Roeger_2012}. The majority of modelling studies to date in the literature focus only on phenomena at the scale of a single cell or simply naively `average out' the cell scale dependence for tissue level modelling \cite{Marciniak_2003, Menshykau_2013, Sherratt_1995, Wearing_2000}.   
However, the spatial separation between ligands diffusing in the intercellular space and receptors restricted to the cell membrane could be important even in tissue level models as shown, for example, in  \cite{Kurics_2014, Menshykau_2014} where it is crucial
 to ensuring robust branching in models for morphogenesis in organogenesis (e.g., in the formation of the lungs or the kidney). The heterogeneity in the interactions between ligands and receptors on the cell membrane given  by receptor clustering  on cell membranes  \cite{Hartman_2011, Thomason_2002, Welf_2009}  and/or lipid rafts \cite{Alonso_2001, Gaus_2005, Simons_2000} is also important for  intercellular signalling processes.  Similarly, in the mathematical and computational modelling of chemotaxis, cell polarisation through the clustering of  receptors at the leading edge and gradients in the macroscopic ligand field generated by the binding of these receptors  appear crucial to successful migration \cite{Elliott06062012, macdonald2016computational,  mackenzie2016local}.  Thus microscopic modelling of receptor-ligand-based intercellular signalling processes  in which both cell and tissue scale phenomena are accounted for is essential for a better understanding of biological systems. 
  
In this work we consider the multiscale modelling and analysis  of signalling processes in biological tissues.  Starting from a microscopic description consisting of coupled bulk-surface systems of partial differential equations (PDEs) posed in a domain consisting of cells and the extra cellular space, we will derive a macroscopic two-scale model as the number of cells tends to infinity. In contrast to  previous models for receptor-based signalling processes in biological tissues \cite{Ptashnyk_2008},  we consider diffusion of membrane resident species on the cell surface and we also extend previous  models by considering  interactions between receptors and co-receptors on the cell membrane leading to  activation of intracellular signalling processes. Furthermore, we propose a  robust and efficient numerical method for the approximation of the  macroscopic two-scale problem and apply it in a biologically relevant parameter regime.

\changes{
 The main difficulty in the multiscale  analysis of the microscopic  problem considered here is the strong nonlinearity of reaction terms coupled with surface diffusion and the dependence on a small parameter,  corresponding to the size of the microstructure. This requires a rather delicate analysis and a new approach in the derivation of a priori estimates. We employ  the trace and Gagliardo-Nirenberg inequalities together with an iteration processes to show the a priori estimates and boundedness of the solutions of the model equations. Similar ideas were used in \cite{Alphonse_2016}  to show the well-posedness of a system describing nonlinear ligand-receptor interactions for a single cell, whose shape is evolving in time.  However due to the multiscale nature  and the corresponding scaling in the microscopic equations, the techniques from \cite{Alphonse_2016} cannot be applied directly to obtain  uniform a priori estimates for  the solutions of our microscopic model.  To overcome this difficulty we use the structure of the nonlinear reaction terms and the periodic unfolding operator \cite{Cioranescu_2012, Cioranescu_2008, Graf_2014}.   
 }

 The bulk-surface coupling in the homogenised model induces some challenges in the  design of a two-scale numerical scheme. 
For the numerical approximation of the macroscopic two-scale system we employ a two-scale bulk-surface finite element method. Bulk-surface finite element methods have been used in a number of recent studies for the approximation of coupled bulk-surface systems of elliptic and parabolic equations, including those modelling receptor-ligand interactions \cite{elliott2012finite,macdonald2013simple,MadChuVen14,Raetz_Roeger_2012}, however to the best of the authors knowledge all such works have focussed on interactions at the scale of a single cell.  Coupling the bulk-surface finite  element approach with a two-scale finite element method \cite{muntean2010rate}, we are able to treat the approximation of the full macroscopic two-scale system and hence provide, as far as we are aware, the first work in which tissue level models for receptor-ligand interaction are simulated where receptor binding, unbinding and transport as well as cell signalling are taken into account at the cell scale. In order to validate the method we perform some benchmark tests to investigate the convergence of the method. We then propose and simulate a macroscopic two-scale cell signalling model  in a biologically relevant regime. Our results illustrate the influence of the cell shape on the transport of macroscopic species as well as spatial heterogeneities at the cell-scale and their influence on tissue level behaviour. We focus on incorporating the single cell model within a generic  cell signalling process outlined in \cite{Garcia_2014} into our multiscale modelling framework. However we note that the majority of signalling pathways that are described in the literature lie within the  general model framework considered in this work.  For example,\  GTPase (e.g. Rho)  and GPCR (G-protein coupled receptors)  related signalling pathways \cite{Lawson_2014},  uPAR-mediated signalling  processes in human tissue \cite{Smith_2010}  and Brassinosteroid hormone mediated  signalling in plant cells \cite{Clouse_2011}. 
 
The remainder of this paper is organised as follows. In Section~\ref{sec:micro_model} we derive our microscopic model for cell-signalling processes consisting of coupled bulk-surface systems of PDEs.   In Section~\ref{sec:well_pos} we prove existence and uniqueness results and derive some a priori estimates for solutions of the microscopic model. Convergence results in the limit as the number of cells tends to infinity and the resultant macroscopic two-scale model equations satisfied by the limiting solutions are presented in Section~\ref{sec:macro_model}. In Section~\ref{sec:scheme} we formulate a numerical scheme for the approximation of the macroscopic two-scale model. We benchmark the convergence of the scheme in Section~\ref{sec:benchmark} and in Section~\ref{sec:bio_model} we apply the numerical method to the approximation of a biological example of a GTPase signalling network taking parameter values from previous studies. 
\changes{The definitions and main properties of the two-scale convergence and the unfolding method as well as some technical calculations for the proof of the boundedness of a solution of the microscopic model are summarised in the Appendix.}

\section{Microscopic model}\label{sec:micro_model}

In this section we present  a derivation of a microscopic mathematical model    for signalling processes in biological tissues. 
  We consider a Lipschitz domain $\Omega\subset \IR^d$, with $d=2,3$, representing a part of a biological tissue and  assume a periodic distribution of cells in the tissue.  To describe the microscopic structure of the tissue, given by extra- and intracellular spaces separated by cell membranes, we consider a `unit cell' $Y=[0,1]^d$,  and the subdomains  $\overline Y_i\subset Y$ and $Y_e = Y \setminus Y_i$, together with  the boundary  $\Gamma = \partial Y_i$.   The domain occupied  by the intracellular space  is given by $\Omega_i ^\ve= \bigcup_{\xi \in \Xi^\ve} \ve(Y_i + \xi) $, 
where $\Xi^\ve = \{ \xi \in \mathbb Z^d, \; \; \ve(Y_i + \xi) \subset \Omega \}$,  and  the extracellular space is denoted by  $\Omega_e^\ve= \Omega\setminus \overline \Omega_i^\ve$.  The surfaces that describe cell membranes are denoted by $\Gamma^\ve= \bigcup_{\xi \in \Xi^\ve} \ve(\Gamma + \xi)$, see Figure~\ref{fig:micro_domain} for a sketch of the geometry.

  \begin{figure}[htbp]
    \begin{minipage}[b]{0.3\linewidth}
\includegraphics[trim = 0mm 0mm 0mm 0mm,  clip, width=\linewidth]{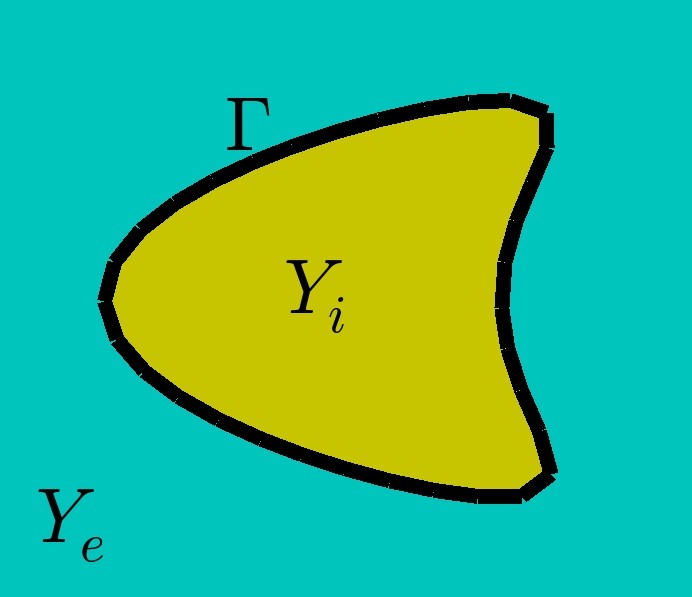}\\[1em]
\end{minipage}
\hskip 1em
  \begin{minipage}[b]{0.65\linewidth}
  \includegraphics[trim = 0mm 0mm 0mm 0mm,  clip, width=\linewidth]{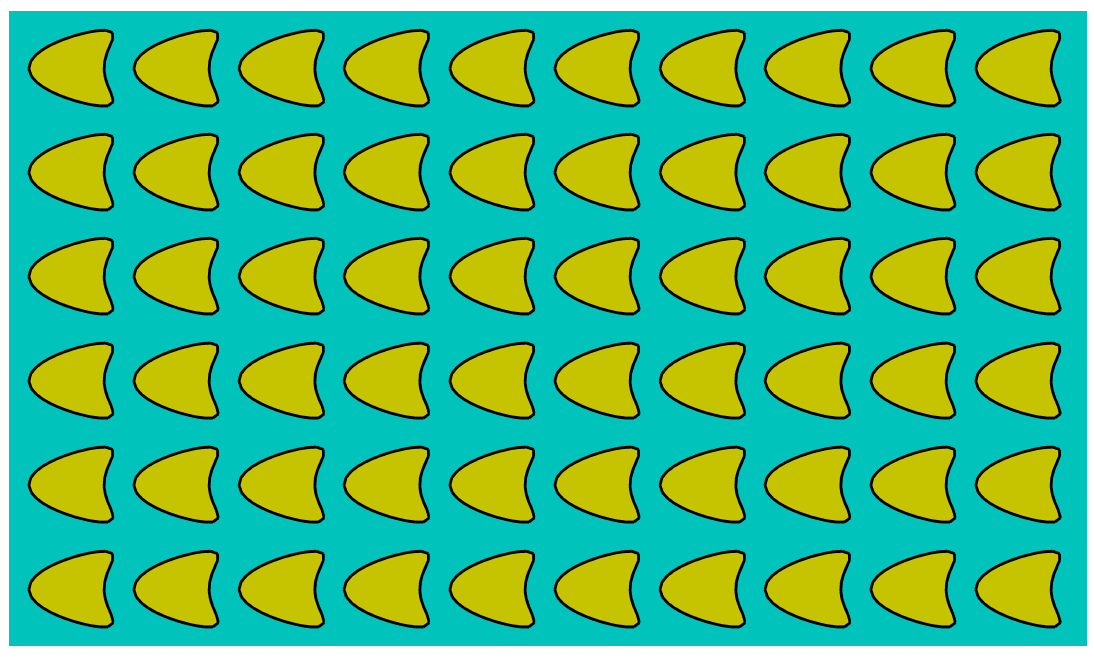} 
\end{minipage}
\caption{The left hand subfigure shows the `unit cell' that describes the microstructure consisting of a single cell with the intra- and extracellular spaces denoted by $Y_i$ and $Y_e$ respectively and the cell membrane by $\Gamma$. The right hand subfigure is a sketch of the tissue consisting of a periodic distribution of identically shaped cells surrounded by the extracellular space.}
  \label{fig:micro_domain}
  \end{figure}

In modelling intercellular signalling processes, we assume that the   signalling molecules (ligands) $c_e^\ve$ diffuse in the extracellular space and  interact with  cell membrane receptors. We distinguish between  free receptors $r_f^\ve$ (or extracellular domains of the free receptors)  and  bound receptors  $r_b^\ve$ (free receptor-ligand complexes). 
The model for the evolution of the ligand concentration \changes{ $c^\ve_e$} in the extracellular space $\Omega_e^\ve$ reads
 \begin{equation} \label{main1}
\begin{aligned}
&\partial_t c^\ve_e - \nabla\cdot ( D^\ve_e(x) \nabla c^\ve_e) = F_e(c^\ve_e) && \text{ in } \Omega_e^\ve,  \; t>0,\\
&D^\ve_e(x) \nabla c^\ve_e \cdot \nu = - \ve G_e (c^\ve_e, r_f^\ve,  r_b^\ve) && \text{ on } \Gamma^\ve,  \; t>0.
\end{aligned}
\end{equation}
\changes{Here the nonlinear Robin boundary condition $G_e (c^\ve_e, r_f^\ve,  r_b^\ve)$ defined by
\[
G_e (u, v,  w)  := a_e^\ve(x) u\,  v - b_e^\ve(x)\,  w,
\]
}
 describes the binding of ligands to free receptors located on the cell membranes, i.e., the creation of receptor-ligand complexes,  with binding rate $a_e^\ve$ and spontaneous dissociation of the complexes back into free receptors and ligands, with dissociation rate $b_e^\ve$. The function  $F_e$ models the production and/or decay of ligands in the extracellular space. 

 The signal from the extracellular  domain is transduced into the cell through the  activation by bound receptors $r_b^\ve$ of either membrane proteins, as is the case in signalling processes mediated  by G-protein-coupled receptors, or, the intracellular domains of  enzyme-linked membrane receptors or  co-receptors, as observed in plant hormone signalling processes. 
    Thus we shall  distinguish  between  active $p_a^\ve$ and  inactive $p_d^\ve$ proteins (co-receptors) or active and inactive intracellular domains of receptors.   We also consider   spontaneous deactivation of proteins (or intracellular domains of receptors)  with the deactivation rate $b_i^\ve$, as well as natural decay of all molecules with decay rates $d_j$, for $j=f,e,d,a$.  
Hence for the receptors and proteins on the cell membrane we obtain the following reaction-diffusion equations 
\begin{equation} \label{main3}
\begin{aligned}
& \partial_t r^\ve_f - \ve^2 D_{f}  \Delta_\Gamma r^\ve_f =  F_f(r^\ve_f, r^\ve_b)-  G_e (c^\ve_e, r_f^\ve, r_b^\ve)  - d_f\, r_f^\ve && \text{on } \Gamma^\ve,  t>0,  \\
& \partial_t r^\ve_b - \ve^2 D_{b} \Delta_\Gamma  r^\ve_b = \qquad  G_e (c^\ve_e, r_f^\ve, r_b^\ve) - G_d(r_b^\ve, p_d^\ve, p_a^\ve)  - d_b \, r_b^\ve && \text{on } \Gamma^\ve,  t>0, \\
& \partial_t p^\ve_d -   \ve^2 D_d \Delta_\Gamma  p^\ve_d  =  F_d(p^\ve_d) - G_d(r_b^\ve, p_d^\ve, p_a^\ve) - d_d \, p_d^\ve && \text{on } \Gamma^\ve,    t>0, \\
& \partial_t p^\ve_a - \ve^2  D_a \Delta_\Gamma  p^\ve_a = \qquad  G_d(r_b^\ve, p_d^\ve, p_a^\ve)  
- G_i(p^\ve_a, c_i^\ve) - d_a \, p_a^\ve && \text{on } \Gamma^\ve,   t>0,
\end{aligned}
\end{equation}
where  $\Delta_\Gamma$ denotes the Laplace-Beltrami operator on the surfaces $\Gamma^\ve$ and  the \changes{activation/deactivation}  reactions 
are defined by 
\changes{
\begin{equation*}
G_d(u, v, w)  := a_i^\ve(x)\,  u \,  v - b_{i}^\ve(x) \, w, 
\end{equation*}
}
\noindent with an \changes{activation (binding)} rate $a_i^\ve$. 
The function 
\changes{
$$
G_i(w, v) := \gamma^\ve_i(x) w-  \kappa_i^\ve(x) v
$$
} 
\noindent describes the transduction of the signal  into the cell interior  by activated proteins on the cell membrane (GTPase molecules)  or activated intracellular domains of enzyme-linked  receptors. 
The functions $F_f$ and $F_d$ model the production of new free receptors and inactive proteins, respectively. 

For the molecules involved in the intracellular part of the signalling pathway, we consider  
\begin{equation} \label{main2}
\begin{aligned}
&\partial_t c_i^\ve -\ve^2 \nabla\cdot ( D^\ve_i(x)  \nabla c_i^\ve ) = F_i(c_i^\ve) && \text{ in } \Omega_i^\ve,   \; t>0,\\
&\ve^2 D^\ve_i (x) \nabla c_i^\ve \cdot \nu = \ve G_i(p^\ve_a, c_i^\ve) && \text{ on } \Gamma^\ve,  \; t>0,
\end{aligned}
\end{equation}
where the function  $F_i$ models production and/or decay of the intracellular signalling molecules $c_i^\ve$. 

 We complete the microscopic model with the initial conditions
\begin{equation} \label{init_cond}
\begin{aligned}
& c^\ve_e(0, x) = c_{e,0}(x) \; && \text{ for } \; x\in \Omega_e^\ve, \\
& c^\ve_i(0,x) = c^\ve_{i,0}(x),  && c^\ve_{i,0}(x)  = c_{i,1}(x) c_{i,2}(x/\ve) \quad && \text{ for } \; x\in \Omega_i^\ve,   \\
& r_j^\ve(0,x) = r^\ve_{j,0}(x),   &&  r^\ve_{j,0}(x)= r_{j,1}(x) r_{j,2}(x/\ve) \; && \text{ for } \;   x\in \Gamma^\ve, \\
& p_s^\ve(0,x) = p^\ve_{s,0}(x),   &&  p^\ve_{s,0}(x)= p_{s,1}(x) p_{s,2}(x/\ve)  \; && \text{ for } \;  x\in  \Gamma^\ve, 
\end{aligned}
\end{equation}
where  $j=f,b$, and $s=d,a$, and the boundary condition for $c^\ve_e$ on the external boundary $\partial \Omega$ is given by,
\begin{equation} \label{bc_cond}
\begin{aligned}
 D^\ve_e(x) \nabla c^\ve_e \cdot \nu = 0  \quad  \text{ on } \; \partial \Omega, \;  \; t>0. 
\end{aligned}
\end{equation}

\changes{
\begin{remark}[Modelling generalisations]
   For simplicity of presentation we consider constant diffusion coefficients in the equations on $\Gamma^\ve$, however both the mathematical analysis and the numerical implementation allow for general space dependence $(x \text{ and/or } x/\ve)$  in the diffusion coefficients. 

The $\ve$-dependent scaling in the microscopic model \eqref{main1}--\eqref{bc_cond} yields  nontrivial equations in the limit and indeed is consistent with biological estimates of the parameter values c.f., Section~\ref{sec:bio_model}. 

The structure of space-dependent initial conditions ensures the  uniform in $\ve$ boundedness and strong two-scale convergence of the initial data $c^\ve_{i,0}$,  $r^\ve_{j,0}$,  and $p_{s,0}^\ve$ as $\ve \to 0$, where $j=f,b$ and $s=d,a$. It is possible to consider more general initial conditions, i.e.\
 $c^\ve_{i,0}(x)= c_{i,0}(x, x/\ve)$,  $r^\ve_{j,0}(x) = r_{j,0}(x, x/\ve)$, and $p^\ve_{s,0}(x) = p_{s,0}(x, x/\ve)$ if one assumes continuity of $c_{i,0}$, $r_{j,0}$, and $p_{s,0}$ with respect to at least one of the spatial variables, i.e., macroscopic $(x\in \Omega)$ or  microscopic $(y\in Y_i$ or $y\in \Gamma)$. 
\end{remark}

\begin{remark}[Binding kinetics]
  For reasons of clarity of exposition in the microscopic model we consider  linear or quadratic reactions for interactions between signalling molecules, receptors and proteins. Such reactions capture the main features of the biologically relevant interactions. The extension of the analysis and numerical simulations considered here to more general binding models such as cooperative binding or Michaelis-Menten terms and the addition of general Lipschitz functions in the reaction terms modelling additional phenomena  should be a relatively straightforward  task and is not anticipated to induce any major technical complications. 
\end{remark}
}

\changes{To ease readability, we introduce the following notation for $\tau \in (0, T]$  and any $T>0$, $\Omega^\ve_{l,\tau} :=(0,\tau)\times  \Omega^\ve_{l}$,  for $l=e,i$, 
$\Gamma^\ve_{\tau} :=(0,\tau)\times  \Gamma^\ve$,  $\Omega_\tau := \Omega\times(0, \tau)$, $\Gamma_\tau := \Gamma\times (0, \tau)$,
$$
\begin{aligned} 
&\la \phi, \psi \ra_{\Omega^\ve_{l, \tau} }:= \int_0^\tau \hspace{-0.11 cm}  \int_{\Omega^\ve_l} \phi\, \psi \, dx dt, \text{ for } \; l=e,i, && 
\la \phi, \psi \ra_{\Gamma^\ve_{\tau} }:= \int_0^\tau \hspace{-0.11 cm}  \int_{\Gamma^\ve} \phi\, \psi \, d\sigma^\ve dt, \\
& \la \phi, \psi \ra_{Y_l\times \Omega_{\tau} }:= \int_0^\tau \hspace{-0.11 cm}  \int_{\Omega} \int_{Y_l} \phi\, \psi \, dy dx dt,   \text{ for }  l=e,i, 
&& \la \phi, \psi \ra_{\Omega_{\tau} }:= \int_0^\tau \hspace{-0.11 cm}  \int_{\Omega} \phi\, \psi \, dx dt, \\
 &\la \phi, \psi \ra_{\Gamma\times\Omega_{\tau} }:= \int_0^\tau \hspace{-0.11 cm}  \int_{\Omega} \int_\Gamma \phi\, \psi \, d\sigma_y dx dt,  &&
\la \phi, \psi \ra_{\Gamma_{\tau} }:= \int_0^\tau \hspace{-0.11 cm}  \int_{\Gamma} \phi\, \psi \, d\sigma_y dt.
\end{aligned} 
$$
By $\la \cdot, \cdot \ra$ we denote the dual product in $H^1(\Omega_l^\ve)$,  with  $l=i,e$, or in $H^1(\Gamma^\ve)$ where it is clear from the arguments which of the three is meant.
} 

\section{Well posedness  and a priori estimates for the microscopic model} \label{sec:well_pos}

\changes{
In this section, we prove existence and uniqueness of a solution to  the microscopic problem \eqref{main1}--\eqref{bc_cond}.   We also derive a priori estimates that allow us to pass to the limit as the number of cells tends to infinity.

We use a  Galerkin method together with fixed point arguments to show the existence of a weak solution of \eqref{main1}--\eqref{bc_cond}. The main difficulty in the analysis is  to show  a priori estimates for solutions of the microscopic problem,  which are  global in time and independent of $\ve$. This is technically challenging due to the quadratic nonlinearities in the reaction terms and the scaling of the diffusion of the microscopic species. The tools we use to derive the estimates are the periodic unfolding method, Gagliardo-Nirenberg inequalities and in the proof of  boundedness of the species, we employ an Alikakos iteration technique~\cite{Alikakos_1979}. 
Uniqueness of the solution to \eqref{main1}--\eqref{bc_cond} follows from the boundedness result and the local Lipschitz continuity of the nonlinear terms. 

We find it convenient to use the periodic unfolding method described in Appendix \ref{appendix_A1}, see also  \cite{Cioranescu_2012, Cioranescu_2008}. There are two main advantages in using  unfolding methods in relation to the present study:
\begin{itemize} 
\item Unfolding operators  map functions defined on the oscillating $\ve$-dependent domains to functions defined on fixed domains which now depend on both macroscopic and microscopic variables; i.e., we can study functions on fixed domains whose geometry is independent of $\ve$ but in exchange must double the spatial dimension.
\item
The unfolding results in a separation between microscopic and macroscopic variables in the unfolded functions. This allows us to take advantage of the fact that under the action of the unfolding operator the differential operator (the Laplace-Beltrami operator)  in the equations defined on the oscillating surfaces is transformed into a differential operator with respect to the microscopic variables only.  Thus we are able to utilise the   higher regularity with respect to microscopic variables of the species  defined on the  oscillating surfaces and this appears to be crucial in establishing boundedness of the species uniformly in $\ve$.
\end{itemize}

We make the following biologically reasonable assumptions on the  coefficients  in the model equations and on the initial data. 

\begin{assumption}[Assumptions on the problem data for \eqref{main1}--\eqref{bc_cond}] \label{assumption}

\begin{itemize} 
\item We assume the usual ellipticity and boundedness conditions on the diffusivities of the different species, i.e., 
$
D_e \in C(\overline \Omega; L^\infty(Y_e))\text{ , with } D_e(x,y) \geq  \changes{\alpha_e } >0\text{ for a.a. }\ y \in Y_e\text{ and  }x \in \Omega,
$   $D_i \in L^\infty(Y_i)\text{ with }D_i(y) \geq \changes{ \alpha_i }>0\text{ for a.a. } y \in Y_i,
$
and
  $
D_j >0,\text{  for }j=f, b, d, a.
$ 
\item For the reaction kinetic coefficients, for $l=e,i$, we assume

 \[
 a_l , b_l, \gamma_i , \kappa_i \in L^\infty(\Gamma)\text{ and } a_l , b_l, \gamma_i , \kappa_i \text{ are nonnegative}.
 \]
 Moreover, we assume   for  $l=e,i$
\[
a_l^\ve(x) = a_l(x/\ve) , b_l^\ve(x) = b_l(x/\ve)\text{ and }  \gamma_i^\ve(x) = \gamma_i(x/\ve), \kappa_i^\ve(x) = \kappa_i(x/\ve).
\]

\item We assume boundedness of the initial conditions, i.e.,  
\[
c_{e,0}, c_{i,1} \in L^\infty(\Omega)\text{ and } c_{i,2} \in L^\infty(Y),
\]
and that for $j = f,b\text{  and  }s = a,d,$
\[
 r_{j,1}, p_{s,1} \in L^\infty(\Omega)\text{ and }r_{j,2}, p_{s,2} \in L^\infty(\Gamma).
 \]
 
\item  We further assume that the production/decay terms satisfy,
$F_l: \IR \to \IR$,  for $l=e,i,d$, and  $F_f:\IR^2 \to \IR$ are locally Lipschitz continuous  in $(-\mu, \infty)$ and $(-\mu, \infty)^2$, respectively, for some $\mu>0$.

Moreover, we assume the following growth bounds, for $l=e,i,d$
\[
F_l(\xi) \xi_- \leq |\xi_-|^2 \text{ and  }F_f(\xi, \eta)\xi_{-} \leq  C(|\xi_-|^2 + |\eta_-|^2) ,
\]
for $\xi_{-} = \min\{ \xi, 0\}$ and $\eta_{-} = \min\{ \eta, 0\}$,    and for $l=e,i,d$
\[
\text{$|F_l(\xi)|\leq C(1+ \xi)$    and  $|F_f(\xi, \eta)|\leq C(1+ \xi+ \eta)$    for $\xi, \eta \in \IR_+$.}
\] 
\end{itemize}

\end{assumption}

We define  $D_i^\ve(x) := \tilde D_i (x/\ve)$ and $D^\ve_e(x) := \tilde D_e(x, x/\ve)$ for $x\in \Omega$,  where $\tilde D_i$ and $\tilde D_e$ are $Y$-periodic extensions of  $D_i$ and of $D_e(x, \cdot)$ for $x\in \overline \Omega$, respectively. } 

We now introduce our notion of weak solutions of the microscopic problem \eqref{main1}--\eqref{bc_cond}. 
\begin{definition}[Weak solution of the microscopic problem]
A weak solution of the microscopic model \eqref{main1}--\eqref{bc_cond} are functions $c^\ve_l \in L^2(0, T; H^1(\Omega^\ve_l))$ and $r_j^\ve , p_s^\ve \in  L^2(0, T; H^1(\Gamma^\ve))$, with $\partial_t c^\ve_l \in L^2(0, T; H^1(\Omega_l^\ve)^\prime)$  and   $\partial_t r_j^\ve , \partial_t p_s^\ve \in  L^2(0, T; H^1(\Gamma^\ve)^\prime)$, for  $l=e,i$,  $s=a,d$,  and $j=f,b$,
satisfying 
\begin{equation} \label{weak_sol_1}
\begin{aligned} 
& \la \partial_t c^\ve_e, \phi \ra + \la D^\ve_e(x) \nabla c^\ve_e, \nabla \phi \ra_{\Omega_{e,T}^\ve} 
= \la  F_e(c^\ve_e), \phi \ra_{\Omega_{e,T}^\ve}   - \ve \la G_e (c^\ve_e, r_f^\ve, r_b^\ve), \phi \ra_{\Gamma_{T}^\ve},  \\
& \la \partial_t c^\ve_i, \psi \ra +  \la \ve^2 D^\ve_i(x) \nabla c^\ve_i, \nabla \psi \ra_{\Omega_{i,T}^\ve} 
= \la  F_i(c^\ve_i), \psi \ra_{\Omega_{i,T}^\ve}  + \ve \la G_i (p_a^\ve, c^\ve_i), \psi \ra_{\Gamma_{T}^\ve} , 
\end{aligned} 
\end{equation}
and 
\begin{equation} \label{weak_sol_2}
\begin{aligned}
 \la \partial_t r^\ve_f, \varphi \ra+  \la  \ve ^2 D_f \nabla_\Gamma r^\ve_f, \nabla_\Gamma  \varphi \ra_{\Gamma_{T}^\ve} 
&  =   \la F_f(r^\ve_f, r^\ve_b) - G_e (c^\ve_e, r_f^\ve, r_b^\ve) - d_f  r_f^\ve, \varphi \ra_{\Gamma_{T}^\ve} , \\
 \la \partial_t r^\ve_b, \varphi \ra +  \la \ve^2 D_b \nabla_\Gamma  r^\ve_b, \nabla_\Gamma  \varphi \ra_{\Gamma_{T}^\ve} 
&  =  \la G_e (c^\ve_e, r_f^\ve, r_b^\ve) - G_d ( r_b^\ve, p_d^\ve, p_a^\ve), \varphi \ra_{\Gamma_{T}^\ve} \\
& \qquad \qquad  \qquad \qquad \qquad  - \la d_b  \, r_b^\ve, \varphi \ra_{\Gamma_{T}^\ve} ,\\
 \la \partial_t p^\ve_d, \varphi \ra + \la \ve^2 D_d \nabla_\Gamma  p^\ve_d, \nabla_\Gamma  \varphi \ra_{\Gamma_{T}^\ve} 
& =  \la F_d(p^\ve_d) - G_d ( r_b^\ve, p_d^\ve, p_a^\ve) - d_d  p_d^\ve, \varphi \ra_{\Gamma_{T}^\ve} , \\
 \la \partial_t p^\ve_a, \varphi \ra +  \la \ve^2 D_a \nabla_\Gamma  p^\ve_a, \nabla_\Gamma  \varphi \ra_{\Gamma_{T}^\ve} 
&  =  \la G_d ( r_b^\ve, p_d^\ve, p_a^\ve) - G_i(p_a^\ve, c_i^\ve)- d_a  p_a^\ve, \varphi \ra_{\Gamma_{T}^\ve} ,
\end{aligned} 
\end{equation}
for all  $\phi \in L^2(0,T; H^1(\Omega^\ve_e))$,  $\psi \in L^2(0,T; H^1(\Omega^\ve_i))$, and $\varphi \in L^2(0,T; H^1(\Gamma^\ve))$, with the  initial conditions~\eqref{init_cond} satisfied  in the $L^2$-sense. 
\end{definition}

\changes{ 
In the subsequent analysis we shall make repeated use of the the following scaled trace inequality.
\begin{remark}[Scaled trace inequality]
 Using the assumptions on the microscopic geometry of $\Omega^\ve_l$ and applying the standard trace inequality for functions $v \in H^1(Y_l)$,  see e.g.~\eqref{trace_gen}, together with a scaling argument, we obtain the following trace inequality for the $L^2$-norm on $\Gamma^\ve$: 
\begin{equation}\label{ineq:trace}
\ve \|v\|^2_{L^2(\Gamma^\ve)} \leq \mu_\delta \|v\|^2_{L^2(\Omega_l^\ve)}  + \ve^2 \delta \|\nabla v\|^2_{L^2(\Omega_l^\ve)} \qquad \text{ with  } l = e,i, 
\end{equation}
for any fixed $\delta >0$, where the constant $\mu_\delta>0$   depends only on  $\delta$, $Y$, $Y_i$ and $\Gamma$,  and is independent of $\ve$, see e.g.\  \cite{Hornung_1991, Ptashnyk_2008}.  Notice that the natural $\ve$-scaling in the $L^2$-norm on the oscillating boundaries (surfaces of the microstructure) reflects the difference between volume and surface dimensions.   
 \end{remark}
 }
\changes{
\begin{remark}[$H^1$ extension]\label{rem:extension}
The assumptions on the structure of the microscopic domain $\Omega^\ve_e$ ensure that for  $v\in W^p(\Omega_e^\ve)$, with $1\leq p < \infty$,  there exists   an extension $\bar v$  from $\Omega^\ve_e$ into $\Omega$ such that 
\begin{equation}\label{estim:extension}
\|\bar v \|_{L^p(\Omega)} \leq \mu \|v \|_{L^p(\Omega_e^\ve)}, \quad \|\nabla \bar v \|_{L^p(\Omega)} \leq \mu \| \nabla v \|_{L^p(\Omega_e^\ve)}, 
\end{equation}
where the constant $\mu$ is independent of $\ve$, see e.g.\ \cite{Acerbi_1992, Cioranescu_II_1999, Hornung_1991}. 
 \end{remark}
}

\changes{ 
We now state our main result of this section, specifically the existence and uniqueness of  a weak solution of microscopic model  \eqref{main1}--\eqref{bc_cond} together with  
uniform (in $\ve$) estimates.
\begin{theorem}\label{theorem_exist}
Under Assumption~\ref{assumption}, for every fixed $\ve>0$,  there exists a unique nonnegative weak solution  of  the microscopic problem \eqref{main1}--\eqref{bc_cond}, which    satisfies the  a~priori estimates
\begin{equation}\label{estim:apriori_1} 
\begin{aligned}
&\|c_e^\ve\|_{L^\infty(0,T; L^2(\Omega_e^\ve))} + \|\nabla c_e^\ve\|_{ L^2(\Omega_{e,T}^\ve)}+\sqrt{\ve} \|c_e^\ve\|_{ L^2(\Gamma_T^\ve)} \leq C ,\\
&\|c_i^\ve\|_{L^\infty(0,T; L^2(\Omega_i^\ve))} + \|\ve \nabla c_i^\ve\|_{ L^2(\Omega_{i, T}^\ve)}  + \sqrt{\ve} \|c_i^\ve\|_{ L^2(\Gamma_T^\ve)} \leq C , \\
&\sqrt{\ve} \|r_j^\ve\|_{L^\infty(0,T; L^2(\Gamma^\ve))} + \sqrt{\ve} \|\ve\nabla_\Gamma r_j^\ve\|_{L^2(\Gamma_T^\ve)}  \leq C,  \\
&\sqrt{\ve}\|p_s^\ve\|_{L^\infty(0,T; L^2(\Gamma^\ve))} + \sqrt{\ve} \|\ve\nabla_\Gamma p_s^\ve\|_{L^2(\Gamma_T^\ve)}   \leq C, 
\end{aligned} 
\end{equation}
and for $l=e,i$, $s=a,d$, and $j=f,b$
\begin{equation}\label{estim:bound} 
\begin{aligned}
\|c_l^\ve\|_{L^\infty(0,T; L^\infty(\Omega_l^\ve))} +  \|r_j^\ve\|_{L^\infty(0,T; L^\infty(\Gamma^\ve))} + \|p_s^\ve\|_{L^\infty(0,T; L^\infty(\Gamma^\ve))}  \leq C, 
\end{aligned} 
\end{equation}
where  the constant $C$ is independent of  $\ve$.
\end{theorem}

To aid readability, we split the proof of Theorem \ref{theorem_exist} into a series of Lemmas. Namely, in Lemmas \ref{lem:existence},  \ref{lem:positivity}, \ref{lem:apriori}, \ref{lem:boundedness} and \ref{lem:uniqueness} we show existence, nonnegativity, the apriori estimates \eqref{estim:apriori_1}, boundedness and uniqueness of solutions to \eqref{main1}--\eqref{bc_cond} respectively.
 
\begin{lemma} \label{lem:existence}
There exists a weak solution to the microscopic problem  \eqref{main1}--\eqref{bc_cond}.
\end{lemma}
\begin{proof}  
Existence of a weak solution to problem  \eqref{main1}--\eqref{bc_cond}  is demonstrated  by showing the existence of a fixed point of the operator equation   $K: \mathcal A \to \mathcal A$, with 
$$
\mathcal A= \{ (u,v) \in L^2(0, T; L^4(\Gamma^\ve))^2, \; \text{ with } \;  u\geq 0 \text{ and } v \geq 0 \text{ on } (0,T)\times \Gamma^\ve \}, 
$$
 defined such that for given $(g^\ve,h^\ve)\in \mathcal A$ we consider   $(r_f^\ve, r_b^\ve) = K(g^\ve, h^\ve)$, where the   functions  $r_f^\ve$ and $r_b^\ve$ are solutions of the following coupled problem 
\begin{equation} \label{weak_sol_1_iter}
\begin{aligned} 
& \partial_t c^\ve_e- \nabla\cdot( D^\ve_e(x) \nabla c^\ve_e)
=  F_e(c^\ve_e) && \text{ in } \; \Omega_{e,T}^\ve, \\
& D^\ve_e(x) \nabla c^\ve_e \cdot \nu = - \ve  G_e (c^\ve_e, g^\ve, h^\ve)  && \text{ on } \; \Gamma_{T}^\ve, \\
& D^\ve_e(x) \nabla c^\ve_e \cdot \nu = 0 &&  \text{ on } \; (\partial\Omega)_{T}, \\
&  \partial_t c^\ve_i-  \ve^2 \nabla\cdot( D^\ve_i(x) \nabla c^\ve_i)  
=  F_i(c^\ve_i)  && \text{ in } \; \Omega_{i,T}^\ve, \\
& D^\ve_i(x) \nabla c^\ve_i \cdot \nu =   \ve  G_i (p_a^\ve, c^\ve_i) && \text{ on } \;  \Gamma_{T}^\ve,  
\end{aligned} 
\end{equation}
and 
\begin{equation} \label{weak_sol_2_iter}
\begin{aligned}
&  \partial_t r^\ve_f -   \ve ^2  \nabla_\Gamma\cdot (D_f \nabla_\Gamma r^\ve_f) 
=    F_f(r^\ve_f, h^\ve) - G_e (c^\ve_e, g^\ve,  r_b^\ve) - d_f\,  r_f^\ve , \\
&  \partial_t r^\ve_b  -   \ve^2  \nabla_\Gamma\cdot (D_b \nabla_\Gamma  r^\ve_b)
=    G_e (c^\ve_e, g^\ve, r_b^\ve) - G_d (h^\ve, p_d^\ve, p_a^\ve) - d_b\,  r_b^\ve ,\\
&  \partial_t p^\ve_d  -  \ve^2  \nabla_\Gamma\cdot (D_d \nabla_\Gamma  p^\ve_d) 
=   F_d(p^\ve_d) - G_d ( h^\ve, p_d^\ve, p_a^\ve) - d_d\,  p_d^\ve , \\
&  \partial_t p^\ve_a   -   \ve^2 \nabla_\Gamma\cdot ( D_a \nabla_\Gamma  p^\ve_a) 
=   G_d (h^\ve, p_d^\ve, p_a^\ve) - G_i(p_a^\ve, c_i^\ve)- d_a\,  p_a^\ve, 
\end{aligned} 
\end{equation}
together with the initial  conditions \eqref{init_cond}. 

 To prove the nonnegativity of solutions of problem \eqref{weak_sol_1_iter}, \eqref{weak_sol_2_iter}, and  \eqref{init_cond} we start by taking $c_e^{\ve, -}  = \min\{c_e^\ve, 0\}$ as a test function in the equation for $c_e^{\ve}$ in \eqref{weak_sol_1_iter}.
  Using the nonnegativity of  $g^\ve$ and  $h^\ve$, the assumptions on $F_e$ and the structure of function $G_e$  we obtain that $\|c_e^{\ve, -}\|_{L^\infty(0,T; L^2(\Omega_e^\ve))} \leq 0$.  Hence $c_e^{\ve, -} =0$ a.e. in $(0,T)\times \Omega_{e}^\ve$ and $c_e^\ve \geq 0$ a.e.\ in $(0,T)\times \Omega_{e}^\ve$. 
Then using the nonnegativity of $c^\ve_e$, $g^\ve$, and $h^\ve$,  and choosing $c_i^{\ve, -}$,   $r_l^{\ve, -}$, $p_s^{\ve, -}$, with $l=f,b$ and $s=a,d$,  as test functions in the  equation for $c_i^{\ve}$  in   \eqref{weak_sol_1_iter} and in equations  in \eqref{weak_sol_2_iter}, respectively, and using the assumptions on the functions $F_f$, $F_d$, $F_i$, $G_i$ and $G_d$ we obtain nonnegativity of 
$c_i^\ve$, $r_l^\ve$, and $p_s^\ve$, where  $l=f,b$ and $s=a,d$.   

The existence of a solution of problem \eqref{weak_sol_1_iter}, \eqref{weak_sol_2_iter}, and  \eqref{init_cond} for given  $(g^\ve, h^\ve)\in \mathcal A$  can be shown using a Galerkin method and  a priori  estimates, equivalent to those stated in \eqref{estim:apriori_1} (where we now consider estimates for the solutions of  problem \eqref{weak_sol_1_iter}, \eqref{weak_sol_2_iter}, and  \eqref{init_cond}). 
As is standard the necessary estimates are derived for Galerkin approximation sequences and passing to the limit yields the estimates for the problem \eqref{weak_sol_1_iter}, \eqref{weak_sol_2_iter}, and  \eqref{init_cond}.  We note that the derivation of the estimates    \eqref{estim:apriori_1}   for solutions  of  problem \eqref{weak_sol_1_iter}, \eqref{weak_sol_2_iter}, and  \eqref{init_cond}  follows exactly the same argument as in the proof of Lemma~\ref{lem:apriori},   with  $\ve\| h^\ve\|^2_{L^2(0, \tau; L^q(\Gamma_\tau^\ve))}$ in place of $\ve\|r_b^\ve\|^2_{L^2(0,\tau; L^q(\Gamma^\ve))}$ for $q=2,4$.

The a priori estimates in \eqref{estim:apriori_1},  together with standard arguments for parabolic equations,  ensure that for any fixed $\ve >0$ we have $\partial_t c^\ve_l \in L^2(0,T; H^{1}(\Omega_l^\ve)^\prime)$ for $l=e,i$  and  $\partial_t r_j^\ve, \partial_t p_s^\ve \in L^2(0,T; H^{1}(\Gamma^\ve)^\prime)$ for  $j=f,b$, $s = a,d$.   

Now using  the compact embedding of $[H^1(0,T; H^{1}(\Gamma^\ve)^\prime) \cap L^2(0,T; H^1(\Gamma^\ve))]^2$ in $L^2(0,T; L^4(\Gamma^\ve))^2$  and the fact that $\mathcal A$ is a convex subset of $L^2(0,T; L^4(\Gamma^\ve))^2$  and applying  the Schauder fixed-point theorem yields the  existence of a weak solution to the microscopic problem   \eqref{main1}--\eqref{bc_cond}  for each fixed $\ve$. 
\end{proof} 

}

\changes{
To show nonnegativity of solutions, the a priori estimates \eqref{estim:apriori_1} and boundedness of solutions of the microscopic problem  \eqref{main1}--\eqref{bc_cond}, we first consider a truncated model obtained by taking  $r_{f,M}^\ve$ instead of $r_f^\ve$ in function $G_e$
and $r_{b,M}^\ve$ instead of $r_b^\ve$ in function $G_d$ in equations  \eqref{main1}--\eqref{bc_cond}, where 
\[
r_{j,M}^\ve := \min\{M, r_{j}^\ve\} + \max\{ -M, r_{j}^\ve\} -r_{j}^\ve, 
\text{ for $j=f,b$ and some  $M>0$}.
\]
Then we show that all solutions of the truncated model are nonnegative. For nonnegative solutions in Lemmata~\ref{lem:apriori} and \ref{lem:boundedness}  we prove the a priori estimates \eqref{estim:apriori_1} and boundedness, independent of the truncation constant $M$. 
Thus passing to the limit as $M \to \infty$ yields the nonnegativity, a priori estimates \eqref{estim:apriori_1} and boundedness of solutions of the original problem  
 \eqref{main1}--\eqref{bc_cond}.  
 
 For simplicity of presentation we derive the a priori estimates and boundedness of nonnegative  solutions of original problem \eqref{main1}--\eqref{bc_cond}, clearly the same arguments apply for the corresponding truncated model.

\begin{lemma} \label{lem:positivity} 
Under Assumptions~\ref{assumption} solutions of  problem  \eqref{main1}--\eqref{bc_cond} are nonnegative.
\end{lemma} 

\begin{proof}
We first  consider the truncated model with   $r_{f,M}^\ve$ instead of $r_f^\ve$ in function $G_e$
and $r_{b,M}^\ve$ instead of $r_b^\ve$ in function $G_d$ in equations  \eqref{main1}--\eqref{bc_cond}.  Then considering $c_l^{\ve, -}$, $r_j^{\ve, -}$, and $p_s^{\ve, -}$ as test functions in equations  \eqref{weak_sol_1} and  \eqref{weak_sol_2} with $G_e(c_e^\ve, r_{f,M}^\ve, r_b^\ve)$ and $G_d(r_{b,M}^\ve, p_d^\ve, p_a^\ve)$ instead of $G_e(c_e^\ve, r_{f}^\ve, r_b^\ve)$ and $G_d(r_{b}^\ve, p_d^\ve, p_a^\ve)$, respectively,  using the trace and Gronwall inequalities we obtain 
$$
\| c_l^{\ve, -}\|_{L^\infty(0,T; L^2(\Omega^\ve_l))} + \| r_j^{\ve, -}\|_{L^\infty(0,T; L^2(\Gamma^\ve))} +  \| p_s^{\ve, -}\|_{L^\infty(0, T; L^2(\Gamma^\ve))} \leq 0, 
$$
for $l=e,i$, $j=f,b$,  and $s=a,d$, and hence  solutions of the truncated problem $c_l^{\ve}$, $r_j^{\ve}$, and $p_s^{\ve}$  are nonnegative.  Since for nonnegative solutions  we have a priori estimates and  boundedness uniformly with respect to $M$, see Lemmata~\ref{lem:apriori} and \ref{lem:boundedness}, we can pass to the limit as $M \to \infty$  and obtain that   solutions of the  original  problem  \eqref{main1}--\eqref{bc_cond} are nonnegative. 
\end{proof}

}

\changes{
Next we   derive the {a priori} estimates \eqref{estim:apriori_1} for  solutions of problem  \eqref{main1}--\eqref{bc_cond}.

\begin{lemma} \label{lem:apriori} 
 Under Assumptions~\ref{assumption} nonnegative   solutions of the microscopic problem  \eqref{main1}--\eqref{bc_cond} satisfy  the a~priori estimates 
 \eqref{estim:apriori_1}. 
\end{lemma} 

\begin{proof} 
Considering $c_e^\ve$ and $c_i^\ve$ as test functions in  the  weak formulation  \eqref{weak_sol_1} of the equations for $c_e^\ve$ and $c_i^\ve$  yields 
\begin{equation}
\begin{aligned} 
\frac 1 2 \|c_e^\ve(\tau) \|^2_{L^2(\Omega_e^\ve)} + \alpha_e \|\nabla  c_e^\ve \|^2_{L^2(\Omega_{e, \tau}^\ve)}  \leq  & \;
 \frac 12 \|c_e^\ve(0) \|_{L^2(\Omega_e^\ve)} +  \la F_e(c^\ve_e) , c_e^\ve\ra_{\Omega_{e, \tau}^\ve}\\
 & - \ve \la G_e (c^\ve_e, r_f^\ve, r_b^\ve), c^\ve_e \ra_{\Gamma_{\tau}^\ve},  \\
\frac 12  \|c_i^\ve(\tau) \|^2_{L^2(\Omega_i^\ve)} +  \alpha_i   \| \ve \nabla  c_i^\ve \|^2_{L^2(\Omega_{i, \tau}^\ve)}  \leq & \; 
\frac 12 \|c_i^\ve(0) \|_{L^2(\Omega_i^\ve)} +  \la F_i(c^\ve_i) , c_i^\ve\ra_{\Omega_{i, \tau}^\ve} \\
&+ \ve \la G_i (p_a^\ve, c^\ve_i), c^\ve_i \ra_{\Gamma_{\tau}^\ve}, 
 \end{aligned} 
\end{equation}
for $\tau \in (0, T]$. Using the structure of $G_e$ and $G_i$, the nonnegativity of solutions,    and the assumptions on the coefficients  in Assumption~\ref{assumption}, together with the  trace inequality \eqref{ineq:trace} we obtain 
\begin{equation}\label{estim_ce_ci} 
\begin{aligned}
\|c_e^\ve(\tau) \|^2_{L^2(\Omega_e^\ve)} +  \|\nabla  c_e^\ve \|^2_{L^2(\Omega_{e, \tau}^\ve)} \leq C\left[1+ \ve \|r_b^\ve\|^2_{L^2(\Gamma^\ve_\tau)} + \| c_e^\ve \|^2_{L^2(\Omega_{e, \tau}^\ve)}  \right], \\
\|c_i^\ve(\tau) \|^2_{L^2(\Omega_i^\ve)} +  \|\ve \nabla  c_i^\ve \|^2_{L^2(\Omega_{i, \tau}^\ve)} \leq C\left[1+ \ve \|p_a^\ve\|^2_{L^2(\Gamma^\ve_\tau)} + \| c_i^\ve \|^2_{L^2(\Omega_{i, \tau}^\ve)}  \right].
\end{aligned} 
\end{equation}
Taking  $r_f^\ve$ as a test function in the equation for $r_f^\ve$, and using  the nonnegativity of $c_e^\ve$, $r_f^\ve$ and $r_b^\ve$, the structure of $G_e$, and the assumptions on $F_f$ we have 
\begin{equation}\label{estim_r_f}
\ve \|r_f^\ve(\tau) \|^2_{L^2(\Gamma^\ve)} +\ve   \| \ve \nabla_\Gamma r_f^\ve\|^2_{L^2(\Gamma_\tau^\ve)}   \leq  C\left[1+ \ve \|r^\ve_b\|^2_{ L^2(\Gamma_\tau^\ve)} +  \ve \| r_f^\ve\|^2_{L^2(\Gamma_\tau^\ve)} \right],
\end{equation}
for  $\tau \in (0, T]$.  
Considering  the equation for the sum of $r_b^\ve$ and $r_f^\ve$,  taking $r_b^\ve+ r_f^\ve$ as a test function, and using the structure of the function $G_d$, together with the nonnegativity of  $r_f^\ve$, $r_b^\ve$, and $p_d^\ve$ and the estimate \eqref{estim_r_f},  yields 
\begin{equation}\label{estim_rf_rb}
\begin{aligned}
\ve \|r_f^\ve(\tau) +r_b^\ve(\tau)\|^2_{L^2(\Gamma^\ve)} & \, +  \ve \|\ve \nabla_\Gamma r_b^\ve\|^2_{L^2(\Gamma_\tau^\ve)}  \leq C_1 \big[ 1+   \ve \|\ve \nabla_\Gamma r_f^\ve\|^2_{L^2(\Gamma_\tau^\ve)} \\ 
& +\ve \|r_f^\ve \|^2_{L^2(\Gamma^\ve_\tau)}    + \ve \|r_b^\ve \|^2_{L^2(\Gamma^\ve_\tau)} + \ve \|p_a^\ve \|^2_{L^2(\Gamma^\ve_\tau)} \big] \\
& \leq  C_2 \left[1+ \ve \|r^\ve_b\|^2_{ L^2(\Gamma_\tau^\ve)} + \ve \|r^\ve_f\|^2_{ L^2(\Gamma_\tau^\ve)} + \ve \|p_a^\ve \|^2_{L^2(\Gamma^\ve_\tau)} \right], 
\end{aligned}
\end{equation}
for  $\tau \in (0, T]$.  
In a similar way as for $r_f^\ve$,  using the structure of $G_d$, the assumptions on $F_d$, and the nonnegativity of $r_b^\ve$, $p_d^\ve$ and $p_a^\ve$,  we obtain
\begin{equation}\label{estim_p_d} 
\ve \|p_d^\ve(\tau) \|^2_{L^2(\Gamma^\ve)} +  \ve   \| \ve \nabla_\Gamma p_d^\ve\|^2_{L^2(\Gamma_\tau^\ve)}   \leq  C \left[1+ \ve \|p^\ve_a\|^2_{ L^2(\Gamma_\tau^\ve)} +  \ve  \| p_d^\ve\|^2_{L^2(\Gamma_\tau^\ve)}  \right],
\end{equation}
for $\tau \in (0,T]$. 
Considering the equation for the sum of $p_d^\ve$ and $p_a^\ve$ and taking $p_d^\ve+ p_a^\ve$ as a test function yields
\begin{equation}\label{estim_pd_pa}
\begin{aligned}
\ve \|p_d^\ve(\tau)+p_a^\ve(\tau)\|^2_{L^2(\Gamma^\ve)} + \ve \|\ve \nabla_\Gamma p_a^\ve\|^2_{L^2(\Gamma_\tau^\ve)}   \leq C_1 \big[1+  \ve  \|\ve \nabla_\Gamma p_d^\ve\|^2_{L^2(\Gamma_\tau^\ve)} \\
+ \ve \|p_a^\ve \|^2_{L^2(\Gamma^\ve_\tau)}  + \ve  \|p_d^\ve \|^2_{L^2(\Gamma^\ve_\tau)} + \ve \|c_i^\ve \|^2_{L^2(\Gamma^\ve_\tau)} \big] \\
  \leq C_2\left[1+ \ve \|p_a^\ve \|^2_{L^2(\Gamma^\ve_\tau)} + \ve  \|p_d^\ve \|^2_{L^2(\Gamma^\ve_\tau)} + \ve \|c_i^\ve \|^2_{L^2(\Gamma^\ve_\tau)} \right],  
\end{aligned}
\end{equation}  
for  $\tau \in (0, T]$.   Combining estimates \eqref{estim_ce_ci}--\eqref{estim_pd_pa} and using  the Gronwall inequality and trace inequality \eqref{ineq:trace},   imply  the a priori estimates  stated  in \eqref{estim:apriori_1}. 
\end{proof}

The main technical result of this section is the following uniform boundedness result. A number of the more laborious calculations are given in the Appendix~\ref{boundedness_sketch} in order to aid readability of the manuscript.

\begin{lemma} \label{lem:boundedness} 
Under Assumptions~\ref{assumption} nonnegative solutions of  problem  \eqref{main1}--\eqref{bc_cond} are bounded, uniformly in $\ve$.
\end{lemma} 

\begin{proof}
To show boundedness of solutions to the  microscopic model \eqref{main1}--\eqref{bc_cond} we introduce the  periodic unfolding operator $\T^\ve_{Y_l}: L^p(\Omega^\ve_{l, T})\to L^p(\Omega_T\times Y_l)$,  with $l=i,e$,  and the boundary unfolding operator $\T^\ve_\Gamma: L^p(\Gamma^\ve_T)\to L^p(\Omega_T\times \Gamma)$, where $1\leq p \leq \infty$, see Appendix \ref{appendix_A1} or e.g.\ \cite{Cioranescu_2012, Cioranescu_2008}  for the definition and properties of the periodic unfolding operator.
For simplification of the presentation we use the same notation $\T^\ve$ for the unfolding operator $\T^\ve_{Y_l}$, for $l=e,i$,   and the boundary unfolding operator $\T^\ve_\Gamma$ as it is clear from the context which operator is applied. 

Integration by parts in time of the terms in  equations    \eqref{weak_sol_1} and \eqref{weak_sol_2} that involve time derivatives,   applying the unfolding operator  and using the nonnegativity of the solutions, we obtain the following estimates for $x\in \Omega$, (to aid readability of the manuscript the details of the derivation of the estimates are given in Appendix \ref{boundedness_sketch})
\begin{equation}
\begin{split}
\label{estim_pdaci_L2}
 \|\T^\ve (r_f^\ve) (\tau) \|^2_{L^2(\Gamma)} +& \| \nabla_{\Gamma, y} \T^\ve (r_f^\ve)\|^2_{L^2(\Gamma_\tau)} 
 + \|\T^\ve (r_b^\ve)(\tau)\|^2_{L^2(\Gamma)} + \| \nabla_{\Gamma, y} \T^\ve( r_b^\ve)\|^2_{L^2(\Gamma_\tau) } \\
   &\leq    C \left[ 1
 + \|\T^\ve (r_b^\ve)\|^2_{L^2(\Gamma_\tau)} +  \|\T^\ve (r_f^\ve) \|^2_{L^2(\Gamma_\tau) } +   \|\T^\ve (p_a^\ve) \|^2_{L^2(\Gamma_\tau)} \right]    \\
  \|\T^\ve (p_d^\ve)(\tau) \|^2_{L^2(\Gamma)} +& \| \nabla_{\Gamma, y} \T^\ve (p_d^\ve)\|^2_{L^2(\Gamma_\tau)}
+  \|\T^\ve (p_a^\ve)(\tau)\|^2_{L^2(\Gamma)} + \| \nabla_{\Gamma, y} \T^\ve (p_a^\ve)\|^2_{L^2(\Gamma_\tau)}
 \\
  &\leq C_1\big[ 1+ \|\T^\ve (p_a^\ve)\|^2_{L^2(\Gamma_\tau)}  + \|\T^\ve (p_d^\ve)\|^2_{L^2(\Gamma_\tau)}+ \|\T^\ve (c_i^\ve)\|^2_{L^2(\Gamma_\tau)}\big],  \\
  \|\T^\ve (c_i^\ve)(\tau)\|^2_{L^2(Y_i)}  + & \| \nabla_{y} \T^\ve (c_i^\ve)\|^2_{L^2(Y_{i,\tau})} 
  \\ 
  \leq C_3 &\left[1+  \|\T^\ve (c_i^\ve)\|^2_{L^2(\Gamma_\tau)}   \right]+ C_4\big[ \|\T^\ve (c_i^\ve)\|^2_{L^2(Y_{i,\tau})}  +   \|\T^\ve (p_a^\ve)\|^2_{L^2(\Gamma_\tau)}\big].
\end{split}
\end{equation}
 Gronwall's inequality and a trace estimate for $c_i^\ve$, similar to \eqref{ineq:trace}, yields for $x\in \Omega$,
\begin{equation}\label{estim_rfb}
\begin{aligned}
& \|\T^\ve (r_l^\ve)\|_{L^\infty(0,T; L^2(\Gamma))} + \| \nabla_{\Gamma, y} \T^\ve (r_l^\ve)\|_{L^2(\Gamma_T)} \leq C, \quad \text{ for } \; l=f,b, \\
 & \|\T^\ve (p_s^\ve)\|_{L^\infty(0, T; L^2(\Gamma))} + \| \nabla_{\Gamma, y} \T^\ve (p_s^\ve)\|_{L^2(\Gamma_T)} \leq C, \quad \text{ for } \; s=d,a, \\
  & \|\T^\ve (c_i^\ve)\|_{L^\infty(0, T; L^2(Y_i))} + \| \nabla_y \T^\ve (c_i^\ve)\|_{L^2(Y_{i,T})} \leq C.
 \end{aligned} 
\end{equation} 
 Applying the Gagliardo-Nirenberg and trace inequalities and using the fact that ${\rm dim}(\Gamma) \leq  2$ we obtain for $l=f,b$ and $s=d,a$, and for $x\in \Omega$ and $\tau \in (0, T]$,
\begin{equation}\label{estim_GN_rpc}
\begin{aligned}
& \|\T^\ve (r_l^\ve)\|_{L^4(\Gamma_\tau)}\leq  C_1 \| \nabla_{\Gamma, y} \T^\ve (r_l^\ve)\|^{1/2}_{L^2(\Gamma_\tau)} \| \T^\ve (r_l^\ve)\|^{1/2}_{L^\infty(0,\tau; L^2(\Gamma))}  \leq C, \\
 & \|\T^\ve (p_s^\ve)\|_{L^4(\Gamma_\tau)} \leq C_2 \| \nabla_{\Gamma, y} \T^\ve (p_s^\ve)\|^{1/2}_{L^2(\Gamma_\tau)}\| \T^\ve (p_s^\ve)\|^{1/2}_{L^\infty(0,\tau; L^2(\Gamma))}  \leq C, \\
  & \|\T^\ve (c_i^\ve)\|_{L^2(\Gamma_\tau)} \leq C_3\left[ \| \nabla_{y} \T^\ve (c_i^\ve)\|_{L^2(Y_{i,\tau})} + \| \T^\ve (c_i^\ve)\|_{L^2(Y_{i,\tau})}  \right]\leq C, \\
   & \|\T^\ve (c_i^\ve)\|_{L^2(0,\tau; L^4(\Gamma))} \leq C_4\left[ \| \nabla_{y} \T^\ve (c_i^\ve)\|_{L^2(Y_{i,\tau})} + \| \T^\ve (c_i^\ve)\|_{L^2(Y_{i,\tau})}  \right]\leq C, 
 \end{aligned} 
\end{equation} 
where the constant $C$ is independent of $\ve$  and $x$. We also make use of the inequality    
\begin{equation} \label{estim_GN_1}
\begin{aligned} 
& \|\T^\ve(r_b^\ve)\|_{L^2(\Gamma)}  \leq \mu \|\nabla_{\Gamma, y} \T^\ve(r_b^\ve)\|^{1/2}_{L^2(\Gamma)} \|\T^\ve(r_b^\ve)\|^{1/2}_{L^1(\Gamma)} , \\
& \|\T^\ve(r_b^\ve)\|_{L^4(\Gamma)}  \leq \mu \|\nabla_{\Gamma, y} \T^\ve(r_b^\ve)\|^{1/2}_{L^2(\Gamma)} \|\T^\ve(r_b^\ve)\|^{1/2}_{L^2(\Gamma)} , 
\end{aligned} 
\end{equation}
for $x\in \Omega$, $t \in (0,T]$, and a constant $\mu>0$ independent of $\ve$. 

The   a priori estimates \eqref{estim:apriori_1} and the properties of the unfolding operator, see Appendix~\ref{appendix_A1} or \cite{Cioranescu_2012}  for more details,   imply 
\begin{equation}
\begin{aligned} 
\|\nabla_y \T^\ve(c^\ve_e)\|_{L^2(\Omega_T\times Y_e)} =  \| \ve \T^\ve(\nabla c^\ve_e)\|_{L^2(\Omega_T\times Y_e)} \leq 
C_1 \| \ve \nabla c^\ve_e\|_{L^2(\Omega^\ve_{e, T})}  \leq C_2  \ve , \\
 \|\nabla_y \T^\ve(c^\ve_i)\|_{L^2(\Omega_T\times Y_i)} =  \| \ve \T^\ve(\nabla c^\ve_i)\|_{L^2(\Omega_T\times Y_i)}
 \leq C_3 \| \ve  \nabla c^\ve_i\|_{L^2(\Omega_{i,T}^\ve)}  \leq C_4 . 
 \end{aligned} 
\end{equation}
Then using the Sobolev embedding theorem, where $\text{dim} (Y_l)\leq 3$ for $l=e,i$,    and the trace inequality we obtain 
\begin{equation}\label{eq:unfold_domain}
\begin{aligned} 
 \|\T^\ve(c^\ve_l)\|_{L^2(\Omega_T\times \Gamma)} +
 \|\T^\ve(c^\ve_l)\|_{L^2(\Omega_T; L^4( \Gamma))} +  \|\T^\ve(c^\ve_l)\|_{L^2(\Omega_T; L^4(Y_l))} 
  \\
\leq  \mu \left[\|\T^\ve(c^\ve_l)\|_{L^2(\Omega_T\times Y_l)}  + \|\nabla_y \T^\ve(c^\ve_l)\|_{L^2(\Omega_T\times Y_l)} \right]\leq C, 
 \end{aligned} 
\end{equation}
for $l = e, i$, where the constants $\mu>0$ and $C>0$ are independent of $\ve$. 

We now  use an Alikakos iteration method \cite{Alikakos_1979} to prove  the boundedness of solutions to  \eqref{main1}--\eqref{bc_cond}. 
Considering first  $|\T^\ve( r_f^\ve)|^{p-1}$, for $p\geq 4$, as a test function in the equation  for $\T^\ve( r_f^\ve)$,  (see \eqref{eq:unfold_boundary} in Appendix \ref{boundedness_sketch}), using the assumptions on the function $F_f$,  the nonnegativity of $r_f^\ve$, $r_b^\ve$ and $c_e^\ve$,   and  the Gagliardo-Nirenberg inequality  we obtain the following estimate for $x\in \Omega$ and $\tau\in (0, T]$, (see Appendix \ref{boundedness_sketch} for the details)
 \begin{equation}\label{rf_bound_1}
\begin{aligned} 
  \|\T^\ve (r_f^\ve)(\tau) \|^p_{L^p(\Gamma)}&  +   \|\nabla_{\Gamma, y} |\T^\ve (r_f^\ve)|^{\frac p 2} \|^2_{L^2(\Gamma_\tau)}\\ 
&\leq  C_1 \|\T^\ve (r_b^\ve)\|^p_{L^4(\Gamma_\tau)}   + 
 C_2 p^4\left( \sup\limits_{(0,\tau)} \| |\T^\ve (r_f^\ve)|^{\frac p2}\|_{L^1(\Gamma)}^{2} + 1\right).
 \end{aligned}
 \end{equation}
Then, the Alikakos iteration  Lemma~\cite{Alikakos_1979} ensures that for $x\in \Omega$
\begin{equation}\label{rf_bound}
\begin{aligned} 
& \|\T^\ve( r_f^\ve) \|_{L^\infty(0,T; L^\infty(\Gamma))} \leq C_1\left[1+ \|\T^\ve (r_b^\ve) \|_{L^4(\Gamma_T)} \right]\leq C_2.
\end{aligned}
\end{equation}
 The definition of the unfolding operator and the fact that $C_2$ is independent of $x\in \Omega$ yields the boundedness of $r_f^\ve$ in $(0,T)\times \Gamma^\ve$.   
Due to the structure of the reaction terms,  in the same way as for $\T^\ve(r_f^\ve)$ we  obtain 
\begin{equation}\label{pd_bound}
\begin{aligned} 
& \|\T^\ve( p_d^\ve) \|_{L^\infty(0,T; L^\infty(\Gamma))} \leq C_1\left[1+ \|\T^\ve (p_a^\ve) \|_{L^4(\Gamma_T)} \right] \leq C _2, 
\end{aligned}
\end{equation}
for $x\in \Omega$.
To show the boundedness of $c^\ve_e$ we consider   $|c^\ve_e|^{p-1}$, for $p\geq 4$,  as a test function in the first equation of \eqref{weak_sol_1}  and, using the assumptions on $F_e$ and the nonnegativity of $r_f^\ve$ and $c^\ve_e$ we obtain 
\begin{equation}\label{estim_c_p}
\begin{aligned} 
&  \|c^\ve_e(\tau)\|^p_{L^p(\Omega_e^\ve)} +4 \frac{p-1}{p} \|\nabla |c^\ve_e|^{\frac p 2}\|^2_{L^2(\Omega_{e, \tau}^\ve)} 
\leq  C_1 p \, [1+ \|c^\ve_e\|^p_{L^p(\Omega_{e, \tau}^\ve)} ]  
\\ 
 & +    C_\delta p^3   \|c^\ve_e\|^{p}_{L^p(\Omega_{e, \tau}^\ve)} +  \delta\frac{p-1}p  \|\nabla |c^\ve_e|^{\frac p 2} |^{2}_{L^2(\Omega_{e, \tau}^\ve)} +   \int_{\Omega_\tau}  \|\T^\ve(r^\ve_b)\|^p_{L^4(\Gamma)} dxdt  .
 \end{aligned}
\end{equation}
Using the boundedness of $\T^\ve(r_f^\ve)$  and taking $ |\T^\ve(r_b^\ve)|^3$ as a test function in the equation for $\T^\ve(r_b^\ve)$ (see \eqref{eq:unfold_boundary} in Appendix \ref{boundedness_sketch})  yields
\begin{equation}\label{estim_rbL41}
\begin{aligned} 
\| \T^\ve(r_b^\ve)(\tau) \|_{L^4(\Gamma)} + \|\nabla_{\Gamma, y} |\T^\ve(r_b^\ve)|^2 \|^{\frac 12}_{L^2(\Gamma_\tau)}
 \leq  C[1+ \|\T^\ve(c_e^\ve)\|_{L^p(\Gamma_\tau)}], 
\end{aligned} 
\end{equation}
for $\tau \in (0, T]$ and $x\in \Omega$. Combining the estimates  \eqref{estim_c_p} and \eqref{estim_rbL41} with a trace inequality and a Gagliardo-Nirenberg inequality, applied to the extension of $c^\ve_e$ from $\Omega_e^\ve$ into $\Omega$, (see Appendix~\ref{boundedness_sketch} for more details) yields  
  \begin{equation}\label{estim_cp_2}
  \begin{aligned} 
 \||c^\ve_e(\tau) |^{\frac p2}\|^2_{L^2(\Omega_e^\ve)} + \ \|\nabla |c^\ve_e|^{\frac p 2}\|^2_{L^2(\Omega_{e, \tau}^\ve)} 
\leq   C_\delta p^{8}  \big[\sup\limits_{(0,\tau)}\||c^\ve_e|^{\frac p2}\|^2_{L^1(\Omega_{e}^\ve)}+1\big].
\end{aligned}
\end{equation}
Then the iteration over $p$, similar to \cite{Alikakos_1979}, yields the boundedness of $c^\ve_e$ in $\Omega^\ve_{e,T}$. Since $c^\ve_e \in L^2(0,T; H^1(\Omega_e^\ve))$ we also have the boundedness of $c^\ve_e$ on  $(0,T)\times \Gamma^\ve$, see e.g.\ \cite{Ptashnyk_2012}. 

To  show boundedness of $r_b^\ve$  we consider $|\T^\ve(r_b^\ve)|^{p-1}$  as a test function in the equation for $\T^\ve(r_b^\ve)$  (equation  \eqref{eq:unfold_boundary} in Appendix \ref{boundedness_sketch})  and using the boundedness  of $\T^\ve(r_f^\ve)$ we obtain 
$$
\begin{aligned} 
\|\T^\ve(r_b^\ve(\tau))\|^p_{L^p(\Gamma)} + 4 \frac{ p-1}p\|\nabla |\T^\ve(r_b^\ve)|^{\frac p 2} \|^2_{L^2(\Gamma_\tau)}
\leq C_1p^4[1+\sup\limits_{(0,\tau)} \||\T^\ve(r_b^\ve)|^{\frac p2}\|^2_{L^1(\Gamma)} ] \\
 + C_2\big[\|\T^\ve(c_e^\ve)\|^p_{L^4(\Gamma_\tau)}+  \|\T^\ve(p_a^\ve)\|^p_{L^4(\Gamma_\tau)}\big],  
\end{aligned}
$$
for $x\in \Omega$ and $\tau \in (0, T]$. 
The  iteration over $p$,  boundedness of $c^\ve_e$,  and estimate \eqref{estim_GN_rpc} for $\T^\ve(p_a^\ve)$ ensure  boundedness of $\T^\ve(r_b^\ve)$ in $\Omega_T\times \Gamma$ and hence the boundedness of $r^\ve_b$ in $(0,T)\times \Gamma^\ve$. 
 
 Taking $|\T^\ve( p_a^\ve)|^{p-1}$ as a   test function in the equation for $\T^\ve( p_a^\ve)$, (equation  \eqref{eq:unfold_boundary}  in Appendix \ref{boundedness_sketch})   and using the boundedness of $\T^\ve( p_d^\ve)$  yield
 $$
\begin{aligned} 
\|\T^\ve(p_a^\ve(\tau))\|^p_{L^p(\Gamma)} + 4 \frac{ p-1}p\|\nabla |\T^\ve(p_a^\ve)|^{\frac p 2} \|^2_{L^2(\Gamma_\tau)}
\leq  C_1p^4[1+ \sup\limits_{(0,\tau)}\||\T^\ve(p_a^\ve)|^{\frac p2}\|^2_{L^1(\Gamma)} ] 
\\ +C_2\sup\limits_{(0,\tau)}\||\T^\ve(c_i^\ve)|^{\frac p 2}\|^2_{L^1(Y_{i})} + \delta \|\nabla_{\Gamma, y} |\T^\ve(c_i^\ve)|^{\frac p 2} \|^2_{L^2(Y_{i, \tau})} +  C_3\|\T^\ve(r_b^\ve)\|^p_{L^4(\Gamma_\tau)},    
\end{aligned}
$$
for $x\in \Omega$ and $\tau \in (0, T]$.  
 Similarly considering  $|\T^\ve(c_i^\ve)|^{p-1}$ as a  test function  in the equation for  $\T^\ve(c_i^\ve)$,  (see \eqref{eq:unfolding_c_i} in Appendix~\ref{boundedness_sketch}),  gives
$$
\begin{aligned} 
\|\T^\ve(c_i^\ve(\tau))\|^p_{L^p(Y_i)} + 4 \frac{ p-1}p\|\nabla |\T^\ve(c_i^\ve)|^{\frac p 2} \|^2_{L^2(Y_{i, \tau})}
\leq C_\delta p^4[1+ \sup\limits_{(0,\tau)} \||\T^\ve(c_i^\ve)|^{\frac p 2} \|^2_{L^1(Y_i)} ] \\ + 
\delta \|\nabla_{\Gamma, y} |\T^\ve(c_i^\ve)|^{\frac p 2} \|^2_{L^2(Y_{i, \tau})} 
+C_\delta \sup\limits_{(0,\tau)}\||\T^\ve(p_a^\ve)|^{\frac p 2}\|^2_{L^1(\Gamma)} + \delta \|\nabla |\T^\ve(p_a^\ve)|^{\frac p 2} \|^2_{L^2(\Gamma_\tau)}, 
\end{aligned}
$$
for $x\in \Omega$ and $\tau \in (0,T]$.  Adding the last two inequalities, using the boundedness of  $ \|\T^\ve(r_b^\ve)\|_{L^4(\Gamma_\tau)} \leq C$, for $x\in\Omega$ and $\tau \in (0, T]$, and iterating over $p$  we obtain the boundedness of $\T^\ve(p_a^\ve)$ in $\Omega_T\times \Gamma$  and of   $\T^\ve(c_i^\ve)$ in $\Omega_T \times Y_i$. This also ensures the boundedness of $p_a^\ve$ in $(0,T)\times \Gamma^\ve$ and of $c_i^\ve$ in $(0,T)\times \Omega_i^\ve$ and  $(0,T)\times \Gamma^\ve$.  
\end{proof}

\begin{lemma}\label{lem:uniqueness}
The solution to the microscopic problem~\eqref{main1}--\eqref{bc_cond} is unique.
\end{lemma}
\begin{proof}
Uniqueness follows from standard arguments by taking the difference of two solutions and using the boundedness of solutions, shown in Lemma~\ref{lem:boundedness},   together with the local Lipschitz continuity of the nonlinear reaction terms.  
\end{proof}
}

\section{Convergence results and derivation of macroscopic equations} \label{sec:macro_model}
\changes{In this section, we use the a priori estimates of Theorem~\ref{theorem_exist} to deduce  the  convergence upto a subsequence of  solutions of the microscopic problem  \eqref{main1}--\eqref{bc_cond} to solutions of a limiting two-scale problem. We make use of the theory of two-scale convergence to pass to the limit and the necessary definitions and results are stated in Appendix \ref{appendix_A1}}

\changes{In the convergence results stated below we  consider the $H^1$-extension of  $c_e^\ve$ from $\Omega_e^\ve$ into $\Omega$, which is well defined due to the assumptions on the geometry of $\Omega^\ve_e$, see e.g.\ Remark~\ref{rem:extension} or  \cite{Acerbi_1992, Cioranescu_II_1999} and  we identify $c^\ve_e$ with this extension. By $[c^\ve_e]^\sim$ we will denote the extension of $c^\ve_e$ by zero from $\Omega_e^\ve$ into $\Omega$ and  by $\chi_{Y_e}$ the characteristic function of $Y_e$.   The space  
 $H^1_{\rm per}(Y)$ is defined as  the closure of $C^1_{\rm per}(Y)$ with respect to the $H^1$-norm.}

\begin{lemma}\label{lem_converge}
There exist  functions $c_e \in L^2(0,T; H^1(\Omega))$,  $c_e^1 \in L^2(\Omega_T; H^1_{\rm per}(Y)) $, $c_i \in L^2(\Omega_T; H^1(Y_i))$ and
$r_l, p_s \in L^2(\Omega_T; H^1(\Gamma))$, with $l=f,b$ and $s=a,d$, such that, up to a subsequence,  
\begin{equation}\label{convergence_1}
\begin{aligned} 
& c^\ve_e   \rightharpoonup c_e && \text{weakly  in }  L^2(0,T; H^1(\Omega)), \\
& \nabla c^\ve_e   \rightharpoonup \nabla c_e + \nabla_y c_e^1  && \text{two-scale},\\
& \changes{[c^\ve_e]^\sim \rightharpoonup c_e \chi_{Y_e},\; \;  [ \nabla c^\ve_e]^\sim  \rightharpoonup (\nabla c_e + \nabla_y c_e^1) \chi_{Y_e}}   && \text{two-scale},\\
& c_i^\ve   \rightharpoonup  c_i, \; \;  \ve \nabla c_i^\ve  \rightharpoonup  \nabla_{y} c_i && \text{two-scale}, \\
& r^\ve_l    \rightharpoonup  r_l, \; \;  \ve \nabla_{\Gamma} r_l^\ve  \rightharpoonup  \nabla_{\Gamma, y} r_l && \text{two-scale},\quad l = f,b,  \\
& p^\ve_s    \rightharpoonup  p_s, \; \;  \ve \nabla_{\Gamma} p_s^\ve  \rightharpoonup  \nabla_{\Gamma, y} p_s && \text{two-scale}, \quad s = a,d. 
\end{aligned}
\end{equation}
\end{lemma}
\begin{proof} 
The convergence results in \eqref{convergence_1} follow directly from  the a priori estimates  \eqref{estim:apriori_1},  the extension of $c_e^\ve$ from $\Omega_e^\ve$ into $\Omega$   and compactness theorems for the weak convergence and for the two-scale convergence, see e.g.~\cite{Allaire_1992, Allaire_1996,  Neuss-Radu_1996, Nguetseng_1989} \changes{ and Appendix \ref{appendix_A1}.   Notice that since the  extension of $c^\ve_e$ and $[c^\ve_e]^\sim$ coincide in $\Omega^\ve_e$, we obtain the same function $c_e$  in the two-scale limit for both sequences $\{c^\ve_e\}$ and $\{[c^\ve_e]^\sim\}$.  }
\end{proof}

In order to pass to the limit in nonlinear reaction terms we prove strong convergence upto a subsequence of  solutions of the microscopic problem \eqref{main1}--\eqref{bc_cond}. 
\begin{lemma}\label{lem:strong_converg}
For a subsequence of a sequence of solutions of microscopic model \eqref{main1}--\eqref{bc_cond}, i.e.\  $\{c^\ve_e\}$, $\{\T^\ve(c^\ve_i)\}$,   $\{\T^\ve(r^\ve_l)\}$,  and $\{\T^\ve(p^\ve_s)\}$, where  $l=f,b$ and $s=a,d$,  we have the following convergence results:
\begin{equation}\label{conver_strong}
\begin{aligned} 
c^\ve_e  &  \to c_e && \text{strongly  in }  L^2(\Omega_T), \qquad \ve\|c_e^\ve - c_e\|^2_{L^2(\Gamma_T^\ve)} \to 0, \\
 \T^\ve(c_e^\ve)  & \to c_e && \text{strongly  in }  L^2(\Omega_T\times Y_e)  \text{ and  }   L^2(\Omega_T\times \Gamma), \\
 \T^\ve(c_i^\ve)  & \to c_i && \text{strongly  in }  L^2(\Omega_T\times Y_i), \\
 \T^\ve(r_l^\ve)  & \to r_l && \text{strongly  in }  L^2(\Omega_T\times \Gamma), \quad l = f,b\\
 \T^\ve(p_s^\ve)  & \to p_s && \text{strongly  in }  L^2(\Omega_T\times \Gamma), \quad s=a,d,
\end{aligned}
\end{equation}
as $\ve \to 0$. 
\end{lemma}
\begin{proof} 
We first show the equicontinuity of $c^\ve_e$ with respect to the time variable. The a priori estimates in \eqref{estim:apriori_1} and the boundedness of $r_f^\ve$ yield
\begin{equation*}
\begin{aligned} 
\|\vartheta_\delta c^\ve_e- c_e^\ve\|^2_{L^2(\Omega_{e,\tau}^\ve)} \leq 
& C_1 \int_{\Omega^\ve_{e, \tau}} \int_t^{t+\delta} |\nabla c^\ve_e| ds  |\nabla (\vartheta_\delta c^\ve_e- c^\ve_e)| dx dt
\\ & +  C_2\int_{\Omega^\ve_{e, \tau}} \int_t^{t+\delta}(1+ |c^\ve_e|) ds  |(\vartheta_\delta c^\ve_e- c^\ve_e)| dx dt
\\ & + \ve C_3\int_{\Gamma^\ve_{\tau}} \int_t^{t+\delta}\left( |c^\ve_e r_f^\ve|+ |r_b^\ve|\right) ds  |\vartheta_\delta c^\ve_e- c^\ve_e| d\sigma^\ve dt
\leq C \delta, 
\end{aligned}
\end{equation*}
for $\tau \in (0, T- \delta]$ and $\delta >0$, where $\vartheta_\delta v(t,x) = v(t+\delta, x)$ for $x\in \Omega$ and $t \in [0, T-\delta]$.  Then the properties of an extension of $c^\ve_e$ from $\Omega_e^\ve$ into $\Omega$ together with the uniform in $\ve$ estimate for $\nabla c^\ve_e$ and a Kolmogorov compactness  result \cite{Brezis_2011} ensure the strong convergence of $c^\ve_e$ in $L^2((0,T)\times \Omega)$. Applying the Simon compactness theorem \cite{Simon_1986} and the compact embedding of $H^1(\Omega)$ into $H^\beta(\Omega)$ for $1/2 < \beta <1$,  together with the trace inequality and  a scaling argument, similar to \cite{Ptashnyk_2008},  we also obtain  $\ve\|c_e^\ve - c_e\|^2_{L^2(\Gamma_T^\ve)} \to 0$ as $\ve \to 0$. 
\\
\changes{The properties of the unfolding operator, see \cite{Cioranescu_2012} and Appendix \ref{appendix_A1}, imply
$$\|  \T^\ve(c_e^\ve) \|_{L^2(\Omega_T\times Y_e) } \leq  |Y|^{\frac 12}  \|  c_e^\ve \|_{L^2(\Omega_{e,T}^\ve) }, \; \;  
\|  \T^\ve(c_e^\ve) \|_{L^2(\Omega_T\times \Gamma) } \leq  |Y|^{\frac 12}  \ve^{\frac 12} \|  c_e^\ve \|_{L^2(\Gamma_T^\ve)}, $$
 and  for $c_e\in L^2(\Omega_T)$, considered as constant with respect to $y \in Y_e$ or $y \in \Gamma$, respectively, we have 
 $\T^\ve(c_e) \to c_e$ strongly in $L^2(\Omega_T\times Y_e)$ and in $L^2(\Omega_T\times \Gamma)$ as $\ve \to 0$.   Then we obtain  
\begin{equation*}
\begin{aligned} 
\|  \T^\ve(c_e^\ve)   - c_e \|_{L^2(\Omega_T\times Y_e) } & \leq 
\|  \T^\ve(c_e^\ve)   - \T^\ve(c_e) \|_{L^2(\Omega_T\times Y_e) }   + \|  \T^\ve(c_e)   - c_e \|_{L^2(\Omega_T\times Y_e) } 
\\  & \leq  |Y|^{\frac 12} \|  c_e^\ve   - c_e \|_{L^2(\Omega_{e,T}^\ve) }  + \|  \T^\ve(c_e)   - c_e \|_{L^2(\Omega_T\times Y_e) }, \\
\|  \T^\ve(c_e^\ve)   - c_e \|_{L^2(\Omega_T\times \Gamma) } 
&  \leq \ve^{\frac 12}  |Y|^{\frac 12} \|  c_e^\ve   - c_e \|_{L^2(\Gamma^\ve_T) }  + \|  \T^\ve(c_e)   - c_e \|_{L^2(\Omega_T\times \Gamma) }.
\end{aligned}
\end{equation*}
Hence the strong convergence of $\{c_e^\ve\}$ in $L^2(\Omega_T)$ and the convergence result for $\ve\|c_e^\ve - c_e\|^2_{L^2(\Gamma_T^\ve)}$ in  \eqref{conver_strong} ensure the strong convergence of $\{ \T^\ve(c_e^\ve)\}$ in $L^2(\Omega_T\times Y_e)$ and in $L^2(\Omega_T\times \Gamma)$, respectively. 
}

To obtain  the strong convergence of $\mathcal T^\ve(c_i^\ve)$,  $\mathcal T^\ve(r^\ve_l)$, and $\mathcal T^\ve(p^\ve_s)$, with $l=f,b$ and $s=a,d$,  we show the Cauchy property for the corresponding sequences. 
Considering the difference of equations for $\ve_m$ and $\ve_k$ and using  the boundedness of $c^\ve_e$, $r_f^\ve$,  and $r_b^\ve$  yields 
\begin{equation*}
\begin{aligned} 
\sum_{l=f,b} \|\T^{\ve_m}(r_l^{\ve_m}(\tau)) - \T^{\ve_k}(r_l^{\ve_k}(\tau))\|^2_{L^2(\Omega\times \Gamma)}
+ \|\nabla_{\Gamma, y} (\T^{\ve_m}(r_l^{\ve_m}) - \T^{\ve_k}(r_l^{\ve_k}))\|^2_{L^2(\Omega_\tau\times \Gamma)}
\\
\leq C_1\Big[  \|\T^{\ve_m}(c_e^{\ve_m}) - \T^{\ve_k}(c_e^{\ve_k})\|^2_{L^2(\Omega_\tau\times \Gamma)} + \sum_{l=f,b} \|\T^{\ve_m}(r_l^{\ve_m}) - \T^{\ve_k}(r_l^{\ve_k})\|^2_{L^2(\Omega_\tau\times \Gamma)}\Big] \\
 +  C_2\Big[  \sum_{j=a,d} \|\T^{\ve_m}(p_j^{\ve_m}) - \T^{\ve_k}(p_j^{\ve_k})\|^2_{L^2(\Omega_\tau\times \Gamma)}  
+\sum_{l=f,b} \|\T^{\ve_m}(r_{l,0}^{\ve_m}) - \T^{\ve_k}(r_{l,0}^{\ve_k})\|^2_{L^2(\Omega\times \Gamma)} \Big], 
\end{aligned}
\end{equation*}
for $\tau \in (0,T]$. 
Similarly, the boundedness of $p_d^\ve$ yields 
\begin{equation*}
\begin{aligned} 
\sum_{j=a,d}\|\T^{\ve_m}(p_j^{\ve_m}(\tau)) - \T^{\ve_k}(p_j^{\ve_k}(\tau))\|^2_{L^2(\Omega\times \Gamma)}
+ \|\nabla_{\Gamma, y} (\T^{\ve_m}(p_j^{\ve_m}) - \T^{\ve_k}(p_j^{\ve_k}))\|^2_{L^2(\Omega_\tau\times \Gamma)} \; 
\\ \leq C_1 \sum_{j=a,d}\Big[ \|\T^{\ve_m}(p_j^{\ve_m}) - \T^{\ve_k}(p_j^{\ve_k})\|^2_{L^2(\Omega_\tau\times \Gamma)} + \|\T^{\ve_m}(p_{j,0}^{\ve_m}) - \T^{\ve_k}(p^{\ve_k}_{j, 0})\|^2_{L^2(\Omega\times \Gamma)}\Big] 
\\ + 
 C_2\Big[ \|\T^{\ve_m}(r_b^{\ve_m}) - \T^{\ve_k}(r_b^{\ve_k})\|^2_{L^2(\Omega_\tau\times \Gamma)} +  \|\T^{\ve_m}(c_i^{\ve_m}) - \T^{\ve_k}(c_i^{\ve_k})\|^2_{L^2(\Omega_\tau\times \Gamma)} \Big]. 
\end{aligned}
\end{equation*}
For $\T^\ve(c_i^\ve)$ the trace inequality implies 
\begin{equation*}
\begin{aligned} 
\|\T^{\ve_m}(c_i^{\ve_m}(\tau)) - \T^{\ve_k}(c_i^{\ve_k}(\tau))\|^2_{L^2(\Omega\times Y_i)}
+ \|\nabla_{y} (\T^{\ve_m}(c_i^{\ve_m}) - \T^{\ve_k}(c_i^{\ve_k}))\|^2_{L^2(\Omega_\tau\times Y_i)} \quad 
\\ \leq  C_\delta \|\T^{\ve_m}(c_i^{\ve_m}) - \T^{\ve_k}(c_i^{\ve_k})\|^2_{L^2(\Omega_\tau\times Y_i)}  
+ \delta  \|\nabla_y(\T^{\ve_m}(c_i^{\ve_m}) - \T^{\ve_k}(c_i^{\ve_k}))\|^2_{L^2(\Omega_\tau\times Y_i)}  
\\ +
C_1\Big[ \|\T^{\ve_m}(p_a^{\ve_m}) - \T^{\ve_k}(p_a^{\ve_k})\|^2_{L^2(\Omega_\tau\times \Gamma)} + \|\T^{\ve_m}(c_{i,0}^{\ve_m}) - \T^{\ve_k}(c^{\ve_k}_{i, 0})\|^2_{L^2(\Omega\times Y_i)}\Big] .
\end{aligned}
\end{equation*}
Using the three estimates above, the strong convergence  of the initial conditions and the strong convergence of $\{\T^\ve (c_e^\ve)\}$ in $L^2(\Omega_T\times \Gamma)$  we obtain  the Cauchy property  and  hence the strong convergence upto  subsequences of $\{\mathcal T^\ve(c_i^\ve)\}$,  $\{\mathcal T^\ve(r^\ve_l)\}$, and $\{\mathcal T^\ve(p^\ve_s)\}$, with $l=f,b$ and $s=a,d$. 
\end{proof} 

The convergence results in Lemmata~\ref{lem_converge} and \ref{lem:strong_converg} allow us to derive the  corresponding macroscopic equations obtained in the limit as $\ve\to0$  from the  microscopic model \eqref{main1}--\eqref{bc_cond}. 
\begin{theorem} 
A sequence  $\{c_e^\ve, c_i^\ve, r_f^\ve, r_b^\ve, p_a^\ve, p_d^\ve\}$ of solutions  of   \eqref{main1}--\eqref{bc_cond}  converge as $\ve \to 0$ to functions
 $c_e \in \Lp{2}(0,T;\Hil{1}(\O))$, $c_i\in\Lp{2}\left(0,T;\Lp{2}\left(\O;\Hil{1}(Y_i)\right)\right)$ and 
$r_l, p_s \in \Lp{2}\left(0,T;\Lp{2}\left(\O;\Hil{1}(\G)\right)\right)$,    for $l=f,b$, $s=a,d$,  that satisfy  the following macroscopic equations:
 \begin{equation} \label{eq:hom_bulk} 
\begin{aligned}
\theta_e \partial_t c_e - \nabla\cdot ( D^{\rm hom}_e(x) \nabla c_e) & = \theta_e F_e(c_e)  - \frac 1 {|Y|} \int_\Gamma  G_e(c_e, r_f, r_b)  d{\sigma_y} && \text{in } \Omega, \\
D^{\rm hom}_e(x) \nabla c_e \cdot \nu &= 0 && \text{ on } \partial \Omega, \\ 
\partial_t c_i - \nabla_y\cdot ( D_i(y) \nabla_y c_i) &= F_i(c_i)  && \text{ in } \Omega\times Y_i, \\
D_i(y) \nabla_y c_i \cdot \nu &= G_i(p_a, c_i) && \text{ on } \Omega \times\Gamma, 
\end{aligned}
\end{equation}
where $\theta_e = |Y_e|/|Y|$,  and $$D_{e,ij}^{\rm hom}(x)= \dfrac 1{|Y|} \int_{Y_e} \big[ D_{e, ij} (x,y) + \big(D_{e}(x,y) \nabla_y w^j(y)\big)_i \big] dy, $$  with $w^j$ being solutions of the unit cell problems
$$
\begin{aligned}
&{\rm div}_y ( D_e(x,y)(\nabla_y w^j + e_j)) = 0 && \text{ in } Y_e \times\O, \quad \int_{Y_e} w^j(x,y) dy = 0, \\
&  D_e(x,y)(\nabla_y w^j + e_j) \cdot \nu = 0 && \text{ on } \Gamma \times\O, \quad w^j(x, \cdot) \; \; Y-\text{periodic}, 
\end{aligned}
$$
for $x\in \Omega$, where $\{ e_j\}_{j=1, \ldots, d}$ is the standard basis in $\mathbb R^d$,   together with the dynamics of receptors and proteins on the cell membrane $\Omega\times\Gamma$
\begin{equation}\label{eq:hom_surf} 
\begin{aligned}
& \partial_t r_f - \nabla_{\Gamma,y} \cdot (D_{f} \nabla_{\Gamma, y} r_f) =  F_f(r_f, r_b)-  G_e(c_e, r_f, r_b)  - d_f r_f, \\
& \partial_t r_b - \nabla_{\Gamma,y}\cdot (D_{b} \nabla_{\Gamma, y} r_b) =  \phantom{  f_r(r_f)+}  G_e(c_e, r_f, r_b) - G_d(r_b, p_d, p_a)  - d_b  r_b, \\
& \partial_t p_d - \nabla_{\Gamma,y} \cdot (D_{d} \nabla_{\Gamma, y} p_d) =  F_d(p_d) -  G_d(r_b, p_d, p_a) - d_d p_d, \\
& \partial_t p_a - \nabla_{\Gamma,y} \cdot (D_{a} \nabla_{\Gamma, y} p_a) = \phantom{ f_p(p_d) }   G_d(r_b, p_d, p_a)  -  G_i(p_a, c_i)- d_a p_a, 
\end{aligned}
\end{equation}
and initial conditions 
\begin{equation}\label{intial_cond_macro} 
\begin{aligned}
& c_e(0,x) = c_{e,0}(x)\;  \text{ for } x \in \Omega, \quad c_i(0,x,y) = c_{i,1}(x)c_{i,2}(y)  \text{ for } x \in \Omega, \; y \in Y_i, \\
& r_j(0, x,y) = r_{j,1}(x) r_{j,2}(y), \quad p_s(0,x,y) = p_{s,1}(x) p_{s,2} (y) \; \text{ for } x \in \Omega, \; y \in \Gamma,    
\end{aligned}
\end{equation}
where $j= f,b$ and  $s = a,d$.
\end{theorem} 

\begin{proof}  To derive the macroscopic problem   take  $\phi^\ve(t,x)= \phi_1(t,x) + \ve \phi_2(t, x, x/\ve)$, where $\phi_1 \in H^1(\Omega_T)$ and  $\phi_2 \in C^1_0(\Omega_T; C_{\rm per}(Y))$,   and 
$\psi^\ve(t,x) = \psi_1(t,x, x/\ve)$, with  $\psi_1 \in C^1([0, T]; C^1_0(\Omega; H^1(Y_i)))$,  as test functions in \eqref{weak_sol_1} and $\varphi^\ve(t,x) = \varphi_1(t,x,x/\ve)$, with $\varphi_1\in C^1([0,T]; C^1_0(\Omega; H^1(\Gamma)))$,  as a test function in \eqref{weak_sol_2}, respectively, where $\psi_1$ and $\varphi_1$ are $Y$-periodically extended to $\mathbb R^d$. Then we obtain 
\changes{
\begin{eqnarray} \label{converge_1}
 \la \partial_t c^\ve_e, \phi_1 + \ve \phi_2  \ra + \la D^\ve_e(x) \nabla c^\ve_e, \nabla (\phi_1 + \ve \phi_2) + \nabla_y \phi_2  \ra_{\Omega_{e,T}^\ve} 
= \la  F_e(c^\ve_e), \phi_1 + \ve \phi_2 \ra_{\Omega_{e,T}^\ve}  \quad \nonumber\\
  - \ve \la G_e (c^\ve_e, r_f^\ve, r_b^\ve), \phi_1  + \ve \phi_2\ra_{\Gamma_{T}^\ve},  \\
 \la \partial_t c^\ve_i, \psi_1 \ra +  \la \ve D^\ve_i(x) \nabla c^\ve_i, \ve \nabla \psi_1 +  \nabla_y \psi_1 \ra_{\Omega_{i,T}^\ve} 
= \la  F_i(c^\ve_i), \psi_1 \ra_{\Omega_{i,T}^\ve}  + \ve \la G_i (p_a^\ve, c^\ve_i), \psi _1\ra_{\Gamma_{T}^\ve} , \nonumber
\end{eqnarray}
and 
\begin{equation} \label{converge_2}
\begin{aligned}
\ve  \la \partial_t r^\ve_f, \varphi_1 \ra+ \ve  \la  \ve D_f \nabla_\Gamma r^\ve_f, \, & \ve \nabla_\Gamma  \varphi_1 + \nabla_{\Gamma, y}  \varphi_1  \ra_{\Gamma_{T}^\ve} 
\\ & = \ve  \la F_f(r^\ve_f, r^\ve_b) - G_e (c^\ve_e, r_f^\ve, r_b^\ve) - d_f  r_f^\ve, \varphi_1  \ra_{\Gamma_{T}^\ve} , \\
\ve  \la \partial_t r^\ve_b, \varphi_1 \ra +\ve  \la \ve D_b \nabla_\Gamma  r^\ve_b,  \, & \ve \nabla_\Gamma  \varphi + \nabla_{\Gamma, y}  \varphi_1 \ra_{\Gamma_{T}^\ve} 
\\ & =\ve  \la G_e (c^\ve_e, r_f^\ve, r_b^\ve) - G_d ( r_b^\ve, p_d^\ve, p_a^\ve) - d_b  r_b^\ve, \varphi_1 \ra_{\Gamma_{T}^\ve} ,\\
 \ve \la \partial_t p^\ve_d, \varphi_1 \ra + \ve \la \ve D_d \nabla_\Gamma  p^\ve_d, \, & \ve \nabla_\Gamma  \varphi_1 + \nabla_{\Gamma, y}  \varphi_1 \ra_{\Gamma_{T}^\ve} 
\\  & =\ve  \la F_d(p^\ve_d) - G_d ( r_b^\ve, p_d^\ve, p_a^\ve) - d_d  p_d^\ve, \varphi_1 \ra_{\Gamma_{T}^\ve} , \\
 \ve \la \partial_t p^\ve_a, \varphi_1 \ra + \ve \la \ve D_a \nabla_\Gamma  p^\ve_a, \, & \ve\nabla_\Gamma  \varphi_1 +  \nabla_{\Gamma, y}  \varphi_1 \ra_{\Gamma_{T}^\ve} 
\\  &= \ve \la G_d ( r_b^\ve, p_d^\ve, p_a^\ve) - G_i(p_a^\ve, c_i^\ve)- d_a  p_a^\ve, \varphi_1 \ra_{\Gamma_{T}^\ve}.
\end{aligned} 
\end{equation}
Considering  first  $\phi_1 \in H^1(\Omega_T)$ with $\phi_1(0,x) = 0$ and  $\phi_1(T,x) =0$ for $x\in \Omega$ and using integration by parts and  the two-scale converge results, see Lemma~\ref{lem_converge}, yield
$$
\begin{aligned} 
\lim\limits_{\ve \to 0}  \la \partial_t c^\ve_e, \phi_1 + \ve \phi_2 \ra =  - \lim\limits_{\ve \to 0}  \la  c^\ve_e, \partial_t \phi_1 + \ve \partial_t \phi_2 \ra_{\Omega_{e, T} ^\ve} 
= - |Y|^{-1} \la  c_e, \partial_t \phi_1  \ra_{Y_e\times \Omega_{T}}.  
\end{aligned} 
$$
Similar calculations ensure convergence  of $ \la \partial_t c^\ve_i, \psi_1 \ra $, $ \ve \la \partial_t r^\ve_l, \varphi_1 \ra$, and $\ve \la \partial_t p^\ve_s, \varphi_1 \ra$, with $l=f,b$ and $s=a,d$.   The two-scale convergence results, see Lemma~\ref{lem_converge},  directly imply  
$$
\begin{aligned}
& \lim\limits_{\ve \to 0} \la D^\ve_e(x) \nabla c^\ve_e, \nabla (\phi_1 + \ve \phi_2) + \nabla_y \phi_2  \ra_{\Omega_{e, T}^\ve}  \\ & \qquad \qquad \qquad = 
\frac 1 {|Y|} \la D_e(x,y) (\nabla c_e + \nabla_y c_e^1), \nabla \phi_1  + \nabla_y \phi_2  \ra_{Y_e\times \Omega_{T}}\\
&\lim\limits_{\ve \to 0}  \la \ve D^\ve_i(x) \nabla c^\ve_i, \ve \nabla \psi_1 +  \nabla_y \psi_1 \ra_{\Omega_{i,T}^\ve}  = \frac 1 {|Y|}  \la D_i(y) \nabla_y c_i, \nabla_y \psi_1 \ra_{Y_i\times \Omega_{T}}, \\
 & \lim\limits_{\ve \to 0} \ve  \la  \ve D_f \nabla_\Gamma r^\ve_f, \,  \ve \nabla_\Gamma  \varphi_1 + \nabla_{\Gamma, y}  \varphi_1  \ra_{\Gamma_{T}^\ve} =
 \frac 1{|Y|}  \la   D_f \nabla_{\Gamma, y} r_f, \,   \nabla_{\Gamma, y}  \varphi_1  \ra_{\Gamma\times \Omega_T}, 
\end{aligned}
$$
and convergence of  the linear  term  $\ve \la d_f r_f^\ve, \varphi_1  \ra_{\Gamma_{T}^\ve} \to |Y|^{-1} \la d_f r_f, \varphi_1  \ra_{\Gamma \times\Omega_T}$ as $\ve \to 0$. 
The convergence of the corresponding terms in equations for $r_b^\ve$, $p_a^\ve$, and $p_d^\ve$ is obtained in the same way.

To pass to the limit in the nonlinear reaction terms we use the strong convergence results proven in  Lemma~\ref{lem:strong_converg}. 
The definition and properties of the unfolding operator (c.f., Appendix \ref{appendix_A1}), together with the assumptions on functions $F_l$ and $G_j$, for $l=e,i,f, d$ and $j=e,d, i$ and the boundedness of solutions of the microscopic problem  \eqref{main1}--\eqref{bc_cond},  imply 
$$
\begin{aligned}
\|F_l(\mathcal T^\ve( c_l^\ve))\|_{L^2(\Omega_T\times Y_l)} \leq |Y|^{\frac 12 } \|F_l(c_l^\ve)\|_{L^2(\Omega_{l,T}^\ve)}\leq C_1(1+ \|c_l^\ve\|_{L^2(\Omega_{l,T}^\ve)}) \leq C_2, \\
\|G_e(\mathcal T^\ve( c_e^\ve), \mathcal T^\ve(r_f^\ve), \mathcal T^\ve(r_b^\ve))\|_{L^2(\Omega_T\times \Gamma)} \leq 
\ve^{\frac 12}|Y|^{\frac 12 } \|G_e( c_e^\ve, r_f^\ve, r_b^\ve)\|_{L^2(\Gamma_{T}^\ve)}\leq C. 
\end{aligned}
$$
Here we used the fact that $\mathcal T^\ve(F_l(c_l^\ve)) = F_l(\mathcal T^\ve( c_l^\ve))$, for  $l = e,i$,  $\mathcal T^\ve(G_e( c_e^\ve, r_f^\ve, r_b^\ve) )= 
G_e(\mathcal T^\ve(c_e^\ve), \mathcal T^\ve(r_f^\ve), \mathcal T^\ve(r_b^\ve))$.

Similar estimates hold for  $F_f$, $F_d$, $G_d$ and $G_i$.  Then the strong convergence  of $\{\mathcal T^\ve (c_l^\ve) \}$, $\{\mathcal T^\ve (r_j^\ve)\}$, and $\{\mathcal T^\ve (p_s^\ve)\}$,  with $l=e,i$, $s=a,d$, and   $j=f,b$,  ensures  the following convergence results,  
$\mathcal T^\ve(F_l(c_l^\ve)) \rightharpoonup F_l(c_l)$   in $L^2(\Omega_T\times Y_j)$, for $l=e,i$,   and $\mathcal T^\ve(G_e( c_e^\ve, r_f^\ve, r_b^\ve) ) \rightharpoonup G_e( c_e, r_f, r_b)$,  $\mathcal T^\ve(F_f(r_f^\ve, r_b^\ve)) \rightharpoonup F_f(r_f, r_b)$, $\mathcal T^\ve(F_d(p_d^\ve)) \rightharpoonup F_d(p_d)$, 
$\mathcal T^\ve(G_d(r_b^\ve, p_d^\ve, p_a^\ve))  \rightharpoonup G_d(r_b, p_d, p_a)$, 
  $\mathcal T^\ve(G_i( p_a^\ve, c_i^\ve) ) \rightharpoonup G_i(p_a, c_i)$   in   $L^2(\Omega_T\times \Gamma)$.  
These convergence results together with the properties of the unfolding operator (c.f., Appendix \ref{appendix_A1}) imply the convergence of the nonlinear terms in the microscopic problem.
} 

To complete the proof, we note that standard results for parabolic equations imply 
\begin{align*}
&\text{$\partial_t c_e \in L^2(0,T; H^1(\Omega)^\prime)$,  $\partial_t c_i \in L^2((0,T)\times \Omega; H^1(Y_i)^\prime)$,}\\ 
 &\text{$\partial_t r_l, \, \partial_t p_s \in L^2((0,T)\times \Omega; H^1(\Gamma)^\prime)$,  for $l=f,b$ and $s=d,a$.}
 \end{align*}
Thus $c_e \in C([0,T]; L^2(\Omega))$,  $c_i \in C([0,T]; L^2(\Omega\times Y_i))$, and $r_j, p_s \in C([0,T]; L^2(\Omega\times \Gamma))$,  for $j=f,b$ and $s=d,a$. Taking   $\phi_1$,  $\psi_1$, and $\varphi_1$ such that $\phi_1(T,x) =0$, $\psi_1(T,x, y) = 0$, and $\varphi_1(T,x, z) = 0$ for $x\in \Omega$, $y \in Y_i$, $z\in \Gamma$ and using the strong two-scale convergence of $c_{i, 0}^\ve(x)$ to $c_{i,1}(x) c_{i,2}(y)$, of   $r_{j,0}^\ve(x)$ to $r_{j,1}(x) r_{j,2}(y)$, and of $p_{s,0}^\ve(x)$ to  $p_{s,1}(x) p_{s,2}(y)$, for $j=f,b$ and  $s=a,d$, we deduce  that the initial conditions  \eqref{intial_cond_macro} are satisfied.
\end{proof}

To design a multiscale numerical scheme for the macroscopic two-scale problem \eqref{eq:hom_bulk}--\eqref{intial_cond_macro} we define our notion of  weak solutions to the problem. 
\begin{definition}
Weak solutions of the macroscopic   problem  \eqref{eq:hom_bulk}--\eqref{intial_cond_macro} are functions
\begin{align*}
&c_e\in\Lp{2}(0,T;\Hil{1}(\O))\text{ with }\pdt c_e\in  L^2(0, T; \Hil{1}(\O)^\prime),\\
&c_i\in\Lp{2}\left(0,T;\Lp{2}\left(\O;\Hil{1}(Y_i)\right)\right)\text{ with }\pdt c_e\in  L^2(\Omega_T; \Hil{1}(Y_i)^\prime),\\
&r_l\in\Lp{2}\left(0,T;\Lp{2}\left(\O;\Hil{1}(\G)\right)\right)\text{ with }\pdt r_l\in L^2(\Omega_T; \Hil{1}(\Gamma)^\prime), \quad l=f,b,\\
&p_s\in\Lp{2}\left(0,T;\Lp{2}\left(\O;\Hil{1}(\G)\right)\right)\text{ with }\pdt p_s\in L^2(\Omega_T; \Hil{1}(\Gamma)^\prime), \quad s=a,d,
\end{align*}
such that
\begin{equation}\label{def_macro_weak_1}
\begin{aligned}
 \la \theta_e\pdt c_e, \phi\ra +\Ltwop{D^{\rm hom}_e(x) \nabla c_e}{\nabla \phi}{\O_T}&=\la\theta_eF_e(c_e), \phi \ra_{\O_T}\\
&  -\Ltwop{ \frac 1 {|Y|} \int_\Gamma  G_e(c_e, r_f, r_b)\,  d{\sigma_y} }{\phi}{\O_T},\\
 \la \pdt c_i, \psi \ra +\Ltwop{D_i (y)\nabla_y c_i}{\nabla_y \psi}{Y_i\times\O_T}&=\Ltwop{F_i(c_i)}{\psi}{Y_i\times\O_T}+\Ltwop{G_i(p_a, c_i)}{\psi}{\G\times\O_T},
\end{aligned}
\end{equation}
and 
\begin{equation}\label{def_macro_weak_2}
\begin{aligned}
 \la{\pdt r_f}, {\chi}\ra &+\Ltwop{D_{f} \nabla_{\G} r_f}{\nabla_{\G} \chi}{\G\times\O_T}\\
&=\Ltwop{F_f(r_f, r_b)-  G_e(c_e, r_f, r_b)  - d_f r_f}{\chi}{\G\times\O_T},\\
 \la {\pdt r_b}, {\chi}\ra & +\Ltwop{D_{b} \nabla_{\G} r_b}{\nabla_{\G} \chi}{\G\times\O_T}\\
&=\Ltwop{G_e(c_e, r_f, r_b) - G_d(r_b, p_a, p_d)  - d_b  r_b}{\chi}{\G\times\O_T},\\
 \la \pdt p_d, \chi\ra &+\Ltwop{D_{d} \nabla_{\G} p_d}{\nabla_{\G} \chi}{\G\times\O_T}\\
&=\Ltwop{F_d(p_d) -  G_d(r_b, p_a, p_d) - d_d p_d}{\chi}{\G\times\O_T},\\
 \la \pdt p_a, \chi \ra &+\Ltwop{D_{a} \nabla_{\G} p_a}{\nabla_{\G} \chi}{\G\times\O_T}\\
&=\Ltwop{G_d(r_b, p_a, p_d)  -  G_i(p_a, c_i)- d_a p_a}{\chi}{\G\times\O_T},
\end{aligned}
\end{equation}
for all $\phi\in\Lp{2}(0,T;\Hil{1}(\O)),\psi\in \Lp{2}\left(\O_T;\Hil{1}(Y_i)\right)$, and $\chi\in \Lp{2}\left(\O_T;\Hil{1}(\G)\right)$, where the initial conditions \eqref{intial_cond_macro} are satisfied in the $L^2$-sense. 
\end{definition}

Notice that the coupling between macroscopic  and microscopic  scales is given through $c_e$ in the equations for $r_f$ and $r_b$ and through the reaction term in the equation for $c_e$.

\section{Numerical scheme for the homogenised problem}\label{sec:scheme}
\changes{
In this section we present a robust numerical method for the simulation of the homogenised macroscopic model of Section~\ref{sec:macro_model}, i.e., equations~\eqref{eq:hom_bulk} and \eqref{eq:hom_surf}.  We employ a tensor product finite element approach for the discretisation of the two-scale systems \cite{Schwab_2002}. For the  bulk-surface systems, we employ a piecewise linear  bulk-surface finite element method. The method is based on the coupled bulk-surface finite element method  proposed and analysed (for linear elliptic systems) in \cite{elliott2012finite}. 

We define computational domains $\O_h$, $Y_{h,e}$, $Y_{h,i}$ and $\G_h$ by requiring that $\O_h$, $Y_{h,e}$ and $Y_{h,i}$ are polyhedral approximations to $\O$, $Y_e$ and $Y_i$ respectively and we set $\G_h=\partial Y_{h,i}$, i.e., $\G_h$ is the boundary of the polyhedral domain $Y_{h,i}$. We assume that $\O_h$, $Y_{h,e}$ and $Y_{h,i}$ consist of the union of $d$ dimensional simplices (triangles for $d=2$ and tetrahedra for $d=3$) and hence the faces of $\G_h$ are $d-1$ dimensional simplices.  

 We define $\Th, \Shi, \She$ to be triangulations of $\O_h, Y_{h,i}$ and $Y_{h,e}$ respectively and assume that each consists of  closed non-degenerate simplices. We denote by $h_\O, h_{Y,i}, h_{Y,e}$ and $h_{\G}$ the maximum diameter of the simplices in   $\Th, \Shi, \She$ and $\G_h$ respectively.
 Furthermore, we assume the triangulation is such that for every $k\in\Shi$, $k\cap\G_h$ consists of at most one face of $k$. 
 We define bulk and surface finite element spaces as follows
 \begin{align*}
 \Vho&=\left\{\Phi\in C(\O_h):\Phi\vert_k\in\mathbb{P}^1,\myall k\in\Th\right\},\\
  \vhi&=\left\{\Phi\in C(Y_{h,i}):\Phi\vert_k\in\mathbb{P}^1,\myall k\in\Shi\right\},\\
  \vhee&=\left\{\Phi\in C(Y_{h,e}):\Phi\vert_k\in\mathbb{P}^1,\myall k\in\She\right\},\\
 \vhe&=\left\{\Phi\in \Hil{1}_{\rm{ per}}(Y_{h,e})\cap C(Y_{h,e}):\Phi\vert_k\in\mathbb{P}^1,\myall k\in\She\right\},\\
 \vhg&=\left\{\Phi\in C(\G_h):\Phi\vert_{k}\in\mathbb{P}^1,\myall r\in\Shi \text{ with }k=r\cap\G_h\neq\emptyset\right\},
\end{align*}
where $\Hil{1}_{\rm{ per}}(Y_{h,e})$ denotes the subspace of $Y$-periodic functions in $\Hil{1}(Y_{h,e})$.
For the discretisation of the two-scale systems we define the tensor product spaces
\begin{align*}
 \Vhi&=\vhi\otimes\Vho,\\
      \Vhee&=\vhee\otimes\Vho,\\
   \Vhe&=\vhe\otimes\Vho,\\
  \Vhg&=\vhg\otimes\Vho.
\end{align*}

The scheme for the solution of the cell problems to obtain the diffusion tensor $D^{\rm hom}$ is, for  $j=1,\dots,n+1$ find $W^j\in\Vhe$ such that  
\[
\begin{aligned} 
&\Ltwop{D_e(x,y)(\nabla_yW^j+ e_j)}{\nabla_y \Phi}{Y_{h,e}\times\O_h} =0 %
\end{aligned} 
\]
for all $\Phi\in\Vhee$.

In order to propose a fully discrete scheme, we divide the time interval $[0, T ]$ into  $N$ subintervals, $0 = t_0 < \dotsc < t_N = T$ and denote by $\tau := t_{n} - t_{n-1}$ the time step, for simplicity we assume a uniform timestep. We consistently use the following shorthand for a function of time:  $ f^n := f (t_n )$, we denote by $\pdtau f^n:={\tau}^{-1}\left(f^n-f^{n-1}\right).$ We propose an IMEX time-stepping method in which the reactions  are treated explicitly and  the diffusive terms implicitly. The fully discrete scheme may be written as,
for $i=1,\dotsc,N$, given  
\[
C_e^{n-1}\in\Vho,\quad
C_i^{n-1}\in\Vhi,\quad R_f^{n-1},R_b^{n-1},P_d^{n-1},P_a^{n-1}\in\Vhg,  
\]
find 
\[
C_e^n\in\Vho,\quad
C_i^n\in\Vhi,\quad R_f^n,R_b^n,P_d^n,P_a^n\in\Vhg,  
\]
such that, for all $\Phi\in\Vho$ and $\Psi\in\Vhi$ 
\begin{equation} \label{eqn:scheme_start}
\begin{aligned} 
 &\Ltwop{\theta_e\pdtau C^n_e}{\Phi}{\O_h}+\Ltwop{D_{h,e}^{\rm hom}(x) \nabla C^n_e}{\nabla \Phi}{\O_h}\\
 &\qquad =\Ltwop{\theta_eF_e(C_e^{n-1}) - \frac 1 {|Y|} \int_{\G_h}  G_e(C_e^{n-1}, R_f^{n-1}, R_b^{n-1})\,  d{\sigma_y} }{\Phi}{\O_h},\\
&\Ltwop{\pdtau C^n_i}{\Psi}{Y_{h,i}\times\O_h}+\Ltwop{D_i(y) \nabla_y C^n_i}{\nabla_y \Psi}{Y_{h,i}\times\O_h}\\
&\qquad =\Ltwop{F_i(C^{n-1}_i)}{\Psi}{Y_{h,i}\times\O_h}+\Ltwop{G_i(P^{n-1}_a, C^{n-1}_i)}{\Psi}{\G_h\times\O_h},
\end{aligned}
\end{equation}
and for all $\Xi\in\Vhg$,
\begin{equation} \label{eqn:scheme_middle}
\begin{aligned} 
&\Ltwop{\pdtau R^n_f}{\Xi}{\G_h\times\O_h}+\Ltwop{D_{f} \nabla_{\G_h} R^n_f}{\nabla_{\G_h} \Xi}{{\G_h\times\O_h}}\\
&\qquad  =\Ltwop{F_f(R^{n-1}_f, R_b^{n-1})-  G_e(C^{n-1}_e, R^{n-1}_f, R^{n-1}_b)  - d_f R^n_f}{\Xi}{\G_h\times\O_h},\\
& \Ltwop{\pdtau R^n_b}{\Xi}{\G_h\times\O_h}+ \Ltwop{D_{b} \nabla_{\G_h} R^n_b}{\nabla_{\G_h} \Xi}{\G_h\times\O_h} \\
& \qquad  = \Ltwop{G_e(C^{n-1}_e, R^{n-1}_f, R^{n-1}_b) - G_d(R^{n-1}_b, P^{n-1}_d, P^{n-1}_a)  - d_b  R^n_b}{\Xi}{\G_h\times\O_h},\\
&\Ltwop{\pdtau P^n_d}{\Xi}{\G_h\times\O_h} +\Ltwop{D_{d} \nabla_{\G_h} P^n_d}{\nabla_{\G_h} \Xi}{\G_h\times\O_h}\\
&\qquad  =\Ltwop{F_d(P^{n-1}_d) -  G_d(R^{n-1}_b, P^{n-1}_d, P^{n-1}_a) - d_d P^n_d}{\Xi}{\G_h\times\O_h}, \\
&\Ltwop{\pdtau P^n_a}{\Xi}{\G_h\times\O_h}  + \Ltwop{D_{a} \nabla_{\G_h} P^n_a}{\nabla_{\G_h} \Xi}{\G_h\times\O_h} \\
& \qquad  = \Ltwop{G_d(R^{n-1}_b, P^{n-1}_d, P^{n-1}_a)  -  G_i(P^{n-1}_a, C^{n-1}_i)- d_a P^n_a}{\Xi}{\G_h\times\O_h}.
\end{aligned}
\end{equation}

\begin{remark}[Comments on the implementation]
The explicit treatment of the reaction terms results in fully decoupled systems of linear equations to be solved at each time-step.  Moreover, we use mass lumping for the approximation, this has two main advantages in the context of the present study. Firstly, lumping is equivalent to employing a nodal quadrature rule \cite{Thomee:06:book:Galerkin},  this allows us to interpret the two-scale systems as parameterised systems with the macroscopic variable playing the role of a parameter that may be solved independently and in parallel at each node of the macroscopic triangulation $\Th$. Secondly, the use of 
lumping and the fact that the system matrices do not change during the time evolution allows an efficient implementation in which 
virtually no assembly needs to  be carried out on each time-step.
 \end{remark}
 
\section{Benchmark computations}
\label{sec:benchmark}
We now carry out some benchmark simulations to illustrate the observed convergence rate of the numerical scheme proposed in Section~\ref{sec:scheme}.  We set $\O=[-0.5,0.5]^2$ and $Y_i$ to be a disc of radius 1. 
 For benchmarking we consider the following system 
 \begin{equation} \label{eq:hom_becnhmark} 
\begin{aligned}
& \partial_t c_e -\Delta c_e = -  \int_\Gamma  G_e(c_e, r_f,p_a)\,  d{\sigma_y} +f_{1}&& \text{ in } \Omega_T, \\
&\nabla c_e \cdot \nu = 0 && \text{ on } (\partial \Omega)_T, \\ 
&\partial_t c_i -  \Delta_y c_i = f_{2} && \text{ in } \Omega_T\times Y_i, \\
&\nabla_y c_i \cdot \nu = G_i(p_a, c_i) && \text{ on } \Omega_T \times\Gamma, \\
& \partial_t r_f -  \Delta_{\Gamma, y} r_f = -G_e(c_e, r_f,p_a)+f_3 && \text{ on } \Omega_T \times\Gamma,\\
& \partial_t p_a - \Delta_{\Gamma, y} p_a =G_e(c_e, r_f,p_a)-G_i(p_a, c_i) +f_4&& \text{ on } \Omega_T \times\Gamma, 
\end{aligned}
\end{equation}
with
\[
G_e(c_e, r_f,p_a)=c_er_f-p_a\quad\text{and}\quad G_i(p_a, c_i)=p_a-c_i. 
\]
The source terms $f_{1}(x,t), f_2(x,y,t), f_3(x,z,t)$ and $f_4(x,z,t)$ for $x\in\O,y\in Y_i, z\in\G$ and $t\in[0,T]$ are determined such that the exact solution to \eqref{eq:hom_becnhmark} is 
\[
c_e(x,t)=\cos(\pi t)e^{-10\left\vert x\right\vert^2}, c_i(x,y,t)=\left(1+\left\vert x\right\vert^2\right)e^{-4t(y_1y_2)^2} \quad\text{ for }x\in\O, y\in Y_i
\]
and
\[
r_f(x,z,t)=p_a(x,z,t)=\left(5+5\left\vert x\right\vert^2\right)e^{-4t(z_1z_2)^2}, \quad\text{ for }x\in\O, z\in\G
\]
and we set as the end time $T=0.25$.

In the numerical method we use the interpolant of the source terms into the appropriate finite element space. We consider a series of  refinements of the meshes with $\tau\sim h^2$  where $h:=\max\{h_\Gamma, h_{Y,i}, h_\Omega\}$. In particular we consider a series of  uniform refinements of the bulk and surface meshes with mesh sizes as given in 
Table~\ref{tab:mesh}. 
\begin{table}[h!]
\begin{center}
\begin{tabular}{|c|c|c|c|c|c|c|} 
$h_\G$&0.765& 0.390& 0.196& 0.098& 0.049\\
$h_{Y,i}$&1.000& 0.571& 0.305& 0.157& 0.080\\
$h_\O$&1.000& 0.500& 0.250& 0.125&0.063
\end{tabular}
\caption{Mesh sizes for the benchmarking study of Section~\ref{sec:benchmark}.}
\label{tab:mesh}
\end{center}
\end{table}

We denote by $e_{c_e}, e_{c_i}, e_{r}$ and $e_p$ the errors in the approximation of $c_e, c_i, r_f$ and $p_a$ respectively. In order to investigate the behaviour of the scheme we report on the experimental order of convergence (EOC) which provides  a numerical measure of the convergence rate. For a series of uniform refinements of a triangulation $\{\mathcal{S}_{h,i}\}_{i=0,\dots,N}$, denoting by $\{h_i\}_{i=0\dots,N}$, $\{e_i\}_{i=0\dots,N}$ the corresponding maximum mesh-size and the corresponding error respectively, then the EOC is given by
\[
\operatorname{EOC}_i(e_{i,i+1},h_{i,i+1}):=\ln(e_{i+1}/e_i)/\ln(h_{i+1}/h_i).
\]   
Figure \ref{EOC:FIG} shows the errors and the corresponding EOCs. 
The convergence rates appear optimal with first order convergence in the energy norm and second order convergence (as $\tau\sim h^2$) in the $\Lp{2}$ norm. 

  \begin{figure}[htbp]
  \centering
\includegraphics[trim = 0mm 0mm 0mm 0mm,  clip, width=.48\linewidth]{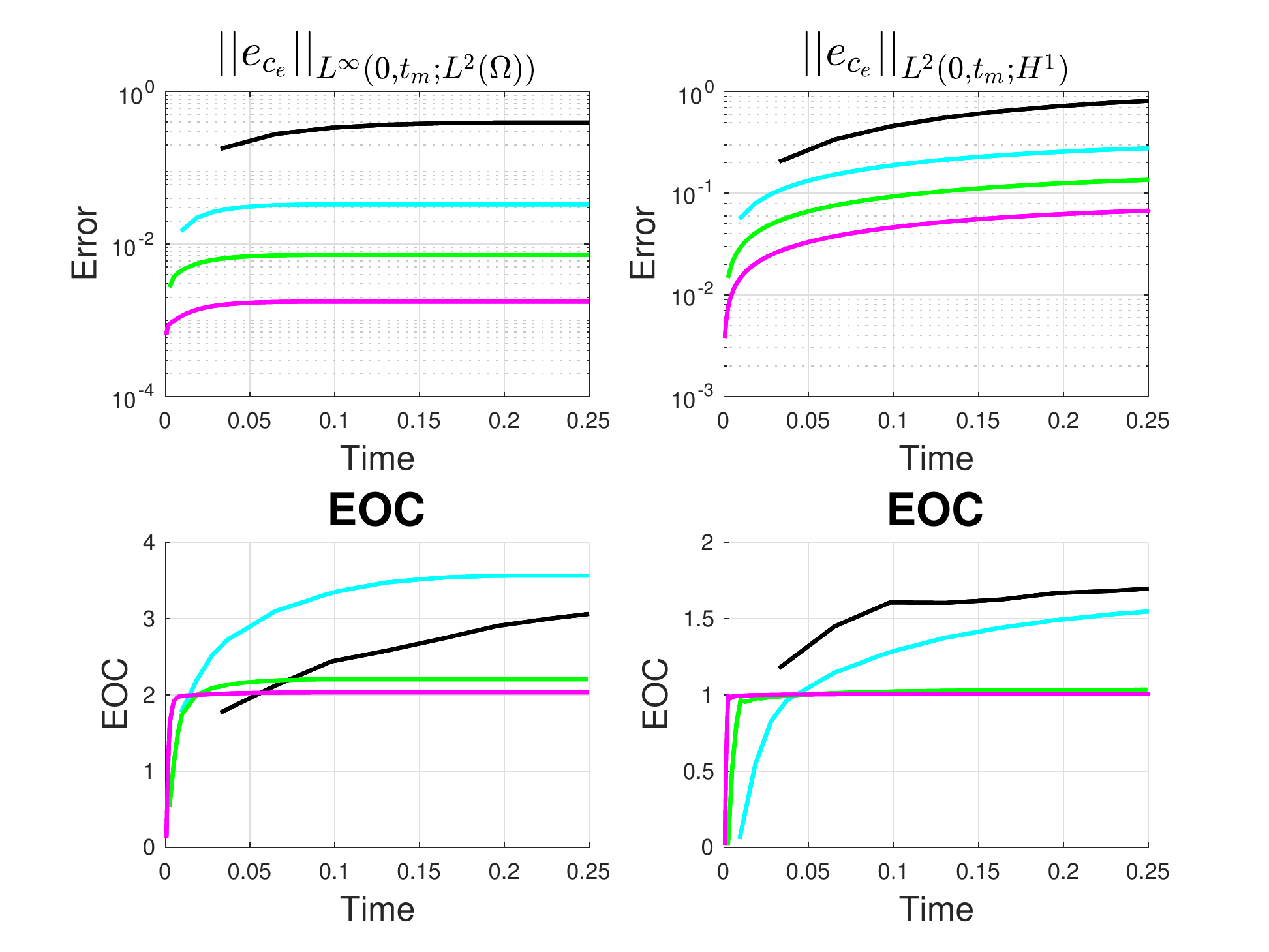}
\includegraphics[trim = 0mm 0mm 0mm 0mm,  clip, width=.48\linewidth]{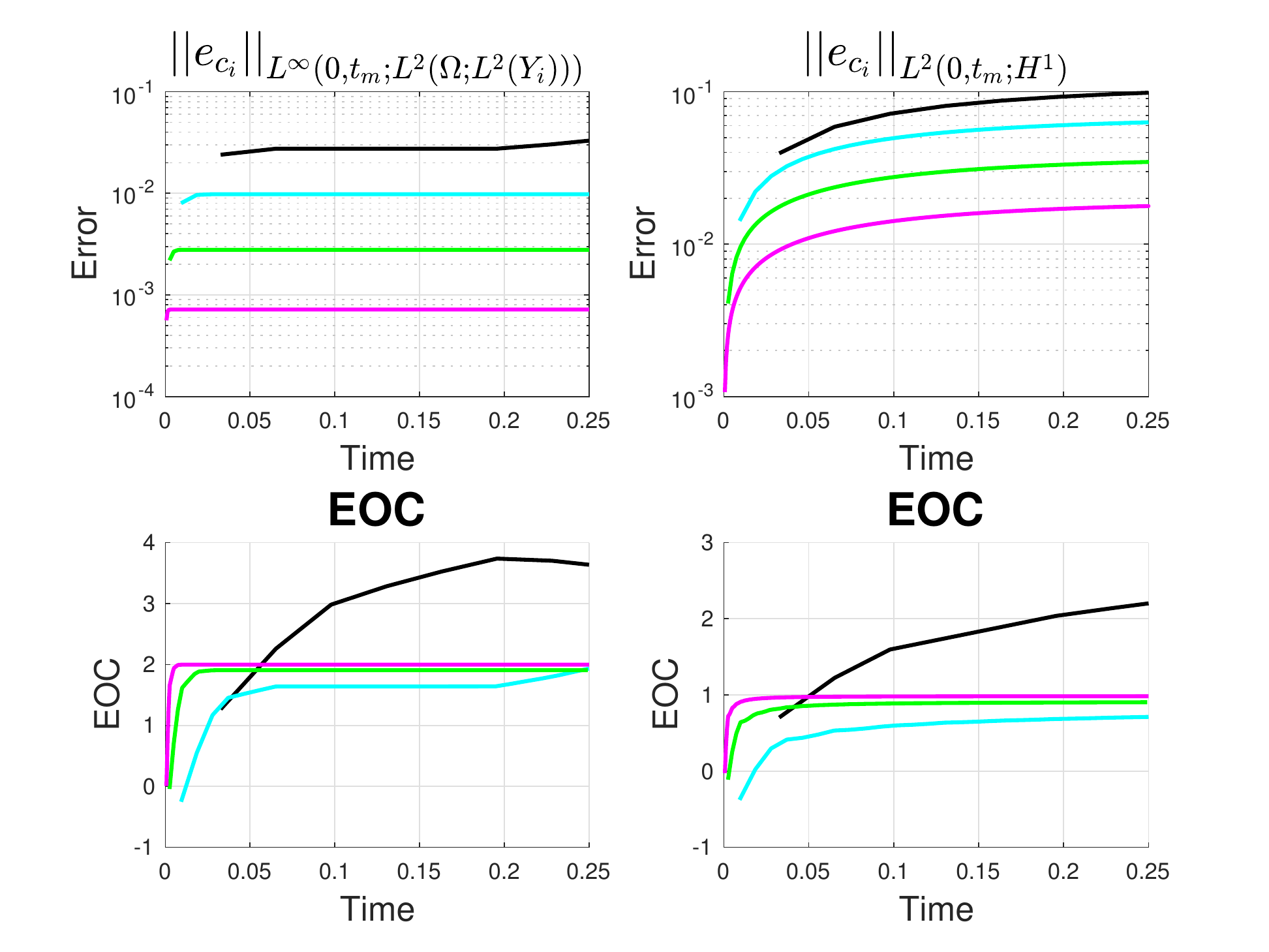} 
\includegraphics[trim = 0mm 0mm 0mm 0mm,  clip, width=.48\linewidth]{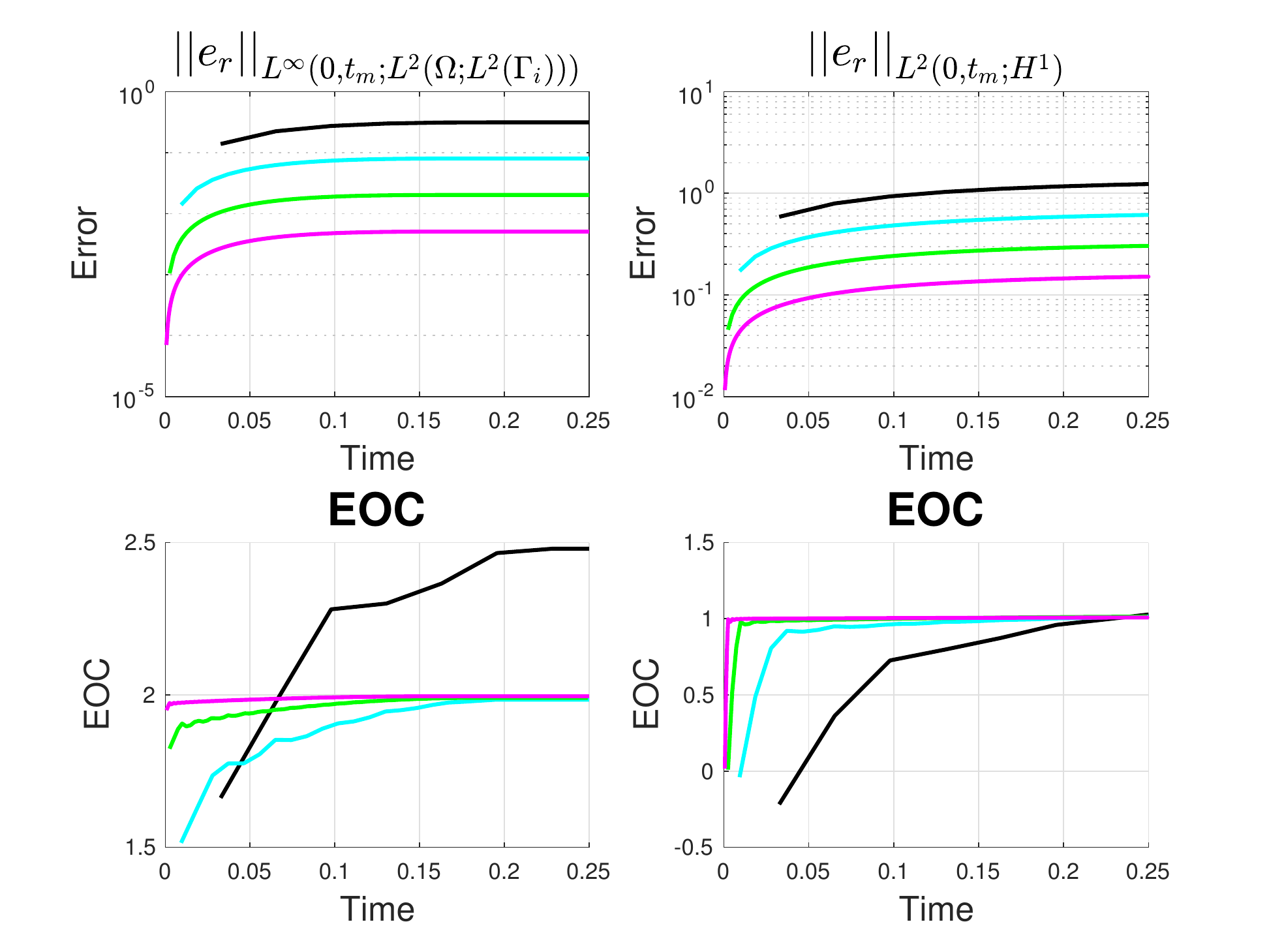}
\includegraphics[trim = 0mm 0mm 0mm 0mm,  clip, width=.48\linewidth]{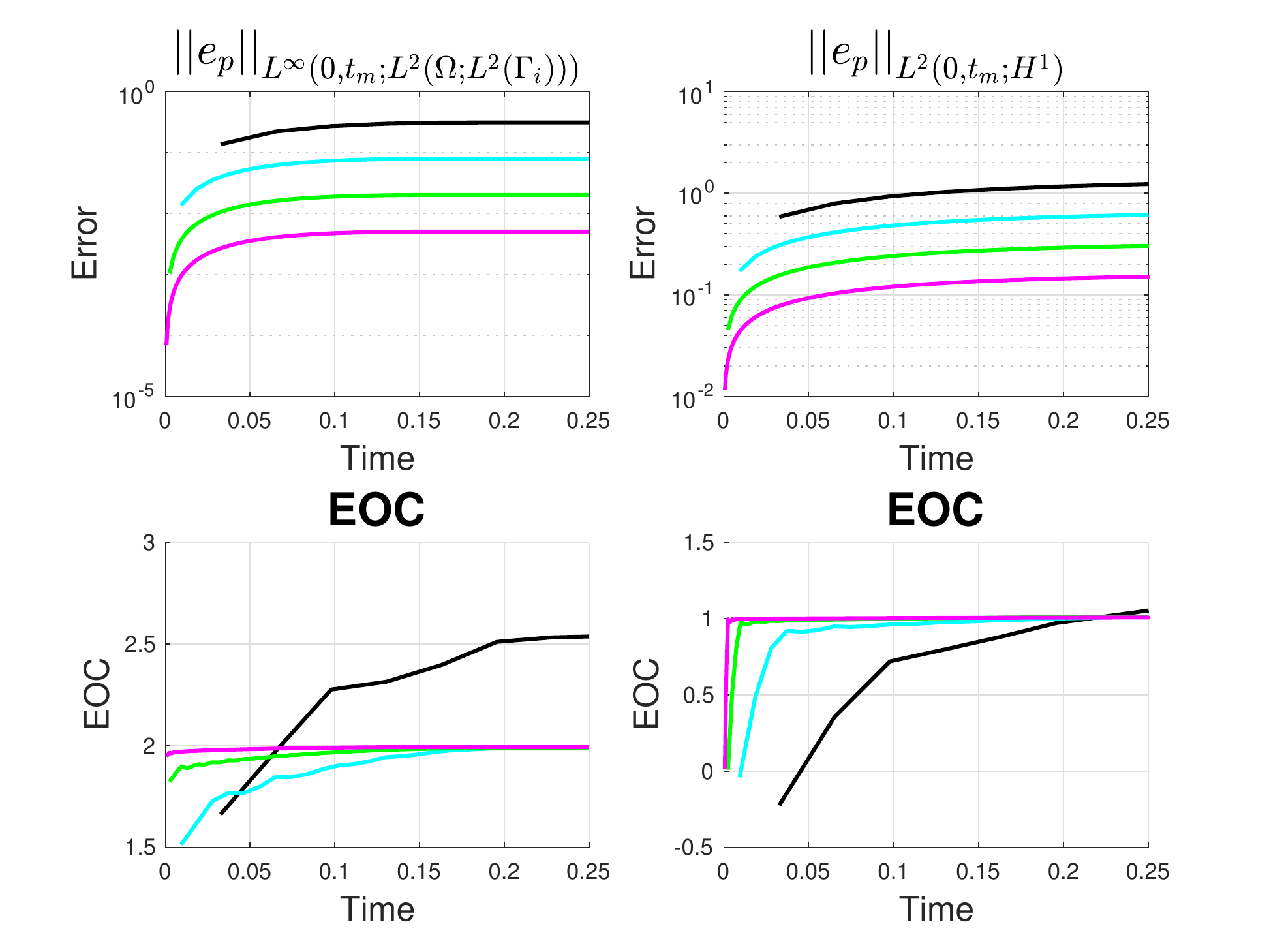} 
\caption{Errors in the $\Lp{2}(\Hil{1})$ and $\Lp{\infty}(\Lp{2})$ norms for a series of mesh refinements (c.f., Table \ref{tab:mesh}) with $\tau\sim h^2$ 
  and experimental order of convergence. The observed convergence rates appear optimal.}
  \label{EOC:FIG}
  \end{figure}
  }
An analysis of the numerical method is beyond the scope of the present work. We believe that combining  the techniques developed in \cite{elliott2012finite} for the analysis of finite element schemes for bulk-surface equations,   \cite{hoang2005high} which deals with multiscale finite element methods and  \cite{lakkis2013implicit} which proves error bounds for IMEX approximations of semilinear systems, it should be possible to prove optimal error bounds for our method that reflect the rates observed numerically in Figure \ref{EOC:FIG}.

\section{Parameterisation and numerical solutions  for  a biologically relevant model}\label{sec:bio_model}

We now present and simulate a biologically relevant model, the model considered in this section is related to  the Langmuir-Hinshelwood mechanism for signalling processes at the level of a single cell considered in \cite{Garcia_2014}, and in particular we take the majority of our parameters from said work. We make the assumption that all the parameters are independent of the microscopic variable and state the parameter values we use in the microscopic model \eqref{main1}--\eqref{bc_cond} along with the source of the parameter value in Table \ref{tab:parameters}. In order to keep the model as simple as possible whilst still illustrating the key phenomena captured by the model, we assume there is no production or linear degradation of any of the species, i.e., we set $F_e(c_e) = F_i(c_i) = F_f(r_f, r_b) = F_d(p_d) =0$ and $d_k =0$ for $k=b,f,a,d$.

\begin{table}
  \centering
  \begin{tabular}{lcr}
    \hline
    Parameter & Value & Source \\
    \hline\\
    $a_e$ & $10^3 ({\rm mol}/ {\rm m}^3)^{-1}\,  {\rm s} ^{-1}$ & \cite{Garcia_2014} \\
    $b_e$ & $5\cdot10^{-3}  {\rm s} ^{-1}$ & \cite{Garcia_2014} \\
    $a_i$ & $10^{-2} (\frac{\rm molecules}{\mu {\rm m}^2})^{-1} \, {\rm s}^{-1} = 6\cdot  10^{9}\,  ({\rm mol}/{\rm m}^2)^{-1} \, {\rm s}^{-1}$ & \cite{Garcia_2014} \\
    $b_i$ & $10^{-2} \,  {\rm s} ^{-1}$ & \cite{Garcia_2014} \\
    $\gamma_i$&$2 \cdot 10^{-3} \, {\rm s}^{-1}$&\cite{Garcia_2014} \\
    $\kappa_i$&$10^{-8}  \, {\rm m} \, s^{-1}$&\cite{Garcia_2014} \\
    $D_e$&$ 10^{-9} \, \mathrm{m}^2 \, \mathrm{s}^{-1}$&\cite{LL_1986}\\
    $D_i$&$ 10^{-11} \, \mathrm{m}^2 \, \mathrm{s}^{-1}$&\cite{Kuhn_2011}\\
     $D_k$, \, $k=b,f,a,d$& $10^{-15} \, \mathrm{m}^2 \, \mathrm{s}^{-1}$&\cite{LL_1986}\\ 
    \hline\\
  \end{tabular}
  \caption{Parameters used for the microscopic model and sources for the parameter estimates.}
  \label{tab:parameters}
\end{table}

The dependent and independent variables of the microscopic model and their associated units are as given in Table \ref{tab:var}.
\begin{table}
  \centering
\begin{tabular}{ccccccc} 
\hline
$x$  & $t$  & $c_e$ &  $c_i$ & $r_f$, $r_b$ & $p_d$, $p_a$ \\
\hline\\
m & s &  mol/m$^3$ &   mol/m$^3$ &  mol/m$^2$ &  mol/m$^2$\\
\hline\\
\end{tabular}
\caption{Variables of microscopic model and associated units.}
\label{tab:var}
\end{table}
Finally, in order to ensure there is some ligand present in the system, we set $\partial\Omega_D$ to be a Dirichlet boundary such that 
the boundary condition \eqref{bc_cond} becomes
\begin{equation*}
\begin{aligned}
   c^\ve_e &= \hat c_e  \quad   && \text{ on } \;  \partial\Omega_D, \;  \; t\geq0,\\
 D^\ve_e(x) \nabla c^\ve_e \cdot \nu &= 0  \quad   && \text{ on } \; \partial \Omega\setminus\partial\Omega_D, \;  \; t>0,
\end{aligned}
\end{equation*}
where we set $\hat c_e=10^{-4} {\rm mol}/{{\rm m}^3}$, c.f., \cite{Garcia_2014}.
Taking $\ve = 10^{-3}$  we introduce the characteristic scales 
\begin{equation}
\begin{aligned}
 \hat t = 10^3\, {\rm s},  \quad \hat x = 10^{-2}\, {\rm m}, \quad   \hat r = \hat r_l = \hat p_k = 10^{-9} \frac{\rm mol}{{\rm m}^2},  \\
 \hat c  = \hat c_e = \hat c_i =  \frac{\hat r }{\ve \, \hat x} = 10^{-4} \frac{\rm mol}{{\rm m}^3} , 
  \end{aligned}
\end{equation} 
where $l =f,b$ and  $k=a,d$,  and then  the dimensionless parameters  are given by 
\begin{equation}
\begin{aligned}
& D_e^\ast= D_e \hat t/ (\hat x)^2 = 10^{-2}, \quad   \ve^2 D_i^\ast= D_i \hat t/ (\hat x)^2 = 10^{-5}, \; D_i^\ast = 10 \\
&  \ve^2 D_l^\ast= D_l \hat t/ (\hat x)^2 = 10^{-8}, \; D_l^\ast = 10^{-2}, \quad l =f,b,d,a, \\
& \ve b^\ast_e  = \frac { \hat t } {\hat x} \frac{ \hat r_b}{\hat c} b_e = 5 \cdot  10^{-3}, \quad \ve a^\ast_e = \frac { \hat t  \hat r_f} {\hat x} a_e = 0.1, \quad b^\ast_e = 5, \quad a_e^\ast = 100, \\  
& \ve \gamma_i^\ast =  \frac{ \hat t}{\hat x} \frac{\hat p_a}{\hat c_i } \gamma_i =  2\cdot 10^{-3}, \quad
 \ve\kappa_i^\ast = \frac{\hat t}{\hat x} \kappa_i = 10^{-3}, \quad  \gamma_i^\ast = 2, \quad \kappa_i^\ast = 1, \\
&    a_i^\ast = a_i \hat p_d \hat t = 6\cdot 10^3, \quad b_i^\ast = b_i \hat t = 10. 
\end{aligned}
\end{equation} 
Notice that we also have $a^\ast_{e} = a_e \hat c \hat t= 100$, $b^\ast_{e}= b_e \hat t = 5$, and  $\gamma^\ast_{i} = \gamma_i \hat t = 2$,  $\kappa^\ast_{i} = \frac{ \hat t \hat c_i}{ \hat p_d} \kappa_i = 1$, which is consistent with scaling above. Following the derivation of the two-scale macroscopic model outlined in Section~\ref{sec:macro_model} we obtain the following dimensionless homogenised system
\begin{equation}
\begin{aligned}
&\theta_e \partial_t c_e - \nabla\cdot ( D^{\rm hom}_e \nabla c_e) = \frac 1 {|Y|} \int_\Gamma (  b_e^\ast\,  r_b  - a_e^\ast c_e r_f)  d{\sigma_y} && \text{ in } \Omega, \\
&c_e  = 1 && \text{ on } \partial\Omega_D, \\ 
&D^{\rm hom}_e\,  \nabla c_e \cdot \nu = 0 && \text{ on } \partial \Omega/\partial\Omega_D, \\ 
&\partial_t c_i - \nabla_y\cdot ( D_i^\ast \nabla_y c_i) = 0  && \text{ in } \Omega\times Y_i, \\
&D_i^\ast \nabla_y c_i \cdot \nu = \gamma_i^\ast\,  p_a - \kappa_i^\ast c_i && \text{ in } \Omega \times\Gamma, \\
\end{aligned}
\end{equation}
where $\theta_e = |Y_e|/|Y|$,  and $D_{e,ij}^{\rm hom}= |Y|^{-1} \int_{Y_e} \big[ D_{e, ij}^\ast + (D_{e}^\ast \nabla_y w^j(y))_i \big] dy$ and $w^j$ are solutions of the unit cell problems
\begin{equation}\label{eq:cell_problem_bio}
\begin{aligned}
&{\rm div}_y ( D_e^\ast(\nabla_y w^j + e_j)) = 0 && \text{ in } Y_e,  \quad \int_{Y_e} w^j(y) dy = 0, \\
&  D_e^\ast(\nabla_y w^j + e_j) \cdot \nu = 0 && \text{ on } \Gamma,  \quad w^j \;  \; \; Y-\text{periodic},
\end{aligned}
\end{equation}
together with the dynamics of receptors on the cell membrane $\Omega\times\Gamma$
\begin{equation} 
\begin{aligned}
& \partial_t r_f - \nabla_{\Gamma,y} \cdot (D_{f}^\ast \nabla_{\Gamma, y}\,  r_f) = - a_{e}^\ast c_e r_f + b_{e}^\ast\,  r_b, \\
& \partial_t r_b - \nabla_{\Gamma,y} \cdot (D_{b}^\ast \nabla_{\Gamma, y} \, r_b) =  \phantom{  f_r(r_f)+} a_{e}^\ast c_e r_f    -  b_{e}^\ast\,  r_b  - a_i^\ast\,  r_b  p_d+ b_{i}^\ast p_a, \\
& \partial_t p_d - \nabla_{\Gamma,y} \cdot (D_{d}^\ast \nabla_{\Gamma,y} \, p_d) =   - a_i^\ast\,  r_b  p_d + b_i^\ast p_a , \\
& \partial_t p_a - \nabla_{\Gamma,y} \cdot (D_{a}^\ast \nabla_{\Gamma,y} \, p_a) = \phantom{ f_p(p_d) }   a_i^\ast\,  r_b p_d  -  b_i^\ast p_a  - \gamma_{i}^\ast  p_a+ \kappa_{i}^\ast c_i.
\end{aligned}
\end{equation}
The dimensionless parameter values are
\begin{equation}\label{param_values}
\begin{aligned}
& D_e^\ast=  10^{-2}, \quad   D_i^\ast = 10, \quad  D^\ast_{f} = D^\ast_{b} = D^\ast_d= D^\ast_a= 10^{-2}, \\
& a_e^\ast = 100, \quad b_e^\ast = 5, \quad a_i^\ast =  6\cdot 10^3, \quad b_i^\ast = 10, \quad  \gamma^\ast_{i} =  2, \quad \kappa^\ast_{i} = 1. 
\end{aligned}
\end{equation} 
Scaling the initial conditions  appropriately  yields the nondimensional initial values 
 \begin{equation} \label{eqn:bio_ic}
\begin{aligned}
&  c^\ast_{i0}(x,y) = 1+c_{i,1}(x) c_{i,2}(y),  \;\\
&   r^\ast_{f0}(x,z) = 0.17(1+r_{f,1}(x) r_{f,2}(z)),\;  p^\ast_{d0}(x,z) = 0.065(1+ p_{d,1}(x)p_{d,2}(z)),
\end{aligned}
\end{equation} 
for $x\in\Omega$, $y\in Y_i$ and $z\in\Gamma$,  with all the remaining initial conditions taken to be zero. The functions $c^\ast_{i0}, r^\ast_{f0}$ and $p^\ast_{d0}$ correspond to scaled, nonnegative perturbations  
\changes{
 of the initial conditions 
\begin{equation}
\begin{aligned}
 c_{i0} = 10^{-7}\,  {\rm M}  = 10^{-4} \,\frac{\rm mol}{\rm m^3},  \;  \;  r_{f0} = 17 \cdot 10^{-11} \, \frac{\rm mol}{\rm m^2}, \; \; p_{d0} =  
  6.5 \cdot 10^{-11} \, \frac{\rm mol}{\rm m^2}, 
\end{aligned}
\end{equation} 
 considered in \cite{Garcia_2014}.  } %

\subsection{Simulations of macroscopic model  in biologically relevant regimes}
We  illustrate the influence that the geometry of the periodic cell in which we solve for the effective homogenised diffusion tensor $D^{\rm hom}$ as well as the associated geometry of the (biological) cells $Y_i$ and membranes $\G$ have on the macroscopic dynamics of signalling molecules (ligands). To this end we consider two different geometries for the microstructure, specifically we let $Y=[-2,2]^2$ and consider  either elliptical cells with 
\[
Y_i=\left\{x\in Y \, | \;  0.26x_1^2+5x_2^2<1\right\},
\]
 i.e., an ellipse centred at $(0,0)$ with major and minor axes of approximate length 1.96   and 0.45 respectively or cells whose shape is defined by
 \begin{equation}
 \label{eqn:Dziuk_cell_shape}
 Y_i=\left\{x\in Y \, |  \; (x_1+0.2-x_2^2)^2+x_2^2<1\right\}.
 \end{equation}
To obtain the homogenised diffusion tensor we solve the cell problems corresponding to \eqref{eq:cell_problem_bio} on $Y_e=[-2,2]^2/Y_i$ for the two different cell geometries. 
For the elliptical cell geometry we used a mesh with 1039514 DOFs and for the other cell geometry we used a mesh with 1008834 DOFs.  Figure \ref{fig:w_j} shows the numerical simulation  results for the solution  $w^2$ of the 'unit cell' problems  \eqref{eq:cell_problem_bio}  on the two different geometries. The resulting homogenised diffusion tensor is given by
\[
D^{\rm hom}_{h,e}=\left[\begin{array}{cc}
8.167\cdot10^{-3}&0\\
0&1.841\cdot10^{-3}
\end{array}
\right]
\]
for the case of the ellipse and
\[
D^{\rm hom}_{h,e}=\left[\begin{array}{cc}
6.556\cdot10^{-3}&0\\
0&6.149\cdot10^{-3}
\end{array}
\right]
\]
for the geometry specified in \eqref{eqn:Dziuk_cell_shape}. 
As expected due to the large aspect ratio of the ellipse the resulting homogenised diffusion tensor exhibits stronger anisotropy than for the other cell shape.
  \begin{figure}[htbp]
\includegraphics[trim = 0mm 40mm 0mm 40mm,  clip, width=\linewidth]{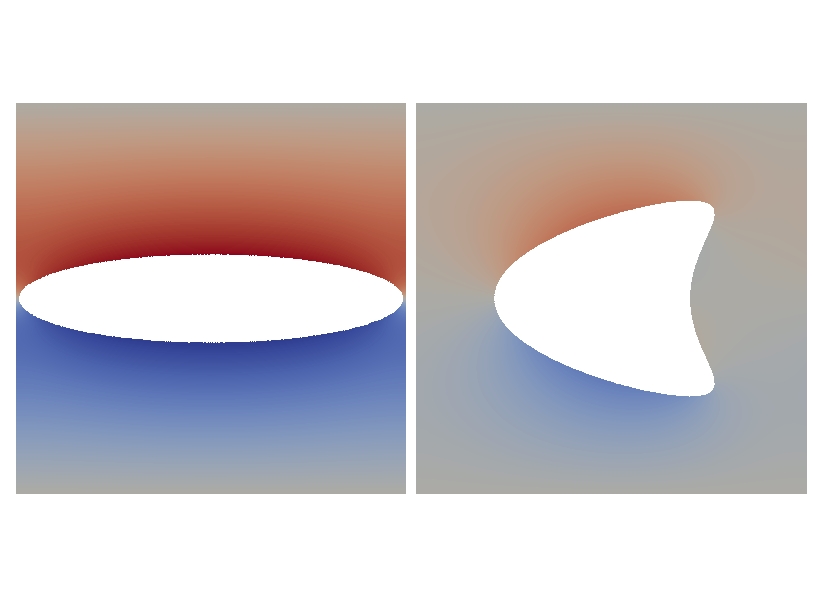}
\caption{Solutions $w^2$ of the cell problem \eqref{eq:cell_problem_bio} for the two different $Y_e$ domains considered.}
  \label{fig:w_j}
  \end{figure}
For the tissue we set $\Omega=[0,0.1]^2$ and take 
\[
\partial\Omega_D=\{x\in\partial\Omega \, |\;  \max\{x_1,x_2\}<5\cdot10^{-2}\},
\]
modelling a constant source of ligands from the south west corner of the domain. As mentioned above on the remainder of the boundary $\partial \Omega$
 we consider  zero-flux boundary condition for $c_e$. For the initial data we set the perturbations c.f., \eqref{eqn:bio_ic} to be of the form
\[
\begin{aligned}
&c_{i,1}(x) c_{i,2}(y) = f_{c_{i0}}(x,y)=0.95\sin\left(\pi\left(2y_1+\frac{y_2}{2}\right)\right)\sin\left(5{\pi| x|}\right),\\
&r_{f,1}(x) r_{f,2}(y)= f_{r_{f0}}(x,y)=0.95\cos\left(\pi\left( y_1+4{y_2}\right)\right)\cos\left(30{\pi| x|}\right),\\
&p_{d,1}(x) p_{d,2}(y) =f_{p_{d0}}(x,y)=0.95\cos\left(\pi\left(2y_1+\frac{y_2}{2}\right)\right)\cos\left(10{\pi|x|}\right), 
\end{aligned}
\]
for $x\in \Omega$ and $y \in Y$. 
For the approximation we used a triangulation $\O_h$ with 1089 DOFs, the triangulation $Y_{h,i}$ of the ellipse had 81 DOFs and the triangulation $Y_{h,i}$ of the domain given by \eqref{eqn:Dziuk_cell_shape} had 89 DOFs, the induced surface triangulations $\G_h$ had 32 and 33 DOFs respectively. For the timestep we used a value of $2\cdot10^{-3}$.

Figures \ref{FIG:E1} and \ref{FIG:E2} show results of the simulation at $t=10,100,200$ and $250$ with the elliptical cell geometry whilst Figure \ref{FIG:D} shows results of the simulation at the same times with the cell geometry given by \eqref{eqn:Dziuk_cell_shape}. In each Figure we also include the microscopic solutions at the DOFs with macroscopic coordinates $(0,0), (0.05,0.05)$ and $(0.1,0.1)$ with the macroscopic DOF associated with each set of microscopic results indicated by a grey line in the Figure to the corresponding point in the macroscopic domain.
Focusing on the differences between the two sets of results, we see that the strongly anisotropic homogenised diffusion tensor associated with the elliptical cell geometry leads to faster transport in the horizontal direction and slower vertical transport. As a result for $t=200$, see Figure \ref{FIG:E_200},  there are very few bound receptors present on the cell  at the macroscopic point $(0.1,0.1)$ and it is only by $t=250$ that bound receptors are clearly visible on this cell. On the other hand the almost isotropic homogenised diffusion tensor associated with the cell geometry specified in  \eqref{eqn:Dziuk_cell_shape} leads to equally fast vertical and horizontal transport and by $t=200$ there are clearly  a large number of bound receptors present on the cell membrane at the macroscopic point $(0.1,0.1)$, c.f., Figure \ref{FIG:D_200}. More generally, in both cases we see significant heterogeneity at the microscopic level in the concentrations of the different membrane resident species at different times  during the simulation motivating the multiscale modelling approach we employ.

\begin{figure}
  \centering
    \begin{subfigure}[b]{.6\textwidth}
    \centering
    \includegraphics[trim = 0mm 0mm 0mm 0mm,  clip, width=\linewidth]{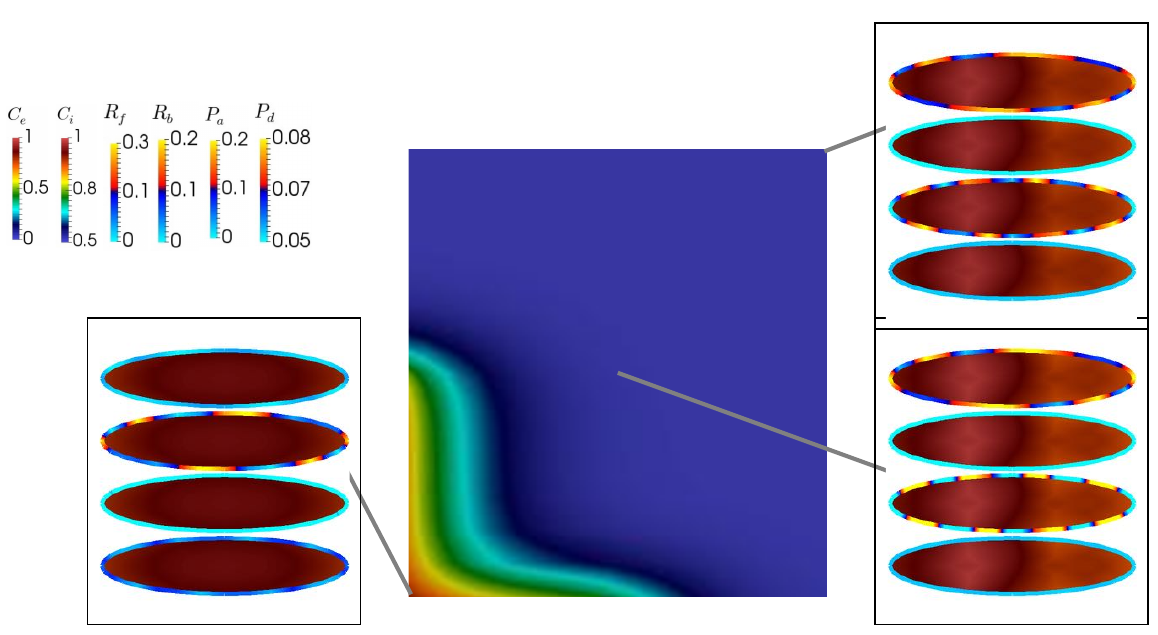}
  \caption{$t=10$}
  \label{FIG:E_10}
  \end{subfigure}
    \begin{subfigure}[b]{.6\textwidth}
    \centering
    \includegraphics[trim = 0mm 0mm 0mm 0mm,  clip, width=\linewidth]{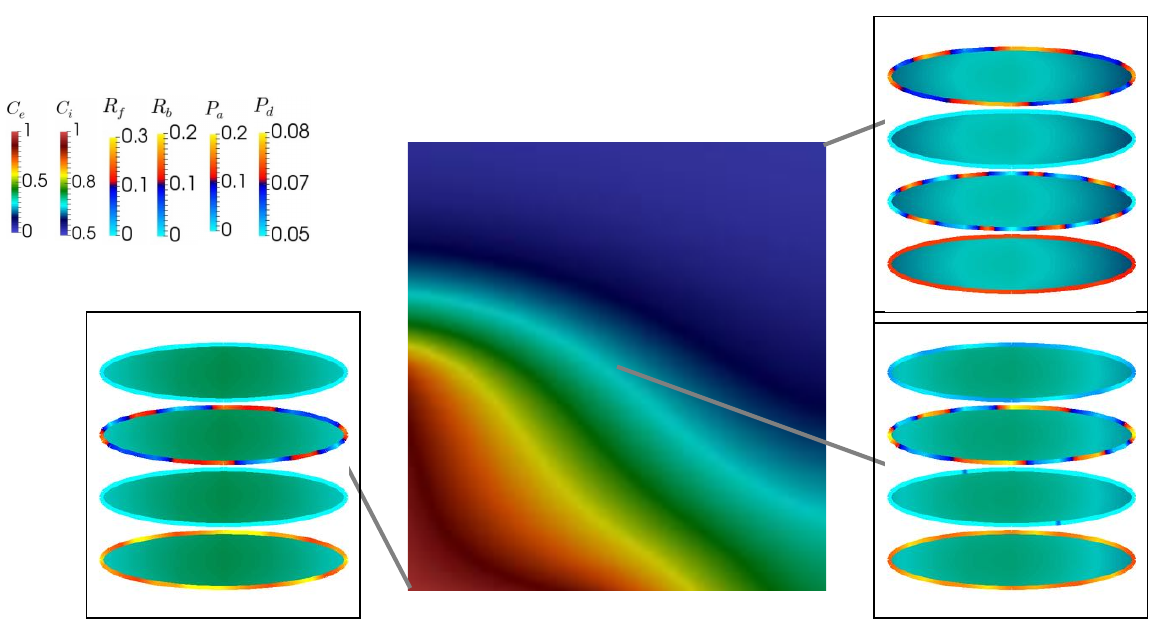}
  \caption{$t=100$}
  \label{FIG:E_100}
  \end{subfigure}
  \caption{Results of the simulation of Section~\ref{sec:bio_model} with the elliptical cell geometry. The inset in each subfigure show the microscopic solutions at the corresponding macroscopic DOF (grey line). The macroscopic domain is shaded by $C_e$ whilst in each inset the cell interior is shaded by $C_i$ and reading from top to bottom, the membrane is shaded by $R_f,R_b,P_d$ and $P_a$ respectively. For further details see text.}
  \label{FIG:E1}
  \end{figure}
  \begin{figure}
  \centering
    \begin{subfigure}[b]{.6\textwidth}
    \centering
    \includegraphics[trim = 0mm 0mm 0mm 0mm,  clip, width=\linewidth]{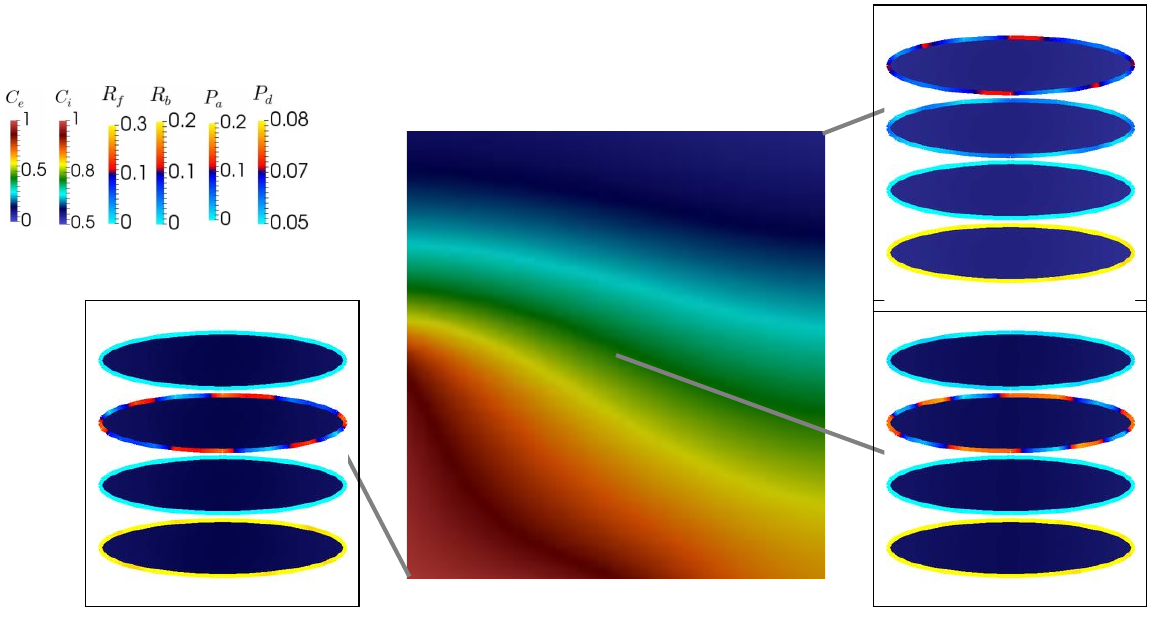}
  \caption{$t=200$}
  \label{FIG:E_200}
  \end{subfigure}
    \begin{subfigure}[b]{.6\textwidth}
    \centering
    \includegraphics[trim = 0mm 0mm 0mm 0mm,  clip, width=\linewidth]{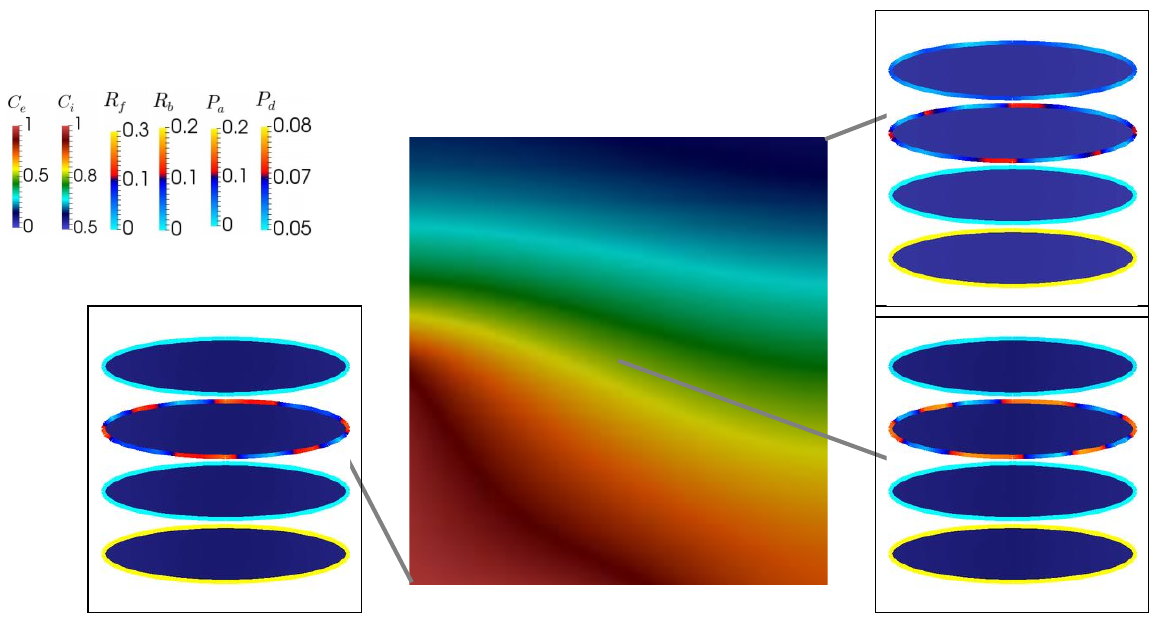}
  \caption{$t=250$}
  \label{FIG:E_250}
  \end{subfigure}
  \caption{Results of the simulation of Section~\ref{sec:bio_model} with the elliptical cell geometry. The inset in each subfigure show the microscopic solutions at the corresponding macroscopic DOF (grey line). The macroscopic domain is shaded by $C_e$ whilst in each inset the cell interior is shaded by $C_i$ and reading from top to bottom, the membrane is shaded by $R_f,R_b,P_d$ and $P_a$ respectively. For further details see text.}
  \label{FIG:E2}
\end{figure}
\begin{figure}
    \begin{subfigure}[b]{\textwidth}
    \centering
    \includegraphics[trim = 0mm 0mm 0mm 0mm,  clip, width=\linewidth]{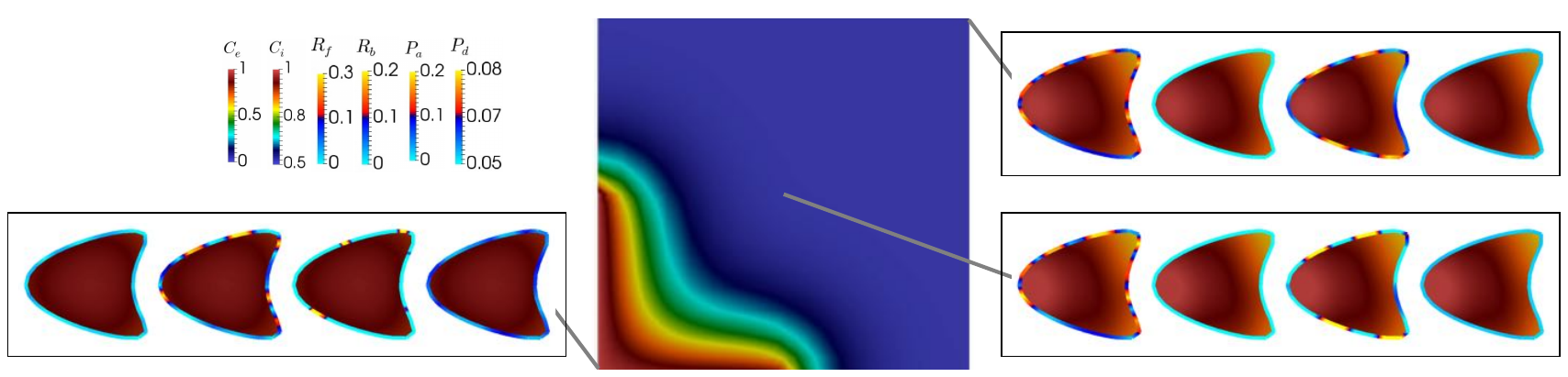}
  \caption{$t=10$}
  \label{FIG:D_10}
  \end{subfigure}
    \begin{subfigure}[b]{\textwidth}
    \centering
    \includegraphics[trim = 0mm 0mm 0mm 0mm,  clip, width=\linewidth]{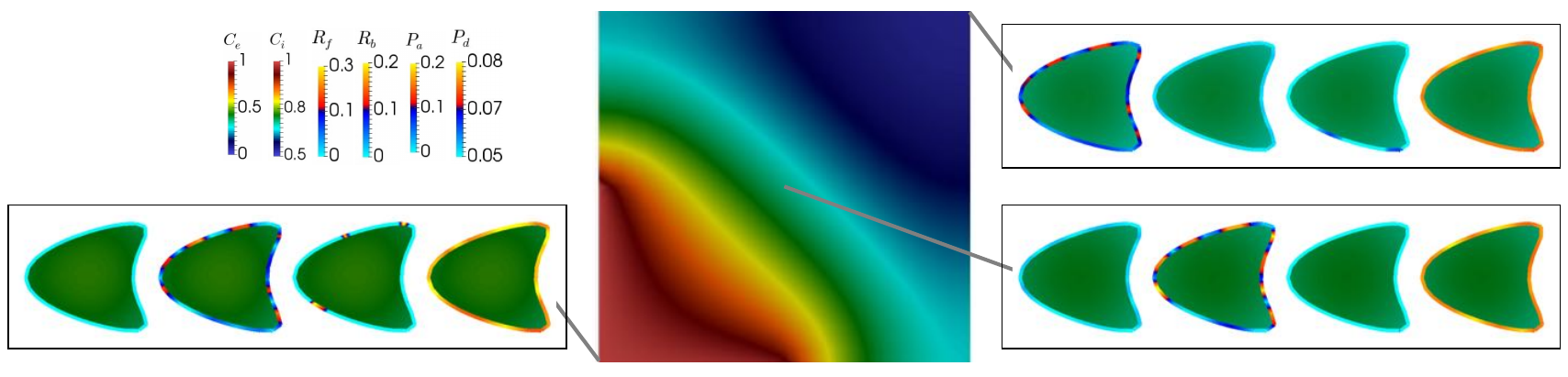}
  \caption{$t=100$}
  \label{FIG:D_100}
  \end{subfigure}
    \begin{subfigure}[b]{\textwidth}
    \centering
    \includegraphics[trim = 0mm 0mm 0mm 0mm,  clip, width=\linewidth]{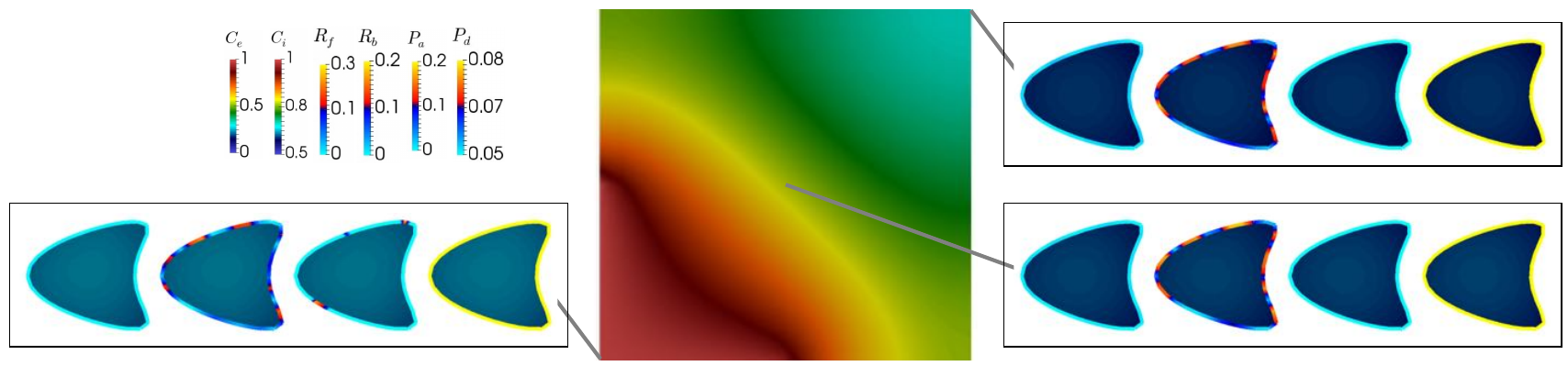}
  \caption{$t=200$}
  \label{FIG:D_200}
  \end{subfigure}
    \begin{subfigure}[b]{\textwidth}
    \centering
    \includegraphics[trim = 0mm 0mm 0mm 0mm,  clip, width=\linewidth]{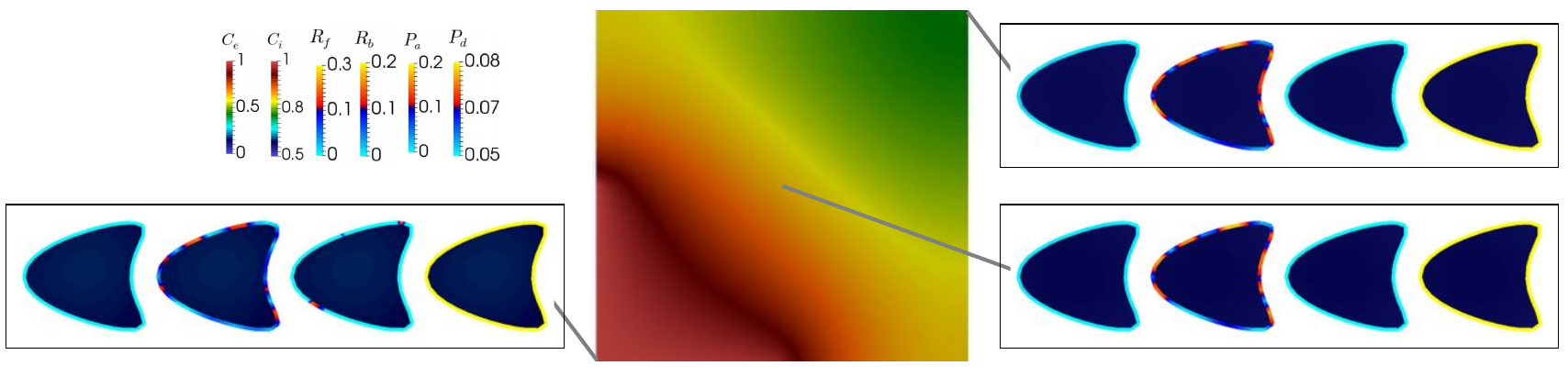}
  \caption{$t=250$}
  \label{FIG:D_250}
  \end{subfigure}
  \caption{Results of the simulation of Section~\ref{sec:bio_model} with the cell geometry given by (\ref{eqn:Dziuk_cell_shape}). The inset in each subfigure show the microscopic solutions at the corresponding macroscopic DOF (grey line). The macroscopic domain is shaded by $C_e$ whilst in each inset the cell interior is shaded by $C_i$ and reading from left to right, the membrane is shaded by $R_f,R_b,P_d$ and $P_a$ respectively. For further details see text.}
  \label{FIG:D}
\end{figure}

\bibliographystyle{plain}
\bibliography{Master_bib}

\changes{ 
\appendix

\section{Generalised trace inequality}
The trace inequality for $v\in W^{1,q}(Y_l)$, with $1<q<\infty$ and $l=e,i$, reads  
\begin{equation}\label{trace_gen} 
\|v\|_{L^r(\Gamma)}  \leq \mu \big[ \|v\|^{1-\lambda}_{L^q(Y_l)} \| v \|^\lambda_{W^{1,q}(Y_l)} +  \|v\|^{(1-1/r)(1-\lambda)}_{L^q(Y_l)} \|v \|^{1/r + \lambda(1-1/r)}_{W^{1,q}(Y_l)} \big]
\end{equation}
for $\lambda = \frac { {\rm dim}(Y_l)(r-q)}{q(r-1)}$ and $\mu= \mu(r,q, Y_l) >0$, see e.g.\ \cite{Galdi_2011}.

\section{Two-scale convergence and periodic unfolding operator}\label{appendix_A1}

We recall the definition and some properties of two-scale convergence and the unfolding operator.

\begin{definition}[Two-scale convergence]\cite{Allaire_1992,   Lukkassen_2002,  Nguetseng_1989}
A sequence $\{u^\ve\}$ in
$L^p(\Omega)$, with $1<p<\infty$,  is  two-scale convergent to $u\in L^p(\Omega\times Y)$ 
if 
$$
  \lim_{\ve\to 0}\int_\Omega u^\ve(x)\phi\bigg(x,\frac{x}{\ve}\bigg)dx
	= \int_{\Omega\times Y} u(x,y)\phi(x,y) dydx
$$
for any $\phi\in L^q(\Omega; C_{\rm per}(Y))$, with $1/p+1/q=1$.
\end{definition}

\begin{theorem}\cite{Allaire_1996, Neuss-Radu_1996}
Let  $\{v^\ve\} \subset L^2(\Gamma^\ve)$ satisfies $\ve \|v^\ve\|^2_{L^2(\Gamma^\ve)} \leq C$, 
then there exists a two-scale limit $v\in L^2(\Omega; L^2(\Gamma))$ such that, up to a subsequence,   
$v^\ve$  two-scale converge to $v$ in the sense that 
$$
  \lim_{\ve\to 0} \ve \int_{\Gamma^\ve} v^\ve(x)\phi\bigg(x,\frac{x}{\ve}\bigg)d\sigma^\ve
	= \int_{\Omega\times \Gamma} v(x,y)\phi(x,y) d\sigma_ydx
$$
for any $\phi\in C_0(\Omega; C_{\rm per}(Y))$.
\end{theorem}

\begin{lemma}[Two-scale compactness \cite{Allaire_1992, Allaire_1996, Nguetseng_1989}]\label{lem1}
\begin{enumerate}[{\rm i.}]
\item If $\{u^\ve\}$ is bounded in $L^2(\Omega)$, there exists
a subsequence (not relabelled) such that $u^\ve \rightharpoonup u$ two-scale as $\ve\to 0$
for some function $u\in L^2(\Omega\times Y)$. 
\item If $u^\ve\rightharpoonup u$ weakly in $H^1(\Omega)$ then 
$u^\ve \rightharpoonup u$  and $\nabla u^\ve \rightharpoonup \nabla u+\nabla_y u_1$ two-scale, where 
$u_1\in L^2(\Omega;H^1_{\rm per}(Y)/\mathbb R)$.
\item If $\|u^\ve\|_{H^1(\Omega_e^\ve)} \leq C $
   and $[u^\ve]^\sim$ and $[\nabla u^\ve]^\sim$ are extensions by zero from $\Omega_e^\ve$ into $\Omega$ of $u^\ve$ and $\nabla u^\ve$ respectively, then, up to a subsequence, $[u^\ve]^\sim$ and  $[\nabla u^\ve]^\sim$ converge two-scale to $u \, \chi $ and $[\nabla u + \nabla_y u_1]\, \chi$ respectively, where $\chi = \chi(y)$ is the characteristic function of  $Y_e$, $u \in H^1(\Omega)$ and $u_1\in L^2(\Omega; H^1_{\rm per} (Y_e)/\mathbb R)$. 
   \item Let $\{ w^\ve\} \subset H^1(\Gamma^\ve)$ satisfies 
   $$
   \ve\|w^\ve\|^2_{L^2(\Gamma^\ve)} +  \ve\|\ve \nabla_\Gamma  w^\ve\|^2_{L^2(\Gamma^\ve)}  \leq C,
   $$
  then there exists  a function $w\in L^2(\Omega; H^1(\Gamma))$ such that, up to a subsequence,  
  $w^\ve$ and $\ve \nabla_\Gamma w^\ve$ two-scale converge  to $w$ and $\nabla_{\Gamma, y} w$, respectively. 
   \end{enumerate}
\end{lemma}

To define the unfolding operator, let $[z]$ for any $z\in\mathbb R^d$ denote the 
unique  combination $\sum_{i=1}^d k_ie_i$ with $k\in\mathbb Z^d$,   
such that $z-[z]\in Y$, where $e_i$ is the $i$th
canonical basis vector of $\mathbb R^d$.

\begin{definition}[Unfolding operator \cite{Cioranescu_2012}] \label{unfold}
Let $p\in[1,\infty]$ and $\phi\in L^p(\Omega)$.  The unfolding 
operator $\T^\ve$ is defined by $\T^\ve(\phi)\in L^p(\Omega\times Y)$, where
$$
  \T^\ve(\phi)(x,y) = \begin{cases}  \phi\bigg(\ve \bigg[\dfrac{x}{\ve}\bigg] + \ve y\bigg) 
	\quad & \text{for a.e. }(x,y) \in \tilde \Omega^\ve \times Y, \\
	0 & \text{for a.e. } x\in \Omega \setminus \tilde \Omega^\ve, \;   y \in Y, 
	\end{cases} 
$$
with  $\tilde \Omega^\ve = \bigcup_{\xi \in \Xi^\ve} \ve(Y+\xi)$.  \\
For $\psi\in L^p(\Omega^\ve_l)$, with $l=e,i$,  the unfolding operator
$\T^\ve_{Y_l}$ is defined by
$$
  \T^\ve_{Y_l}(\psi)(x,y) = \begin{cases} \psi\bigg(\ve \bigg[\dfrac{x}{\ve}\bigg] + \ve y\bigg) 
	\quad & \text{for a.e. }(x,y) \in \tilde \Omega^\ve  \times Y_l, \\
	0 & \text{for a.e. } x\in \Omega \setminus \tilde \Omega^\ve,  \; y \in Y_l, 
	\end{cases} 
$$
and $\T^\ve_{Y_l}(\psi) \in L^p(\Omega\times Y_l)$. \\
For $\psi\in L^p(\Gamma^\ve)$ the boundary unfolding operator
$\T^\ve_{\Gamma}$ is defined by
$$
  \T^\ve_{\Gamma}(\psi)(x,y) = \begin{cases} \psi\bigg(\ve \bigg[\dfrac{x}{\ve}\bigg] + \ve y\bigg) 
	\quad & \text{for a.e. }(x,y) \in \tilde \Omega^\ve  \times \Gamma, \\
	0 & \text{for a.e. } x\in \Omega \setminus \tilde \Omega^\ve,  \;  y \in \Gamma, 
	\end{cases} 
$$
and $\T^\ve_{\Gamma}(\psi) \in L^p(\Omega\times \Gamma)$. 
	\end{definition}

For any function $\psi$ defined on $\Omega_l^\ve$,  for $l=e, i$, we have $\T^\ve_{Y_l}(\psi)  
= \T^\ve ([\psi]^{\sim})|_{\Omega \times Y_l}$,  with $[\psi]^{\sim}$ denoting extension of $\psi$ by zero  into $\Omega\setminus \Omega_l^\ve$,  whereas  
for $\phi$ defined on $\Omega$, it holds that
$\T^\ve_{Y_l}(\phi|_{\Omega^\ve_l}) = \T^\ve (\phi)|_{\Omega \times Y_l}$. 

The following result relates two-scale convergence and weak convergence
involving the unfolding operator.
\begin{proposition}[\cite{Cioranescu_2008}]\label{prop.unfold}
Let $\{\psi^\ve\}$ be a bounded sequence in $L^p(\Omega)$ for some $1<p<\infty$.   
Then the following assertions are equivalent:
\begin{itemize} 
\item[\rm i.] $\{\T^\ve (\psi^\ve)\}$ converges weakly to $\psi$ in
 $L^p(\Omega\times Y)$.
\item[\rm ii.] $\{\psi^\ve\}$ converges two-scale to $\psi$, \; $\psi \in L^p(\Omega\times Y)$. 
\end{itemize}
\end{proposition}

We have the following  properties of the periodic unfolding operator and the boundary unfolding operator: 
\begin{equation}\label{relat_T}
\begin{aligned} 
&  \mathcal T^\ve_{Y_l} (F(u,v)) = F( \mathcal T^\ve_{Y_l} (u), \mathcal T^\ve_{Y_l} (v)), \;  \mathcal T^\ve_{Y_l} (v(t, x/\ve)) = v(t,y), \;  x\in \Omega_l^\ve, y \in Y_l,    \\
&  \mathcal T^\ve_\Gamma (F(u,v)) = F( \mathcal T^\ve_\Gamma (u), \mathcal T^\ve_\Gamma (v)), \;  \mathcal T^\ve_\Gamma (v(t, x/\ve))= v(t,y), \;  x\in \Omega_l^\ve,  y \in \Gamma,   t>0, \\
&|Y| \la v, u \ra_{\Omega_{i,T}^\ve} = \la \mathcal T^\ve_{Y_i} (v),  \mathcal T^\ve_{Y_i} (u) \ra_{Y_i \times \Omega_T},  \;   \; \; 
 |Y| \, \ve\,  \la v, u \ra_{\Gamma^\ve_T} = \la \mathcal T^\ve_\Gamma (v),  \mathcal T^\ve_\Gamma (u) \ra_{\Gamma\times \Omega_T }, 
\\
& \la \mathcal T^\ve_{Y_e} (v),  \mathcal T^\ve_{Y_e} (u) \ra_{Y_e \times \Omega_T} =|Y|  \la v, u \ra_{\Omega_{e, T}^\ve} - |Y|\la v, u \ra_{(\Omega_{e}^\ve\setminus \tilde \Omega^\ve)_T} ,\; \; \\
& \|\mathcal T^\ve_{Y_l} (\phi)\|_{L^p(\Omega_T\times Y_l)} \leq |Y|^{\frac 1 p } \|\phi\|_{L^p(\Omega_{l, T}^\ve)},  \\
& \T^\ve_{Y_l}: L^p(0,T;  W^{1,p}(\Omega^\ve_l) )\to L^p(\Omega_T; W^{1,p}(Y_l)), \\
&  \T^\ve_{\Gamma} : L^p(0,T; W^{1,p}(\Gamma^\ve)) \to L^p(\Omega_T; W^{1,p}(\Gamma)), \\
& \ve \T^\ve_{Y_l}(\nabla u) = \nabla_y \T^\ve_{Y_l}(u),  \; \;   |Y|  \la \ve^2 \nabla v,  \nabla u \ra_{\Omega_{i,T}^\ve} = \la \nabla_y \mathcal T^\ve_{Y_i} (v),  \nabla_y \mathcal T^\ve_{Y_i} (u) \ra_{Y_i \times \Omega_T }, \\
& |Y|\,  \ve \, \la \ve^2 \nabla_\Gamma v, \nabla_\Gamma u \ra_{\Gamma^\ve_T} = \la \nabla_{\Gamma, y}  \mathcal T^\ve_\Gamma (v),  \nabla_{\Gamma, y}  \mathcal T^\ve_\Gamma (u) \ra_{\Gamma \times \Omega_T }, \\
&\|\T^\ve_\Gamma (\psi)\|_{L^p(\Omega_T\times\Gamma)} \leq \ve^{\frac 1 p } |Y|^{\frac 1 p } \|\psi\|_{L^p(\Gamma^\ve_T)} \leq C ( \|\psi\|_{L^p(\Omega^\ve_{l,T})} + \ve \|\nabla \psi\|_{L^p(\Omega^\ve_{l,T})}), 
 \end{aligned} 
 \end{equation}
for   $u, v\in L^2(0,T; H^1(\Omega^\ve_l))$, where $l=e,i$,  or  $u, v\in L^2(0,T; H^1(\Gamma^\ve))$, $\phi \in L^p(\Omega_{l,T}^\ve)$,
$\psi \in L^p(0,T; W^{1,p}(\Omega^\ve_l))$  and $F$ is any linear or nonlinear function , see e.g.\ \cite{Cioranescu_2012, Cioranescu_2008, Graf_2014}. 

We now collect some results on the the convergence of the unfolding of sequences  of functions.
\begin{lemma}[\cite{Cioranescu_2008}] Let $1\leq p < \infty$. 
\begin{enumerate}[\rm i.]
\item If $\phi \in L^p(\Omega)$, then  $\T^\ve(\phi) \to \phi$ strongly in $L^p(\Omega\times Y)$.
\item Let $\{ \psi^\ve\} \subset L^p(\Omega)$,   with $\psi^\ve \to \psi$ strongly in $L^p(\Omega)$, then 
$\T^\ve(\psi^\ve) \to \psi$ strongly in $L^p(\Omega\times Y)$. 
\end{enumerate}
\end{lemma} 

\begin{theorem} [\cite{Cioranescu_2012, Graf_2014}]\label{conv_unfold}
\begin{enumerate}[{\rm i.}]
\item Let $\{\psi^\ve\}$ be a bounded sequence in $W^{1, p}(\Omega^\ve_e)$, for some
$1< p < \infty$. Then  there exist functions $\psi \in W^{1, p}(\Omega)$ and 
$\psi_1 \in L^p(\Omega; W^{1, p}_{\rm per}(Y_e)/\mathbb R)$ such that as $\ve\to 0$, 
up to a subsequence,  
\begin{align*} 
  & \mathcal T^\ve_{Y_e}(\psi^\ve) \rightharpoonup\psi 
	&& \text{weakly in } L^p(\Omega; W^{1,p}(Y_e)), \\
  & \mathcal T^\ve_{Y_e}(\psi^\ve) \to  \psi 
	&& \text{strongly  in } L^p_{\rm loc}(\Omega; W^{1,p}(Y_e)), \\
  & \mathcal T^\ve_{Y_e}(\nabla \psi^\ve) \rightharpoonup \nabla \psi  
	+ \nabla_y \psi_1 && \text{weakly in } L^p(\Omega\times Y_e). 
\end{align*} 
\item Let $\{\phi^\ve\}\subset W^{1, p}(\Omega^\ve_i)$,  for some
$1< p < \infty$,  satisfies
$$
\|\phi^\ve\|_{L^p(\Omega_i^\ve)} + \ve \|\nabla \phi^\ve\|_{L^p(\Omega_i^\ve)} \leq C.
$$ 
Then  there exists   $\phi \in L^p(\Omega; W^{1, p}(Y_i))$ such that as $\ve\to 0$, 
up to a subsequence,  
\begin{align*} 
  & \mathcal T^\ve_{Y_i}(\phi^\ve) \rightharpoonup\phi 
	&& \text{weakly in } L^p(\Omega\times Y_i), \\
  &\ve \mathcal T^\ve_{Y_i}(\nabla \phi^\ve) \rightharpoonup 
	 \nabla_y \phi && \text{weakly in } L^p(\Omega\times Y_i). 
\end{align*} 

\item Let $\{ w^\ve\} \subset H^1(\Gamma^\ve)$ satisfies 
$$
\ve \|w^\ve\|^2_{L^2(\Gamma^\ve)} + \ve \|\ve \nabla_\Gamma w^\ve\|^2_{L^2(\Gamma^\ve)}\leq C, 
$$
then there exists $w\in L^2(\Omega; H^1(\Gamma))$ such that as $\ve \to 0$, up to a subsequence,  
\begin{align*} 
  & \mathcal T^\ve_{\Gamma}(w^\ve) \rightharpoonup w 
	&& \text{weakly in } L^2(\Omega; H^1( \Gamma)), \\
  &\ve \mathcal T^\ve_{\Gamma}(\nabla_\Gamma w^\ve) \rightharpoonup 
	 \nabla_{\Gamma, y} w && \text{weakly in } L^2(\Omega\times \Gamma). 
\end{align*} 

\end{enumerate}
\end{theorem}

\section{Some details on the proof of Lemma~\ref{lem:boundedness}}\label{boundedness_sketch}
In the second equation in \eqref{weak_sol_1}  and  in equations \eqref{weak_sol_2}, integrating by parts with respect to the time variable in the terms involving time derivatives,   applying the periodic unfolding operator and the boundary unfolding operator,   and using the properties of  the unfolding operator, see e.g.\   \eqref{relat_T},  yields for  $x \in \Omega$, $\tau \in (0, T]$
\begin{equation}\label{eq:unfolding_c_i}
\begin{aligned} 
&-  \la \T^\ve(c^\ve_i),  \partial_t \T^\ve(\psi) \ra_{Y_{i,\tau}}  +  \la D_i(y) \nabla_y \T^\ve(c^\ve_i), \nabla_y\T^\ve(\psi) \ra_{Y_{i, \tau}} \\
& \qquad +  \la \T^\ve(c^\ve_{i}(\tau)), \T^\ve(\psi(\tau)) \ra_{Y_{i}}    = \la \T^\ve(c^\ve_{i,0}), \T^\ve(\psi(0)) \ra_{Y_{i}}  
\\ & \qquad  +  \la  F_i(\T^\ve(c^\ve_i)), \T^\ve(\psi) \ra_{Y_{i, \tau}}   + \ \la G_i (\T^\ve(c^\ve_i),  \T^\ve(p_a^\ve)), \T^\ve(\psi) \ra_{\Gamma_\tau},
\end{aligned} 
\end{equation}  
\begin{equation}\label{eq:unfold_boundary}
\begin{aligned} 
  & - \la \T^\ve(r^\ve_f) , \partial_t \T^\ve(\varphi) \ra_{\Gamma_\tau}   +   \la D_{f}  \nabla_{\Gamma, y} \T^\ve( r^\ve_f),  \nabla_{\Gamma, y} \T^\ve(\varphi)  \ra_{\Gamma_\tau} \\
   & +    \la \T^\ve(r^\ve_{f}(\tau)) , \T^\ve(\varphi(\tau)) \ra_{\Gamma}  
  =        \la \T^\ve(r^\ve_{f, 0}) , \T^\ve(\varphi(0)) \ra_{\Gamma}  - d_f  \la \T^\ve(r_f^\ve) ,  \T^\ve (\varphi) \ra_{\Gamma_\tau} \\
 &\qquad    + \la  F_f(\T^\ve(r^\ve_f), \T^\ve(r^\ve_b))-  G_e (\T^\ve(c^\ve_e), \T^\ve(r_f^\ve), \T^\ve(r_b^\ve)),  \T^\ve(\varphi) \ra_{\Gamma_\tau}, 
   \\
&   -\la   \T^\ve(r^\ve_b), \partial_t  \T^\ve(\varphi) \ra_{\Gamma_\tau}  +  \la  D_{b} \nabla_{\Gamma, y}  \T^\ve( r^\ve_b) ,  \nabla_{\Gamma, y} \T^\ve(\varphi)\ra_{\Gamma_\tau} 
  \\ 
  & + \la   \T(r^\ve_{b}(\tau)),   \T^\ve(\varphi(\tau)) \ra_{\Gamma}
  =      \la   \T(r^\ve_{b,0}),   \T^\ve(\varphi(0)) \ra_{\Gamma}
   - d_b  \la \T^\ve(r_b^\ve), \T^\ve(\varphi) \ra_{\Gamma_\tau} \\
 & \quad  +  \la  G_e (\T^\ve(c^\ve_e), \T^\ve(r_f^\ve), \T^\ve(r_b^\ve)) - G_d(\T^\ve(r_b^\ve), \T^\ve(p_d^\ve), \T^\ve(p_a^\ve)) ,  \T^\ve(\varphi) \ra_{\Gamma_\tau}, 
   \end{aligned} 
\end{equation}
and 
\begin{equation}\label{eq:unfold_boundary_2}
\begin{aligned} 
& - \la  \T^\ve(p^\ve_d ), \partial_t \T^\ve(\varphi) \ra_{\Gamma_\tau} +  \la    D_d \nabla_{\Gamma, y}  \T^\ve(p^\ve_d),   \nabla_{\Gamma, y} \T^\ve(\varphi) \ra_{\Gamma_\tau} 
\\
&  +  \la \T^\ve(p^\ve_{d}(\tau)) , \T^\ve(\varphi(\tau)) \ra_{\Gamma}   =  \la \T^\ve(p^\ve_{d, 0}) ,  \T^\ve(\varphi)(0) \ra_{\Gamma}
 - d_d \la \T^\ve(p_d^\ve), \T^\ve(\varphi) \ra_{\Gamma_\tau}\\
& \qquad + \la  F_d(\T^\ve(p^\ve_d)) - G_d(\T^\ve(r_b^\ve), \T^\ve(p_d^\ve), \T^\ve(p_a^\ve)), \T^\ve(\varphi) \ra_{\Gamma_\tau}, 
 \\
 & - \la \T^\ve(p^\ve_a), \partial_t  \T^\ve(\varphi)  \ra_{\Gamma_\tau} +  \la D_a \nabla_{\Gamma, y} \T^\ve( p^\ve_a) ,  \nabla_{\Gamma, y} \T^\ve(\varphi) \ra_{\Gamma_\tau} \\
 &  + \la \T^\ve(p^\ve_{a}(\tau)),   \T^\ve(\varphi (\tau))  \ra_{\Gamma} 
 = \la \T^\ve(p^\ve_{a, 0}),  \T^\ve(\varphi (0))  \ra_{\Gamma} - d_a \la \T^\ve(p_a^\ve), \T^\ve(\varphi) \ra_{\Gamma_\tau}\\
 &\qquad +  \la  G_d(\T^\ve(r_b^\ve), \T^\ve(p_d^\ve), \T^\ve(p_a^\ve))  
- G_i(\T^\ve(p^\ve_a), \T^\ve(c_i^\ve)), \T^\ve(\varphi) \ra_{\Gamma_\tau}, 
 \end{aligned} 
\end{equation}
where  $\psi \in L^2(0,T; H^1(\Omega_i^\ve))$ with $\partial_t \psi \in L^2(\Omega_{i,T}^\ve)$ and $\varphi \in L^2(0, T; H^1(\Gamma^\ve))$ with $\partial_t \varphi \in L^2(\Gamma^\ve_T)$.  Notice that the regularity of solutions of the microscopic problem implies  $c_i ^\ve\in C([0,T];L^2(\Omega^\ve_i))$ and $r_j^\ve, p_s^\ve \in C([0,T]; L^2(\Gamma^\ve))$,  for $j=f,b$ and $s=a,d$.

Considering the sum of equations \eqref{eq:unfolding_c_i}-\eqref{eq:unfold_boundary_2}  with  test functions  $\psi(t,x) = 1$ in $\Omega_{i,T}^\ve$ and $\varphi(t,x) = 1$ on $\Gamma^\ve_T$, respectively,  and  using the nonnegativity of solutions,  the structure of the reaction terms,  and  the assumptions on the initial data yields  
\begin{equation}\label{estim_L1_Gamma}
\begin{aligned}
\| \T^\ve(r^\ve_f)(\tau,x,\cdot)\|_{L^1(\Gamma)} +\| \T^\ve(r^\ve_b) (\tau,x,\cdot)\|_{L^1(\Gamma)}+
\| \T^\ve(p^\ve_d) (\tau,x,\cdot) \|_{L^1(\Gamma)} \qquad \quad \\
+ 2\|\T^\ve(p^\ve_a) (\tau,x,\cdot) \|_{L^1(\Gamma)}  +2 \|\T^\ve(c_i^\ve)(\tau, x, \cdot)\|_{L^1(Y_i)} \leq  
   C_1\| \T^\ve(r^\ve_{f,0}) (x,\cdot)\|_{L^1(\Gamma)} \\ + \|\T^\ve(r^\ve_{b, 0}) (x,\cdot) \|_{L^1(\Gamma)}
    +
   \| \T^\ve(p^\ve_{d, 0}) (x,\cdot) \|_{L^1(\Gamma)} 
+2 \|\T^\ve(p^\ve_{a, 0}) (x,\cdot)\|_{L^1(\Gamma)}   \\
+ C_2 \|\T^\ve(c^\ve_{i,0})(x,\cdot) \|_{L^1(Y_i)} + C_3  
   \leq C, 
 \end{aligned} 
\end{equation}
for $\tau \in (0, T]$ and a.a.\ $x\in \Omega$. 

The estimates in \eqref{estim_pdaci_L2}  are obtained by considering  $\T^\ve(c_i^\ve)$ as  a test function in \eqref{eq:unfolding_c_i},    $\T^\ve(r_f^\ve)$ as a test function in the  equation for $\T^\ve(r_f^\ve)$ and 
$\T^\ve(r_b^\ve)  + \T^\ve(r_f^\ve)$ as a  test function  in the sum of equations for $\T^\ve(r_b^\ve)$ and $\T^\ve(r_f^\ve)$ in \eqref{eq:unfold_boundary}, 
$\T^\ve (p_d^\ve)$  as a test function in the  equation for $\T^\ve(p_d^\ve)$ and  $\T^\ve(p_d^\ve)  + \T^\ve(p_a^\ve)$ as a test function  in the sum of equations for $\T^\ve(p_d^\ve)$ and $\T^\ve(p_a^\ve)$ in \eqref{eq:unfold_boundary_2}, and by using the nonnegativity of solutions of the microscopic problem. To ensure that the time derivative  is well-defined, we consider a standard  approximation, using the Steklov average,  of $r_l^\ve$,  $p_s^\ve$, $c_i^\ve$, for $l=f,b$ and $s=d,a$,  i.e.\
$$
v^\zeta(t,x) = \frac 1 {\zeta} \int_{t-\zeta}^t  \frac 1 {\zeta} \int_s^{s+\zeta}  v(\sigma,x) d\sigma \kappa(s)  ds, 
$$
with $\kappa(s) =1$ for $s\in (0, T-\zeta)$ and $\kappa(s) = 0$ for $s \in [-\zeta, 0] \cup [T-\zeta, T]$,   and then take $\zeta \to 0$, see e.g.\ \cite{Ladyzhenskaya} for more details. 

To show boundedness of solutions of the microscopic problem we first consider   $|\T^\ve(r_f^\ve)|^{p-1}$ for $p\geq 4$ as a test function in the first equation in \eqref{eq:unfold_boundary} and using the nonnegativity of $c^\ve_e$ and $r_f^\ve$ and assumptions on the nonlinear function $F_f$  we obtain 
 \begin{equation}\label{estim_rf_Lp}
\begin{aligned} 
 \|\T^\ve (r_f^\ve) (\tau) \|^p_{L^p(\Gamma)} + 4 \frac {p-1} p \| \nabla_{\Gamma, y}| \T^\ve (r_f^\ve)|^{\frac p 2}\|^2_{L^2(\Gamma_\tau)} \qquad \qquad \\ \leq   C_1p\left[1+    \|\T^\ve (r_f^\ve)\|^p_{L^p(\Gamma_\tau)}\right] + C_2  p \la \T^\ve (r_b^\ve), | \T^\ve (r_f^\ve)|^{p-1} \ra_{\Gamma_\tau} .
 \end{aligned} 
\end{equation} 
  Applying the H\"older inequality and inequalities in \eqref{estim_GN_1}, the last term in \eqref{estim_rf_Lp} is estimated in the following way 
 \begin{equation}\label{estim_rf_p_1}
\begin{aligned} 
& \la \T^\ve (r_b^\ve), | \T^\ve (r_f^\ve)|^{p-1} \ra_{\Gamma_\tau}  \leq  \frac 1 p \|\T^\ve (r_b^\ve)\|^p_{L^4(\Gamma_\tau)}\\
& \quad +C_1\frac{p-1} p \|\nabla_{\Gamma, y}  | \T^\ve (r_f^\ve)|^{\frac p 2}\|_{L^2(\Gamma_\tau)} 
  \||\T^\ve (r_f^\ve)|^{\frac p 2}\|^4_{L^4(0, \tau; L^2(\Gamma) )}
 \\
& \leq 
  \frac 1 p \|\T^\ve (r_b^\ve)\|^p_{L^4(\Gamma_\tau)}  + 
 C_2\frac{p-1} p  \|\nabla_{\Gamma, y}  | \T^\ve (r_f^\ve)|^{\frac p 2}\|_{L^2(\Gamma_\tau)} \times \\
& \quad \;\;   \times \left[\int_0^\tau \|  \nabla_{\Gamma, y} |\T^\ve (r_f^\ve)|^{\frac p 2}\|^2_{L^2(\Gamma)}  \| \T^\ve (r_f^\ve)^{\frac p 2}\|_{L^1(\Gamma)}^2   \, dt \right]^{\frac 1 4}  
 \leq  \frac 1 p \|\T^\ve (r_b^\ve)\|^p_{L^4(\Gamma_\tau)}\\
& +  C_2 \frac{p-1} p \|\nabla_{\Gamma, y}   | \T^\ve (r_f^\ve)|^{\frac p 2}\|_{L^2(\Gamma_\tau)}^{\frac 3 2} 
\sup\limits_{(0,\tau)}\| |\T^\ve (r_f^\ve)|^{\frac p 2}\|_{L^1(\Gamma)}^{\frac 1 2}   \leq \frac 1 p \|\T^\ve (r_b^\ve)\|^p_{L^4(\Gamma_\tau)}
 \\
& + \delta  \frac {p-1} {p^2} \|\nabla_{\Gamma, y}   | \T^\ve (r_f^\ve)|^{\frac p 2}\|_{L^2(\Gamma_\tau)}^2 + C_\delta p^3 \sup\limits_{(0,\tau)}\| \T^\ve (r_f^\ve)^{\frac p 2}\|_{L^1(\Gamma)}^2 , 
   \end{aligned}
 \end{equation}
for $\tau \in (0, T]$ and $x\in \Omega$.  Using the Gagliardo-Nirenberg inequality, see \eqref{estim_GN_1}, we also obtain 
 \begin{equation} \label{estim:bound_Lp_rf}
  \|\T^\ve (r_f^\ve) \|^p_{L^p(\Gamma_\tau)} \leq \delta \frac{p-1}{p^2} \|\nabla_{\Gamma, y} |\T^\ve (r_f^\ve)|^{\frac p 2} \|^2_{L^2(\Gamma_\tau)} 
 + C_\delta \,  p  \sup\limits_{(0,\tau)} \| |\T^\ve (r_f^\ve)|^{\frac p2}\|_{L^1(\Gamma)}^{2},  
 \end{equation} 
for $\tau \in (0, T]$ and $x\in \Omega$. 
Then using estimates \eqref{estim_rf_p_1} and \eqref{estim:bound_Lp_rf} in \eqref{estim_rf_Lp} yields inequality \eqref{rf_bound_1}.

To show boundedness of $c^\ve_e$ 
we consider  $|c^\ve_e|^{p-1}$, for $p\geq 4$,  as a test function in the first equation in \eqref{weak_sol_1}  and, using the assumptions on $F_e$ and the nonnegativity of $r_f^\ve$ and $c^\ve_e$ we obtain
 \begin{equation}\label{estim_c_p_2}
\begin{aligned} 
 \| c_e^\ve (\tau) \|^p_{L^p(\Omega_e^\ve)} + 4 \frac {p-1} p \| \nabla_{\Gamma, y}| c_d^\ve|^{\frac p 2}\|^2_{L^2(\Omega^\ve_{e,\tau})} \leq   C_1p\left[1+    \|c_e^\ve\|^p_{L^p(\Omega^\ve_{e,\tau})}\right] \\ + C_2  p\,  \ve \,  \la r_b^\ve, | c_e^\ve|^{p-1} \ra_{\Gamma^\ve_\tau} .
 \end{aligned} 
\end{equation} 
The last term  in  \eqref{estim_c_p_2} can be estimated in the following way 
\begin{equation}\label{estim_c_rb}
\begin{aligned} 
& \ve \, |Y|\, \la r_b^\ve, | c_e^\ve|^{p-1} \ra_{\Gamma^\ve_\tau} =  \la \T^\ve (r_b^\ve), |\T^\ve (c^\ve_e)|^{p-1}\ra_{\Omega_\tau \times \Gamma} \\
&\hspace{ 2 cm }    \leq \int_{\Omega_\tau}  \left[\int_\Gamma  |\T^\ve (c^\ve_e)|^{\frac{4(p-1)}3} d\sigma_y \right]^{\frac 3 4}   \|\T^\ve(r^\ve_b)\|_{L^4(\Gamma)} dx dt 
\\
& \hspace{ 0.6 cm }  \leq    \frac {p-1} p \int_{\Omega_\tau } \left[\int_\Gamma \big( |\T^\ve(c^\ve_e)|^{\frac p 2} \big)^{\frac 83} d\sigma_y \right]^{\frac 3 4} dx dt + 
\frac 1 p \int_{\Omega_\tau }   \|\T^\ve(r^\ve_b)\|^p_{L^4(\Gamma)} dx dt.  
\end{aligned} 
\end{equation}
Applying   the trace inequality \eqref{trace_gen} 
 to $|\mathcal T^\ve(c_e^\ve)|^{\frac p 2}$ and  using   the properties  of the unfolding operator $\T^\ve$, see e.g.\ \eqref{relat_T},     the first term on the right-hand side of  \eqref{estim_c_rb} is estimated as 
$$
\begin{aligned} 
& \||\T^\ve(c_e^\ve)|^{\frac p 2}\|^2_{L^2(\Omega_\tau; L^{8/3}(\Gamma))} 
 \leq \frac {\delta} {p}  \|\nabla_y |\T^\ve(c_e^\ve)|^{\frac p 2}\|^2_{L^2(\Omega_\tau; L^2(Y_e))}\\
&\quad + C_\delta \, p^3 \, \||\T^\ve(c_e^\ve)|^{\frac p 2} \|^2_{L^2(\Omega_\tau; L^2(Y_e))}  \leq  |Y| \Big[ C_\delta\,  p^3\,   \|c_e^\ve\|^p_{L^p(\Omega^\ve_{e,\tau})}+ \frac \delta p   \ve^2\,  \|\nabla |c_e^\ve|^{\frac  p 2} \|^2_{L^2(\Omega^\ve_{e,\tau})}\Big] . 
\end{aligned}
$$
To   estimate  $ \|\T^\ve(r^\ve_b)\|_{L^p(\Omega_T;L^{4}(\Gamma))}$ we 
consider   $|\T^\ve (r_b^\ve)|^3$  as a test functions in the second equation in  \eqref{eq:unfold_boundary} and obtain 
\begin{equation}\label{estim_L4_rb}
\begin{aligned} 
\| \T^\ve(r_b^\ve)(\tau) \|^4_{L^4(\Gamma)} + \|\nabla_{\Gamma, y} |\T^\ve(r_b^\ve)|^2 \|^2_{L^2(\Gamma_\tau)}
  \leq  \,  \| \T^\ve(r_{b0}^\ve) \|^4_{L^4(\Gamma_\tau)}  \qquad \qquad 
   \\   + C_1 \la  \T^\ve(p_a^\ve), |\T^\ve(r_b^\ve)|^3 \ra_{\Gamma_\tau}  
   + C_2 \la \T^\ve(r_f^\ve) \T^\ve(c_e^\ve), |\T^\ve(r_b^\ve)|^3 \ra_{\Gamma_\tau} , 
\end{aligned} 
\end{equation}
for $x\in \Omega$ and $\tau \in (0, T]$. Applying  the Gagliardo-Nirenberg inequality 
$$
\| v\|_{L^3(\Gamma)} \leq  C_1\|\nabla_{\Gamma, y} v \|_{L^2(\Gamma)}^{2/3} \|v\|^{1/3}_{L^1(\Gamma)} 
$$
to $|\T^\ve(r_b^\ve)|^2$  and using the estimates in \eqref{estim_rfb} yield
$$
\begin{aligned} 
& | \la \T^\ve(p_a^\ve), |\T^\ve(r_b^\ve)|^3 \ra_{\Gamma_\tau}|
\leq \delta  \|\nabla_{\Gamma, y} |\T^\ve(r_b^\ve) |^2 \|^2_{L^2(\Gamma_\tau)} \sup\limits_{(0,\tau)}\|\T^\ve (r_b^\ve)  \|_{L^2(\Gamma)} 
\\ 
&+ 
\| \T^\ve (r_b^\ve)  \|^{6}_{L^6(0, \tau, L^2(\Gamma))}
+ C_\delta  \| \T^\ve(p_a^\ve)\|^2_{L^2(\Gamma_\tau)}  \leq C_1\delta \|\nabla_{\Gamma, y} |\T^\ve(r_b^\ve) |^2 \|^2_{L^2(\Gamma_\tau)}  + C_2, 
\end{aligned} 
$$
for $\tau \in (0, T]$ and $x\in \Omega$. The boundedness of $\T^\ve(r_f^\ve)$, see \eqref{rf_bound},  ensures 
$$
\begin{aligned}
\left |\la  \T^\ve(r_f^\ve) \T^\ve(c_e^\ve), |\T^\ve(r_b^\ve)|^3 \ra_{\Gamma_\tau} \right | & \, \leq 
C_1 \|\T^\ve(c_e^\ve)\|_{L^p(\Gamma_\tau)} \|\T^\ve(r_b^\ve)\|^3_{L^{\frac{3p}{p-1}}(\Gamma_\tau)}  \\
 & \, \leq C_2 \left[ \|\T^\ve(c_e^\ve)\|^4_{L^p(\Gamma_\tau)} + \|\T^\ve(r_b^\ve)\|^4_{L^4(\Gamma_\tau)}\right], 
\end{aligned}
$$
for  $p \geq 4$. 
Then  combining the estimates above and using Gronwall's inequality   in  \eqref{estim_L4_rb}   implies 
\begin{equation}\label{estim_rbL4}
\begin{aligned} 
\| \T^\ve(r_b^\ve)(\tau) \|_{L^4(\Gamma)} + \|\nabla_{\Gamma, y} |\T^\ve(r_b^\ve)|^2 \|^{\frac 12}_{L^2(\Gamma_\tau)}
 \leq  C_1[1+ \|\T^\ve(c_e^\ve)\|_{L^p(\Gamma_\tau)}],  
\end{aligned} 
\end{equation}
for $\tau \in (0, T]$,  $x\in \Omega$, and $p \geq 4$.
Hence using \eqref{estim_rbL4} in \eqref{estim_c_rb} and applying the relations between the original and unfolded sequences, see \eqref{relat_T}, estimate \eqref{estim_c_p_2}  yields
\begin{equation}\label{estim-c_e-p} 
\begin{aligned} 
\||c^\ve_e(\tau) |^{\frac p2}\|^2_{L^2(\Omega_e^\ve)} +4 \frac{p-1}{p} \|\nabla |c^\ve_e|^{\frac p 2}\|^2_{L^2(\Omega_{e, \tau}^\ve)} 
\leq  \delta_1 \frac {p-1}{ p}   \|\nabla |c^\ve_e|^{\frac p 2} \|^{2}_{L^2(\Omega^\ve_{e,\tau})} \\
+ C_\delta p^3(1+   \| c^\ve_e\|_{L^p(\Omega_{e, \tau}^\ve)}^{p}) +
  C_1  (1+ \ve \| c^\ve_e\|^p_{L^p(\Gamma^\ve_\tau)}). 
  \end{aligned} 
\end{equation} 
 Notice that the Gagliardo-Nirenberg inequality and the properties of an extension  $|\bar c^\ve_e|^{\frac p2}$ of $|c^\ve_e|^{\frac p 2}$ from $\Omega_e^\ve$ into $\Omega$, see \eqref{estim:extension}, ensures 
  \begin{equation}\label{estim:c_e_Lp_L1}
  \begin{aligned} 
   \| c^\ve_e\|_{L^p(\Omega_{e, \tau}^\ve)}^{p}  & \leq  \| |\bar c^\ve_e|^{\frac p 2}\|_{L^2(\Omega_{\tau})}^{2}  
  \leq \mu_1 \int_0^\tau \| |\bar c^\ve_e|^{\frac p 2}\|^{4/5}_{L^1(\Omega)} \|\nabla  |\bar c^\ve_e|^{\frac p 2}\|^{6/5}_{L^2(\Omega)} dt \\
&   \leq \frac {\delta_1}{p^3}  \|\nabla  |\bar c^\ve_e|^{\frac p 2}\|^2_{L^2(\Omega_\tau)} 
    + C_{\delta,1}p^{\frac 92}  \sup_{(0,\tau)}\| |\bar c^\ve_e|^{\frac p 2}\|^2_{L^1(\Omega)}  \\
  &\leq \frac \delta {p^3} \|\nabla  |c^\ve_e|^{\frac p 2}\|^2_{L^2(\Omega_{e,\tau}^\ve)} + C_\delta p^{\frac 92}\sup\limits_{(0,\tau)}\| |c^\ve_e|^{\frac p 2}\|^2_{L^1(\Omega_{e}^\ve)}. 
  \end{aligned} 
  \end{equation}
 Then applying  trace inequality \eqref{ineq:trace}  in the last term in \eqref{estim-c_e-p} and using  the estimate \eqref{estim:c_e_Lp_L1}  yield \eqref{estim_cp_2}.

}

\end{document}